\documentclass[10pt,a4paper]{amsart}
\usepackage[latin1]{inputenc}
\usepackage[english]{babel}
\usepackage{amsmath}
\usepackage{amsfonts}
\usepackage{amssymb}
\usepackage{amsthm}
\usepackage[left=2.5cm,right=2.5cm,top=2.5cm,bottom=2.5cm]{geometry}
\usepackage{epsfig}
\usepackage{url}

\newtheorem{theorem}{Theorem}[section]
\newtheorem{definition}[theorem]{Definition}
\newtheorem{remark}[theorem]{Remark}
\newtheorem{example}[theorem]{Example}
\newtheorem{proposition}[theorem]{Proposition}
\newtheorem{assumption}[theorem]{Assumption}
\newtheorem{corollary}[theorem]{Corollary}
\newtheorem{lemma}[theorem]{Lemma}

\newcommand{\C}{\mathbb{C}}
\newcommand{\R}{\mathbb{R}}
\newcommand{\K}{\mathbb{K}}
\renewcommand{\H}{\mathcal{H}}
\newcommand{\N}{\mathbb{N}}
\newcommand{\z}{\zeta}
\renewcommand{\sp}[2]{\left( {#1} \mid {#2} \right)}

\newcommand{\dom}{D}
\newcommand{\norm}[1]{\left\Vert{#1}\right\Vert}
\newcommand{\abs}[1]{\left|{#1}\right|}

\newcommand{\B}{\mathcal{B}}
\newcommand{\ran}{\operatorname{ran}\,}

\renewcommand{\Re}{\operatorname{Re}\,}
\renewcommand{\Im}{\operatorname{Im}\,}
\newcommand{\He}{\operatorname{Sym}\,}
\newcommand{\Sym}{\He}
\newcommand{\Lip}{\operatorname{Lip}}
\newcommand{\dd}{\, \mathrm{d}}
\newcommand{\ii}{\mathrm{i}}
\newcommand{\ee}{\mathrm{e}}

\author{Bj\"orn Augner}
\address{Technische Universit\"at Darmstadt, Fachbereich Mathematik, Mathematische Modellierung und Analysis, Schlossgartenstr.\ 7, 64289 Darmstadt.}
\email{augner@mma.tu-darmstadt.de}
\thanks{This work has been partly supported by Deutsche Forschungsgemeinschaft (Grant JA 735/8-1).}
\title[Interconnection Structures of Port-Hamiltonian Type]{Well-posedness and Stability for Interconnection Structures of Port-Hamiltonian Type}

\begin{document}
\allowdisplaybreaks[1]
 \maketitle
 
 Version of \today.
 
 \begin{abstract}
  We consider networks of infinite-dimensional port-Hamiltonian systems $\mathfrak{S}_i$ on one-dimensional spatial domains.
  These subsystems of port-Hamiltonian type are interconnected via boundary control and observation and are allowed to be of distinct port-Hamiltonian orders $N_i \in \N$.
  Well-posedness and stability results for port-Hamiltonian systems of fixed order $N \in \N$ are thereby generalised to networks of such.
  The abstract theory is applied to some particular model examples.
 \end{abstract}
 
 
 \textit{Secondary:} 35G46, 37L15, 47B44, 47D06.
 
\section{Introduction}

A \emph{port-based} modelling and analysis initially had been introduced in the 1960's to treat complex, multiphysics systems within a unified mathematical framework \cite{VanDerSchaftJeltsema_2014}.
Each of these subsystems, may it be of mechanical, electrical or thermal type etc.\ is described by its inner dynamics, usually by a system of ODEs or PDEs, on the one hand, and \emph{ports}, which enable the interconnection with other subsystems, on the other hand.
For \emph{port-Hamiltonian systems} the notion of an \emph{energy} has been highlighted, similar to classical Hamiltonian systems.
In contrast to the latter, however, the port-Hamiltonian formulation allows besides \emph{conservative}, i.e.\ energy preserving, elements also for \emph{dissipative}, i.e.\ energy dissipating, elements, e.g.\ frictional losses in mechanical systems or energy conversion in resistors within a electric circuit, where the energy leaves the system in form of heat while the latter is not included in the model.

For the description and analysis of port-Hamiltonian systems in a geometrical way, in \cite{VanDerSchaftMaschke_1995} the concept of a \emph{Dirac structure} had been introduced into the theory of port-Hamiltonian systems.
These Dirac structures have the very convenient property that (suitable) interconnections of Dirac structures again give a Dirac structure (of higher dimension).
The underlying models for the physical systems up to the 2000's had been primarily finite-dimensional, i.e.\ the inner dynamics of the subsystems interconnected via ports had usually been described by ODEs.
Probably with the article \cite{VanDerSchaftMaschke_2002} first attempts were made to extend the developed finite-dimensional theory of port-Hamiltonian systems to infinite-dimensional models, i.e.\ PDEs, and thereby filling in the gap between results on finite-dimensional systems and infinite-dimensional port-Hamiltonian systems.
E.g.\ first in \cite{LeGorrecZwartMaschke_2005}, it has been demonstrated that for linear infinite-dimensional port-Hamiltonian systems on an interval, i.e.\ evolution equations of the form
  \[
   \frac{\partial x}{\partial t}
    = \sum_{k=0}^N P_k \frac{\partial^k (\H x)}{\partial \z^k} (t,\z),
    \quad
    t \geq 0, \, \z \in (0,l)
  \]
with $x(t,\cdot) \in L_2(0,l;\K^d)$ (where $K = \R$ or $\K = \C$) and for suitable $P_k \in \K^{d \times d}$ and $\H: [0,1] \rightarrow \K^{d \times d}$, those boundary conditions (or, in a rather systems theoretic interpretation: linear closure relations) that lead to generation of a bounded (even contractive, when $L_2(0,l;\K^d)$ is equipped with an appropriate energy norm) $C_0$-semigroup can be characterised: Crucial is the dissipativity (w.r.t.\ the energy inner product), which can be checked solely via a matrix criterion on the boundary conditions \cite{LeGorrecZwartMaschke_2004}, \cite{LeGorrecZwartMaschke_2005}.
Next steps then have been sufficient conditions for asymptotic or uniform exponential stability of the system \cite{Villegas_2007}, \cite{VillegasEtAl_2009}, \cite{AugnerJacob_2014}.
Then followed efforts to generalise these results to PDE-ODE-systems, i.e.\ feedback control via a finite-dimensional linear control system \cite{RamirezZwartLeGorrec_2013}, \cite{AugnerJacob_2014}, and non-linear boundary feedback \cite{Trostorff_2014}, \cite{RamirezZwartLeGorrec_2017}, \cite{Augner_2018+}.
Here, we want to push forward into a different direction and in a sort return to the beginnings of port-Hamiltonian modelling:
What happens, if we consider a network of infinite-dimensional port-Hamiltonian subsystems instead of a single one, where the subsystems, just in the spirit of port-based modelling, are coupled via boundary control and observation of the distinct port-Hamiltonian subsystems?
To what extend do the results on well-posedness (in the sense of semigroup theory) and stability extend to this network case?
For special classes of PDE, especially the wave equation and several beam models, such an analysis is not new by any means, see e.g.\ \cite{VonBelow_1988}, \cite{LagneseLeugeringSchmidt_1994}, \cite{Leugering_1996}, \cite{DekoninckNicaise_2000} and \cite{MercierRegnier_2008}.

Before giving an outline of the organisation of this paper, let us emphasise that for systems with constant \emph{Hamiltonian energy densities} $\H_i: [0,l_i] \rightarrow \K^{d_i \times d_i}$, already alternative approaches to well-posedness and stability are well-known.
In particular, in that case it is often possible to determine (in an analytical way or via sufficiently good numerical approximation) the eigenvalues of the total system up to sufficient accuracy, and derive conclusions on well-posedness and stability.
For non-constant $\H_i$ such an approach is not that easily accessible, in particular there are situations, in which stability properties of a port-Hamiltonian system are very sensitive to multiplicative perturbation by $\H_i$, see e.g.\ \cite{Engel_2013} for an astonishing counter example.
Therefore, we deem the port-Hamiltonian approach as a legitimate way to describe and analyse such systems.

This manuscript is structured as follows.
Section \ref{sec:notation} serves as an introduction to the (mainly standard) notation we use throughout this paper, we recall some basic facts on (strongly continuous) semigroup theory, and the notion of a port-Hamiltonian system is introduced. In Section \ref{sec:basic_definition}, we recall previous results on the well-posedness and stability of infinite-dimensional linear port-Hamiltonian systems on a one-dimensional domain.
We do this with the background of particular interconnection schemes which have been considered up to now, and also comment on some of the techniques used to prove the corresponding results.
The subsequent Sections \ref{sec:PHS_networks}, \ref{sec:hybrid_multi_phs} and \ref{sec:networks_hybrid_systems} constitute the main sections of this paper:
First, in Section \ref{sec:PHS_networks} we provide the general well-posedness result for multi-port Hamiltonian systems interconnected in a dissipative way: As for single port-Hamiltonian systems, a dissipative linear closure relation is already enough to have existence of unique (strong) solutions for all initial data, and the solution depends continuously on the initial datum, i.e.\ the initial datum to solution map is given by a strongly continuous semigroup of linear operators.
Secondly, the focus lies on asymptotic and exponential stability for closed loop port-Hamiltonian systems, which we investigate in Section \ref{sec:hybrid_multi_phs} under additional structural constraints, e.g.\ the port-Hamiltonian systems being serially interconnected in a chain.
Then, Section \ref{sec:networks_hybrid_systems} is devoted to systems consisting themselves of systems of port-Hamiltonian systems again which for complex structures of the total system might be a helpful point of view for stability considerations.
We illustrate the results of the preceding sections by some applications networks of first order port-Hamiltonian systems and of Euler-Bernoulli beam type.
Finally, in Section \ref{sec:conclusion} we rephrase the main aspects of this paper and comment on further open or related problems.
After that, some technical results on the Euler-Bernoulli beam equation are collected in an appendix.

\section{Preliminaries}
\label{sec:notation}
\subsection{Notation}
 Let us fix some notation.
 Throughout, the field $\K = \R$ or $\C$ denotes real or complex numbers and all Banach or Hilbert spaces appearing are $\K$-Banach spaces or $\K$-Hilbert spaces, respectively.
 Without further notice, we assume that w.l.o.g.\ $\K = \C$ whenever we consider eigenvalues of operators.
 Note that this is no restriction since in case $\K = \R$ we may always consider the complexification of the involved operators and, e.g.\ for a generator $A$ of a $C_0$-semigroup $(T(t))_{t \geq 0}$ on a real Banach space $X$, the complexification $A^\C$ of the operator $A$ on the complexified Banach space $X^\C$ is the generator of a $C_0$-semigroup $(T^\C(t))_{t \geq 0}$ on $X^\C$ and $T^\C(t)$ is just the complexification of $T(t)$ for all $t \geq 0$.
 For any Banach spaces $X$ and $Y$, we denote by $\B(X,Y)$ the Banach space of bounded linear operators $T: X \rightarrow Y$, equipped with the operator norm $\norm{\cdot} = \norm{\cdot}_{\B(X,Y)}$.
 In the special case $X = Y$ we also write $\B(X) := \B(X,X)$.
 For any Banach space $E$, any compact set $K \subseteq \R^n$ and any open set $U \subseteq \R^n$, numbers $k \in \N_0 := \{0, 1 , \ldots \}$ and $p \in [1, \infty]$ we denote by
  $C(K;E)$, $C^k(K;E)$, $L_p(\Omega;E)$ and $W_p^k(\Omega;E)$ (special case $p = 2$: $H^k(\Omega;E) := W_2^k(\Omega;E)$) the spaces of $E$-valued continuous functions, $E$-valued $k$-times continuously differentiable functions, the $E$-valued Bochner-Lebesgue spaces and the $E$-valued Bochner-Sobolev spaces of degree $k$, with norms
   \begin{align*}
    \norm{f}_{C(K;E)}
     &:= \norm{f}_\infty
     := \sup_{e \in K} \norm{f(e)}_E
     \\
    \norm{f}_{C^k(K;E)}
     &:= \norm{f}_{C^k}
     := \sum_{\abs{\alpha} \leq k} \norm{\partial^{\alpha} f}_{C(K;E)}
     \\
    \norm{f}_{L_p(U;E)}
     &:= \norm{f}_p
     := \begin{cases}
       \left( \int_U \norm{f(x)}_E^p \, \dd x \right)^{1/p}, \quad &p \in [1, \infty) \\
       \operatorname{ess sup}_{x \in U} \norm{f(x)}_E, \quad &p = \infty
       \end{cases}
     \\
    \norm{f}_{W^k_p(U;E)}
     &:= \norm{f}_{k,p}
     := \begin{cases}
      \left( \sum_{\abs{\alpha} \leq k} \norm{\partial^\alpha f}_p^p \right)^{1/p}, \quad &p \in [1, \infty) \\
      \sum_{|\alpha| \leq k} \norm{\partial^\alpha f}_{\infty}, \quad &p = \infty.
       \end{cases}
   \end{align*}
 The notions $C(K;E)$ and $C^k(K;E)$ also extend to closed subsets $K$ of more general topological vector spaces $F$.
 For $p = 2$ and any Hilbert space $E$ with inner product $\sp{\cdot}{\cdot}_E$, the spaces $L_2(\Omega;E)$ and $H^k(\Omega;E)$ are Hilbert spaces with standard inner products
  \begin{align*}
   \sp{f}{g}_{L_2(\Omega;E)}
    &= \sp{f}{g}_{L_2}
    = \int_\Omega \sp{f(x)}{g(x)}_E \, \dd x \\
    \sp{f}{g}_{H^k(\Omega;E)}
    &= \sp{f}{g}_{H^k}
    = \sum_{|\alpha| \leq k} \int_\Omega \sp{\partial^\alpha f(x)}{\partial^{\alpha} g(x)}_E \, \dd x.
  \end{align*}
 Note that for $E = \K^d$ the $\K^d$-valued Bochner-Lebesgue and Bochner-Sobolev space are (up to an isomorphism) nothing but the $d$-fold product of the usual Lebesgue spaces and the usual Sobolev spaces, resp., i.e.\
  \[
   L_p(\Omega;\K^d)
    \cong \prod_{j=1}^d L_p(\Omega;\K),
    \quad
   W_p^k(\Omega;\K^d)
    \cong \prod_{j=1}^d W_p^k(\Omega;\K),
    \quad
   H^k(\Omega;\K^d)
    \cong \prod_{j=1}^d H^k(\Omega;\K).
  \]
 In particular, for $E = \R^d$ the strongly measurable functions are simply the measurable functions.

 \subsubsection{Some basic facts on semigroup theory}
 The focus of this manuscript lies on well-posedness (in the sense of semigroups) and stability for linear closure relation to boundary control and observation systems of infinite-dimensional port-Hamiltonian type. Therefore, let us recall some basic definitions and important theoretical results from semigroup theory that will be used heavily later on.
 
 We start with the definition of a $C_0$-semigroup.
 
 \begin{definition}[$C_0$-semigroup]
  Let $X$ be a Banach space and $(T(t))_{t \geq 0}$ be a family of bounded linear operators on $X$.
  Then $(T(t))_{t \geq 0}$ is called strongly continuous semigroup (for short: $C_0$-semigroup), if it has the following properties:
   \begin{enumerate}
    \item
     $T(0) = I$, the identity map on $X$,
    \item
     $T(s+t) = T(s) T(t)$ for all $s, t \geq 0$, (semigroup property) and
    \item
     $T(\cdot) x_0 \in C(\R_+; X)$ for every $x_0 \in X$, i.e.\ $(T(t))_{t \geq 0}$ has continuous trajectories (strong continuity).
   \end{enumerate}
  A $C_0$-semigroup $(T(t))_{t \geq 0}$ is called (strongly continuous) \emph{contraction semigroup}, if the operatornorm $\norm{T(t)}_{\B(X)} \leq 1$ for all $t \geq 0$.
 \end{definition}
  The existence of a $C_0$-semigroup $(T(t))_{t \geq 0}$ is closely related to well-posedness of the abstract Cauchy problem
   \[
    \frac{\dd}{\dd t} x(t)
     = A x(t)
     \quad (t \geq 0),
     \quad
    x(0)
     = x_0
     \tag{ACP}
     \label{ACP}
   \]
  in the sense of \emph{existence} and \emph{uniqueness} of solutions which \emph{continuously depend on the initial datum}.
  Roughly speaking, the abstract Cauchy problem \eqref{ACP} has for every initial datum $x_0 \in \dom(A)$ a unique classical solution $x \in C^1(\R_+; X)$ with values in $\dom(A)$, and the solution continuously depends on the initial datum $x_0$, if and only if there is a $C_0$-semigroup $(T(t))_{t \geq 0}$ on $X$ such that $A x = \lim_{t \rightarrow 0+} \frac{T(t) x - x}{x}$ for every $x \in \dom(A) = \{x \in X: \, \text{this limit exits}\}$, i.e.\ $A$ is the \emph{generator} of $(T(t))_{t \geq 0}$, and then $x(t) := T(t) x_0$ defines the unique classical solution, for every $x_0 \in \dom(A)$.
  For a precise statement of this result, see e.g.\ Proposition II.6.6 in \cite{EngelNagel_2000}.
  
  Linear operators $A$ generating a $C_0$-semigroup $(T(t))_{t \geq 0}$ can be exactly characterised by the general Hille-Yosida Theorem due to Feller, Miyadera and Phillips; see e.g.\ Theorem III.3.8 in \cite{EngelNagel_2000}. In this paper, however, all appearing semigroups will be contraction semigroups on Hilbert spaces, so that the Hilbert space version of the Lumer-Phillips Theorem, a special case of the Hille-Yosida theorem, and with conditions which are much easier to handle, can be applied.
  
  \begin{theorem}[Lumer-Phillips]
   Let $X$ be a Hilbert space with inner product $\sp{\cdot}{\cdot}$.
   Further, let $A: \dom(A) \subseteq X \rightarrow X$ be a densely defined, closed linear operator.
   Then $A$ generates a strongly continuous contraction semigroup if and only if
    \begin{enumerate}
     \item
      $A$ is dissipative, i.e.\ $\Re \sp{A x}{x} \leq 0$ for all $x \in X$, and
     \item
      $\ran (\lambda - A) = X$ for some (then: all) $\lambda > 0$.
    \end{enumerate}
   \begin{proof}
    We refer to Theorem II.3.15 in \cite{EngelNagel_2000} for the general Banach space version thereof.
   \end{proof}
  \end{theorem}
  
 To describe the long-time behaviour of the solutions to the abstract Cauchy problem \eqref{ACP} in terms of the $C_0$-semigroup associated to it, several notions of stability exist. Here, we are interested in \emph{strong stability} and \emph{uniform exponential stability} which are defined as follows.
 Note that these stability concepts coincide for finite dimensional Banachspaces $X$, but are distinct if $\dim X = \infty$.
 
 \begin{definition}[Stability concepts]
  Let $(T(t))_{t \geq 0}$ be a $C_0$-semigroup on some Banach space $X$.
   \begin{enumerate}
    \item
     The semigroup is called (asymptotically) \emph{strongly stable} if for every $x \in X$ one has
      \[
       T(t) x \rightarrow 0
        \quad \text{in } X.
      \]
    \item
     It is called \emph{uniformly exponentially stable}, if there are constants $M \geq 1$ and $\omega < 0$ such that
      \[
       \norm{T(t)}_{\B(X)}
        \leq M \ee^{\omega t},
        \quad
        t \geq 0.
      \]
   \end{enumerate}
 \end{definition}
 
 Stability properties of a $C_0$-semigroup can be tested via certain spectral properties and bounds on the resolvent operators, see the following two theorems which will be employed later on.
 
 \begin{theorem}[Arendt-Batty-Lyubich-V\~u]
  Let $(T(t))_{t \geq 0}$ be a bounded $C_0$-semigroup on some Banach space $X$, i.e.\ there is $M \geq 1$ such that $\norm{T(t)}_{\B(X)} \leq M$ for all $t \geq 0$.
  Further assume that its generator $A$ has compact resolvent, i.e.\ $(\lambda - A)^{-1}: X \rightarrow X$ is a compact operator for some (then: all) $\lambda \in \rho(A)$.
  In this case, $(T(t))_{t \geq 0}$ is strongly stable if and only if the point spectrum satisfies $\sigma_p(A) \subseteq \C_0^- = \{\lambda \in \C: \Re \lambda < 0\}$.
 \end{theorem}
 \begin{proof}
  See Theorem V.2.21 in \cite{EngelNagel_2000} for the general version of this Tauberian type theorem.
 \end{proof}
 
 \begin{theorem}[Gearhart-Pr\"uss-Huang]
  Let $(T(t))_{t \geq 0}$ be $C_0$-semigroup on some Hilbert space $X$.
  It is uniformly exponentially stable if and only if the following two properties hold true
   \begin{enumerate}
    \item
     $\sigma(A) \subseteq \C_0^-$, i.e.\ the spectrum lies in the complex left half-plane, and
    \item
     $\sup_{\beta \in \R} \norm{(\ii \beta - A)^{-1}} < \infty$, i.e.\ the resolvent operators are uniformly bounded on the imaginary axis.
   \end{enumerate}
 \end{theorem}
 \begin{proof}
  See Theorem V.1.11 in \cite{EngelNagel_2000}.
 \end{proof}
 
 \begin{remark}
  \label{rem:GPH}
  For $\ii \R \subseteq \rho(A)$, the condition $\sup_{\beta \in \R} \norm{(\ii \beta - A)^{-1}} < \infty$ is equivalent to the following property:
  \begin{quote}
  For every sequence $(x_n, \beta_n)_{n \geq 1} \subseteq \dom(A) \times \R$ with
   \begin{enumerate}
    \item
     $\sup_{n \geq 1} \norm{x_n}_X < \infty$,
    \item
     $\abs{\beta_n} \rightarrow \infty$ as $n \rightarrow \infty$, and
    \item
     $A x_n - \ii \beta_n x_n \rightarrow 0$ as $n \rightarrow \infty$,
   \end{enumerate}
  it follows that $x_n \rightarrow 0$ in $X$ as $n \rightarrow \infty$.
  \end{quote}
 \end{remark}
 This characterisation proves quite helpful for stability analysis of port-Hamiltonian systems.
 For further details on semigroup theory, we refer to the monograph \cite{EngelNagel_2000}.

\subsection{Basic definitions} 

Within this subsection we introduce the notion of a (linear, infinite-dimensional) port-Hamiltonian system (in boundary control and observation form) as we use it later on for interconnection of several systems of port-Hamiltonian type to networks.
Let us start with the basic definition of a single open-loop infinite-dimensional port-Hamiltonian system in boundary control and observation form.

 \begin{definition}[Port-Hamiltonian System]
  \label{def:PHS}
  We call a triple $\mathfrak{S} = (\mathfrak{A}, \mathfrak{B}, \mathfrak{C})$ of linear operators an (open-loop, linear, infinite-dimensional) \emph{port-Hamiltonian system} (in boundary control and observation form) of order $N \in \N$, if
   \begin{enumerate}
    \item
     The (maximal) \emph{port-Hamiltonian operator} $\mathfrak{A}: \dom(\mathfrak{A}) \subseteq L_2(0,1;\K^d) \rightarrow L_2(0,1;\K^d)$ is a linear differential operator of the form
      \begin{align*}
       \mathfrak{A} x
        &= \sum_{k=0}^N P_k \frac{\dd^k}{\dd \z^k} (\H x) \\
       \dom(\mathfrak{A})
        &= \{ x \in L_2(0,1;\K^d): \, \H x \in H^N(0,1;\K^d) \}
      \end{align*}
     where $\H \in L_\infty(0,1;\K^{d \times d})$ is coercive on $L_2(0,1;\K^d)$, i.e.\ there is $m > 0$ such that
      \[
       \sp{\H(\z) \xi}{\xi}_{\K^d}
        \geq m \abs{\xi}_{\K^d}^2,
        \quad
        \xi \in \K^d,
         \quad
         \text{a.e. }\, \z \in (0,1),
      \]
     and $P_k \in \K^{d \times d}$ ($k = 1, \ldots, N$) are matrices satisfying the anti-/symmetry relations $P_k^* = (-1)^{k+1} P_k$ ($k = 1, \ldots, N$) and such that the matrix $P_N$, i.e.\ the matrix corresponding to the principal part of the differential operator $\mathfrak{A}$, is invertible, whereas $P_0 \in L_\infty(0,1;\K^{d \times d})$ may depend on the spatial variable $\z \in (0,1)$.
    \item
     The \emph{boundary input map} $\mathfrak{B}$ and the \emph{boundary output map} $\mathfrak{C}$ are linear $\K^{Nd}$-valued operators with common domain
      \[
       \dom(\mathfrak{A}) = \dom(\mathfrak{B}) = \dom(\mathfrak{C})
      \]
     of the form
      \begin{align*}
       \left( \begin{array}{c} \mathfrak{B} x \\ \mathfrak{C} x  \end{array} \right)
        &= \left[ \begin{array}{c} W_B \\ W_C \end{array} \right] \tau(\H x),
        &x \in \dom(\mathfrak{A}) \\
       \tau(y)
        &= \left( y(1), y'(1), \ldots, y^{(N-1)}(1), y(0), \ldots, y^{(N-1)}(0) \right) \in \K^{Nd},
        &y \in H^N(0,1;\K^d)
      \end{align*}
      for matrices $W_B, W_C \in \K^{Nd \times 2Nd}$ such that $\left[ \begin{smallmatrix} W_B \\ W_C \end{smallmatrix} \right]$ is invertible.
   \end{enumerate}
 \end{definition}

 \begin{remark}
  More generally, we call a triple $\mathfrak{S}_1 = (\mathfrak{A}_1, \mathfrak{B}_1, \mathfrak{C}_1)$ with $\mathfrak{B}_1, \mathfrak{C}_1: \dom(\mathfrak{B}_1) = \dom(\mathfrak{A})_1 = \dom(\mathfrak{C}_1) \rightarrow \K^k$ for some $k \in \{0, 1, \ldots, Nd\}$ a port-Hamiltonian system as well, if there are linear operators $\mathfrak{A}$, $\mathfrak{B} = (\mathfrak{B}_0, \mathfrak{B}_1)$ and $\mathfrak{C} = (\mathfrak{C}_0, \mathfrak{C}_1)$ such that
   \[
    \mathfrak{S} = (\mathfrak{A}, \mathfrak{B}, \mathfrak{C})
     \quad
     \text{is port-Ham.\ system in the sense of Definition } \ref{def:PHS}
     \quad
     \text{and}
     \quad
     \mathfrak{A}_1 = \mathfrak{A}|_{\ker \mathfrak{B}_0}.
   \]
  This tacit convention makes it possible to consider a partial interconnection of port-Hamiltonian systems (of same order $N$) to be a port-Hamiltonian system itself.
 \end{remark}

Whenever $\mathfrak{S}$ is a port-Hamiltonian system on the space $L_2(0,1;\K^d)$, by coercivity of $\H$ the sesquilinear form
     \[
      \sp{\cdot}{\cdot}_\H: \quad
       \left( L_2(0,1;\K^d) \right)^2 \rightarrow \K,
       \quad
       \sp{f}{g}_\H
        := \int_0^1 \sp{f(\z)}{\H(\z) g(\z)}_{\K^d} \, \dd \z
     \]
    defines an inner product on $L_2(0,1;\K^d)$ and the corresponding norm $\norm{\cdot}_\H$ is equivalent to the standard norm $\norm{\cdot}_{L_2}$.
    We call $\sp{\cdot}{\cdot}_\H$ the \emph{energy inner product} and set the \emph{energy state space} $X$ to be the Hilbert space $L_2(0,1;\K^d)$ equipped with inner product $\sp{\cdot}{\cdot}_X := \sp{\cdot}{\cdot}_\H$ (and hence, the \emph{energy norm} $\norm{\cdot}_X = \norm{\cdot}_\H$).
    Note that the operator $\mathfrak{A}: \dom(\mathfrak{A}) \subseteq X \rightarrow X$ is a closed operator as a conjunction of the continuous matrix multiplication operator $\H(\cdot)$ on $X$ and the closed (thanks to $P_N$ being invertible) differential operator $\sum_{k=0}^N P_k \frac{\dd^k}{\dd \z^k}: \, H^N(0,1;\K^d) \subseteq X \rightarrow X$. 
    
 \begin{remark}
  For the moment, let us consider an infinite-dimensional port-Hamiltonian system $\mathfrak{S}$ with $P_0 = 0$.
    Then for every $x, y \in X$ such that $\H x, \H y \in C_c^\infty(0,1;\K^d)$ it holds via integration by parts that
     \[
      \sp{\mathfrak{A} x}{y}_X
       = - \sp{x}{\mathfrak{A} y}_X
     \]
    i.e.\ the operator is formally skew-symmetric on the space $X$.
    For the case $P_0 \not= 0$ this holds exactly in the case that $P_0(\z)^* = - P_0(\z)$ for a.e.\ $\z \in (0,1)$.
 \end{remark}

 When looking for dissipative closure relations of the type $\mathfrak{B} x = K \mathfrak{C} x$ for some matrix $K \in \K^{Nd \times Nd}$ it is convenient to have the property of passivity for the port-Hamiltonian system.

 \begin{definition}[Passive Systems]
  Let $\mathfrak{S} = (\mathfrak{A}, \mathfrak{B}, \mathfrak{C})$ be a port-Hamiltonian system in boundary control and observation form.
  The system $\mathfrak{S}$ is called
   \begin{enumerate}
    \item
     \ldots \emph{impedance passive}, if
      \[
       \Re \sp{\mathfrak{A} x}{x}_X \leq \sp{\mathfrak{B} x}{\mathfrak{C} x}_{\K^{Nd}},
        \quad
        x \in \dom(\mathfrak{A}).
      \]
    \item
     \ldots \emph{scattering passive}, if
      \[
       \Re \sp{\mathfrak{A} x}{x}_X \leq \abs{\mathfrak{B} x}_{\K^{Nd}}^2 - \abs{\mathfrak{C} x}_{\K^{Nd}}^2,
        \quad
        x \in \dom(\mathfrak{A}).
      \]
   \end{enumerate}
 \end{definition}
 
 \begin{remark}
  \begin{enumerate}
   \item
    Note that both notions of passivity do not depend on the Hamiltonian energy density matrix function $\H$: A port-Hamiltonian system $\mathfrak{S}$ is impedance passive (scattering passive) if and only if the corresponding port-Hamiltonian system for $\H = I$ is impedance passive (scattering passive).
   \item
    A port-Hamiltonian system is impedance passive (scattering passive) if and only if the symmetric part $\He P_0(\z) := \frac{1}{2} (P_0(\z) + P_0(\z)^*)$ of $P_0$ is negative semi-definite for a.e.\ $\z \in (0,1)$ and $W_B, W_C$ satisfy a certain matrix condition (including also the matrices $P_k$ ($k \geq 1$)), see \cite{LeGorrecZwartMaschke_2005}.
  \end{enumerate}
 \end{remark}
 
 Further convention: Besides the energy state space $X = L_2(0,1; \K^d)$ (equipped with the energy inner product), extended energy state spaces $\hat X = X \times X_c$ for some finite dimensional Hilbert space $X_c$ will be used as well, and its elements are denoted by $\hat x = (x, x_c) \in X \times X_c$.
 Operators acting on elements of such product energy state spaces are denoted by a hat, e.g.\ $\hat A$, $\hat {\mathfrak{A}}$, $\hat {\mathfrak{B}}$, $\hat {\mathfrak{C}}$ and $\hat T(t)$.

\section{Examples and Previous Results}
\label{sec:basic_definition}

We give some examples of dissipative closure relations which had been considered previously in the literature.
Additionally, we recall the main results on well-posedness and stability for these linear closure relations.
Starting from open-loop passive port-Hamiltonian systems one can easily obtain dissipative operators when closing with a suitable closure relation and possibly interconnects the port-Hamiltonian system with either another port-Hamiltonian system or an impedance passive control and observation system.
Below we list some particular examples for such static or dynamic closure relations.

 \begin{example}[Dissipative, static closure]
  Assume that $\mathfrak{S}$ is an impedance passive port-Hamiltonian system and let $K \in \K^{Nd \times Nd}$ be a matrix with negative semi-definite symmetric part
   \[
    \He K
     := \frac{1}{2} \left(K + K^* \right)
     \leq 0
   \]
  (the simplest choice being $K = 0$) and define $A: \dom(A) \subseteq X \rightarrow X$ by
   \begin{align*}
    A x
     &:= \mathfrak{A} x
     \\
    \dom(A)
     &:= \{ x \in \dom(\mathfrak{A}): \, \mathfrak{B} x = K \mathfrak{C} x \}.
   \end{align*}
  Then $A$ is a dissipative operator on $X$, and therefore generates a strongly continuous contraction semigroup $(T(t))_{t \geq 0}$ on $X$, see Theorem \ref{thm:generation_thm_static} below.
 \end{example}
 
 \begin{proof}
 Dissipativity can be checked easily, using the impedance passivity of $\mathfrak{S}$ and the negative semi-definiteness of $\He (K)$.
 Then, the generator property follows from Theorem \ref{thm:generation_thm_static} below.
 \end{proof}
 
 The first result on well-posedness of infinite-dimensional port-Hamiltonian systems has been due to Y.~Le~Gorrec, H.~Zwart and B.~Maschke \cite{LeGorrecZwartMaschke_2005} who proved that for operators of port-Hamiltonian type a dissipative linear closure relation, i.e.\ dissipative boundary conditions, is already enough for the corresponding abstract Cauchy problem
  \[
   \begin{cases}
    \frac{\dd}{\dd t} x(t) = A x(t),
     \quad
     t \geq 0
     \\
    x(0)
     = x_0 \in X
   \end{cases}
  \]
 to be well-posed, i.e.\ for every initial value $x_0 \in \dom(A)$, there is a unique classical solution $x \in C^1(\R_+;X) \cap C(\R_+; \dom(A))$ of this Cauchy problem, where $\dom(A)$ is equipped with the graph norm of $A$, and the solution depends continuously on the initial datum $x_0$ and has non-increasing energy $\frac{1}{2} \norm{x(t)}_X^2$.
 In other words, if $A$ is dissipative, then $A$ generates a strongly continuous contraction semigroup $(T(t))_{t \geq 0}$ on $X$.
 This is
 
  \begin{theorem}[Le~Gorrec, Zwart, Maschke (2005)]
   \label{thm:generation_thm_static}
   Let $\mathfrak{S} = (\mathfrak{A}, \mathfrak{B}, \mathfrak{C})$ be any port-Hamiltonian system and $K \in \K^{Nd \times Nd}$.
   Then, the operator \[ A = \mathfrak{A}|_{\ker (\mathfrak{B} - K \mathfrak{C})} \] generates a contractive $C_0$-semigroup $(T(t))_{t \geq 0}$ on $X = (X, \sp{\cdot}{\cdot}_X)$ if and only if $A$ is dissipative on $X$.
  \end{theorem}
  
  \begin{proof}
  For the proof, see \cite{LeGorrecZwartMaschke_2005}.
  \end{proof}

 While well-posedness for itself is an important property, often one is not satisfied with well-posedness alone, but also looks for stability properties of the abstract Cauchy problem associated to $A$.
 In contrast to well-posedness -- for which the case of general coercive $\H \in L_\infty(0,1;\K^{d \times d})$ can be reduced to the special case $\H = I$, see Lemma 7.2.3 in \cite{JacobZwart_2012} -- stability properties of $A$ may (and will, as Engel \cite{Engel_2013} showed) generally depend on the Hamiltonian density matrix function which can be seen as a multiplicative perturbation to the operator $A$ for $\H = I$.
 However, as has been known for the wave equation, the Timoshenko beam equation and the Euler-Bernoulli beam equation, there are examples where one could expect that some classes of linear boundary feedback relations imply asymptotic stability, i.e.\ trajectory-wise for every initial datum $x_0 \in X$, or even uniform exponential stability, i.e.\ the energy decay can be bounded by a exponentially decaying function times initial energy, where the exponential decay rate is \emph{independent} of the initial datum $x_0 \in X$.
 For the particular case of first order port-Hamiltonian systems such stability results have first been proved in the Ph.D. thesis \cite{Villegas_2007} and the research article \cite{VillegasEtAl_2009}, showing that for first order port-Hamiltonian systems it is enough to damp at one end, whereas at the other end arbitrary conservative or dissipative boundary conditions can be imposed.
 
  \begin{theorem}[Villegas et al.\ (2009)]
   Let $A$ be a port-Hamiltonian operator closed with a dissipative boundary condition as in Theorem \ref{thm:generation_thm_static}.
   Further assume that the order of the port-Hamiltonian system is $N = 1$, the Hamiltonian density matrix function $\H: [0,1] \rightarrow \K^{d \times d}$ is Lipschitz continuous and one has the following estimate:
    \begin{equation}
     \Re \sp{A x}{x}_X
      \leq - \kappa \abs{(\H x)(0)}^2,
      \quad
      x \in \dom(A)
      \nonumber
    \end{equation}
   where $\kappa > 0$ does not depend on $x \in \dom(A)$.
   Then, the $C_0$-semigroup $(T(t))_{t \geq 0}$ generated by $A$ is uniformly exponentially stable.
  \end{theorem}
  
  \begin{proof}
  For the proof, see \cite{VillegasEtAl_2009}, where $\H \in C^1([0,1];\K^{d \times d})$ had been assumed.
  However, the proof carries over to $\H \in \Lip([0,1];\K^{d \times d})$, see \cite{AugnerJacob_2014}.
  \end{proof}
 
 Actually, up to now there are at least three approaches known to prove the stability result above:
  \begin{enumerate}
   \item
    The original proof in \cite{VillegasEtAl_2009} is based on some \emph{sideways-energy estimate} (as it is called in \cite{CoxZuazua_1995}) or \emph{final observability estimate}
     \[
      \norm{T(\tau) x_0}_X
       \leq c \norm{[T(\cdot)x](0)}_{L_2(0,\tau;\K^d)},
       \quad
       x_0 \in X 
     \]
    for some sufficiently large $\tau > 0$ and some $c > 0$.
    This approach is very helpful when considering non-linear dissipative boundary feedback, cf.\ \cite{Augner_2018+}, however it seems difficult to extend this result to higher order port-Hamiltonian systems, e.g.\ Euler-Bernoulli type systems.
   \item
    A frequency domain approach, namely showing that the resolvent $(\ii \beta - A)^{-1} \in \B(X)$ exists for all $\beta \in \R$ and is uniformly bounded, employing the Arendt-Batty-Lyubich-V\~u Theorem (asymptotic stability) and the Gearhart-Greiner-Pr\"uss-Huang Theorem (uniform exponential stability) has been applied in \cite{AugnerJacob_2014}.
    This approach is suitable for interconnection with finite dimensional control systems \cite{AugnerJacob_2014}, and as we later see, for linear interconnection with other port-Hamiltonian systems; see the Section \ref{sec:hybrid_multi_phs}.
   \item
    Third, a multiplier approach leading to a Lyapunov function is possible as well.
    Again, this approach is suitable for non-linear feedback interconnection, especially of dynamic type \cite{Augner_2018+}.
  \end{enumerate}
 Strictly speaking, there also is a fourth approach (actually the oldest one!), but it only works under much stronger regularity assumptions, namely analyticity of $\H$; see \cite{RauchTaylor_1974}.

 The second example for a class of closure relations consists of dissipative or conservative feedback interconnection with a linear control system. 

 \begin{example}[Interconnection of a PHS with a finite-dimensional controller]
  Let $\mathfrak{S} = (\mathfrak{A}, \mathfrak{B}, \mathfrak{C})$ be an impedance passive port-Hamiltonian system and $\Sigma_c = (A_c, B_c, C_c, D_c) \in \B(X_c) \times \B(U_c; X_c) \times \B(X_c; Y_c) \times \B(U_c; Y_c)$, for some finite-dimensional Hilbert spaces $X_c$, $U_c$ and $Y_c$, be a finite dimensional control system
   \[
    \begin{cases}
     \frac{\dd}{\dd t} x_c(t)
      &= A_c x_c(t) + B_c u_c(t)
      \\
     y_c(t)
      &= C_c x_c(t) + D_c u_c(t),
      \quad
      t \geq 0
    \end{cases}
   \]
  and impedance passive, i.e.\ $U_c = Y_c$ and
   \[
    \Re \sp{A_c x_c + B_c u_c}{x_c}_{X_c}
     \leq \Re \sp{u_c}{C_c x_c + D_c u_c}_{U_c},
     \quad
     x_c \in X_c, u_c \in U_c.
   \]
  Further assume that $U_c = Y_c = \K^{Nd}$.
  Then $\hat A: \dom(\hat A) \subseteq \hat X \rightarrow \hat X$,
   \begin{align*}
    \hat A(x,x_c)
     &:= (\mathfrak{A} x, A_c x_c + B_c \mathfrak{C} x)
     \\
    \dom(\hat A)
     &:= \{ (x,x_c) \in \dom(\mathfrak{A}) \times X_c: \, \mathfrak{B} x = - C_c x_c - D_c \mathfrak{C} x \}
   \end{align*}
  resulting from the standard feedback interconnection
   \[
    u_c = \mathfrak{C} x
     \quad \text{and} \quad
    \mathfrak{B} x = - y_c
   \]
  is a dissipative operator on the product Hilbert space $\hat X = X \times X_c$, and thus generates a strongly continuous contraction semigroup on $\hat X$.
 \end{example}
 
 \begin{proof}
 Dissipativity follows from impedance passivity of both subsystems and some easy computation.
 For the assertion on semigroup generation, we need the following result.
 \end{proof}

 \begin{theorem}[Villegas (2007), Augner, Jacob (2014)]
  Let $\mathfrak{S} = (\mathfrak{A}, \mathfrak{B}, \mathfrak{C})$ be an infinite-dimensional port-Hamiltonian system and $\Sigma_c = (A_c, B_c, C_c, D_c)$ be a finite dimensional linear control system.
  The operator
   \begin{align*}
    \hat A (x,x_c)
     &:= (\mathfrak{A} x, A_c x_c + B_c \mathfrak{C} x) \\
    \dom (\hat A)
     &:= \{ (x,x_c) \in \dom(\mathfrak{A}) \times X_c: \, \mathfrak{B} x = - C_c x_c - D_c \mathfrak{C} x \}
   \end{align*}
  generates a contractive $C_0$-semigroup \ $(\hat T(t))_{t \geq 0}$ on the product Hilbert space $\hat X = X \times X_c$ if and only if it is dissipative.
 \end{theorem}

 \begin{proof}
 A result like this has probably first been stated in the Ph.D. thesis \cite{Villegas_2007}, however under some slightly more restrictive conditions on the infinite-dimensional port-Hamiltonian system and the finite dimensional linear control system $\Sigma_c$.
 For the general situation stated above, see \cite{AugnerJacob_2014}.
 \end{proof}
 
 As for the static feedback case, the generation theorem is based on the Lumer-Phillips Theorem which states that besides dissipativity of an operator a range condition, namely $\ran (\hat A - \lambda I) = \hat X$ for some (then: all) $\lambda > 0$ is sufficient (and necessary as well) for the operator $\hat X$ to generate a strongly continuous contraction semigroup.
 Here, the range condition for $\hat A$ is reduced to a range condition for some operator $A_\mathrm{cl}$ (with suitable static linear closure relations), i.e.\ the generation theorem for the dynamic case already relies on (the proof of) the generation theorem for the static case.
 
 As for the static case, one can ask for sufficient (hopefully $\H$-independent) conditions on the damping via the controller such that the hybrid PDE-ODE systems is uniformly exponentially stable, i.e.\ its total energy decays uniformly exponentially to zero for all initial data $(x_0, x_{c,0})$.
 
 \begin{theorem}[Ramirez, Zwart, Le Gorrec (2013), Augner, Jacob (2014)]
  Assume that $\mathfrak{S}$ is an impedance passive port-Hamiltonian system of order $N = 1$ and such that $\H: [0,1] \rightarrow \K^{d \times d}$ is Lip\-schitz continuous, and let $\mathfrak{S}$ be interconnected by standard feedback interconnection
   \[
    \mathfrak{B} x = - y_c,
     \quad
     u_c = \mathfrak{C} x
   \]  
  with a strictly impedance passive finite-dimensional control system $\Sigma_c = (A_c, B_c, C_c, D_c)$ in the sense that
   \[
    \Re \sp{A_c x_c + B_c u_c}{x_c}_{X_c}
     \leq \Re \sp{C_c x_c + D_c u_c}{u_c}_{U_c} - \kappa \abs{D_c u_c}^2,
     \quad
     x_c \in X_c, \, u_c \in U_c
   \]
  and such that $(\ee^{t A_c})_{t \geq 0}$ is uniformly exponentially stable, $\ker D_c \subseteq \ker B_c$ and
   \[
    \abs{\mathfrak{B} x}^2 + \abs{D_c \mathfrak{C} x}^2
     \gtrsim \abs{(\H x)(0)}^2,
     \quad
     x \in \dom(\mathfrak{A}),
   \]
  then the $C_0$-semigroup $(\hat T(t))_{t \geq 0}$ is uniformly exponentially stable.
 \end{theorem}
 
 \begin{proof}
 For the situation where $\Sigma_c$ is strictly input passive, in particular $D_c > 0$ is positive definite, see \cite{RamirezZwartLeGorrec_2013}.
 The (slightly) generalised result can be found in \cite{AugnerJacob_2014}.
 \end{proof}
 
 The general idea for the proof of this dynamic feedback result is to consider the state variable $x_c$ as a perturbation to the static boundary feedback one would have for $x_c = 0$, namely $\mathfrak{B} x = - D_c \mathfrak{C} x$.
 From the impedance passivity of $\mathfrak{S}$ and $\Sigma_c$ one then obtains a dissipation estimate of the type
  \[
   \Re \sp{\mathfrak{A} x}{x}_X
    \leq - \kappa \left( \abs{\mathfrak{B} x}_U^2 + \abs{D_c\mathfrak{C} x}_U^2  \right)
    \lesssim - \kappa \abs{(\H x)(0)}^2,
    \quad
    x \in \ker (\mathfrak{B} + D_c \mathfrak{C}).
  \]
 Uniform exponential energy decay then can be expected from the static feedback result and the exponential stability of $(\ee^{t A_c})_{t \geq 0}$ ensures that the perturbation does not hurt this property.
 
 Besides dynamic feedback, the other direction of generalisation aims at higher order port-Hamiltonian systems.
 The first result in this perspective follows rather easily from the Arendt-Batty-Lyubich-V\~u Theorem and considerations on possible eigenfunctions with eigenvalues $\ii \beta$ for some $\beta \in \R$, but only gives \emph{asymptotic stability}.

  \begin{proposition}[Augner, Jacob (2014)]
   Let $A$ be a port-Hamiltonian operator of order $N \in \N$, resulting from linear closure of a port-Hamiltonian system $\mathfrak{S} = (\mathfrak{A}, \mathfrak{B}, \mathfrak{C})$ by a linear closure relation $\mathfrak{B} x = K \mathfrak{C} x$ and assume that $\H: [0,1] \rightarrow \K^{d \times d}$ is Lipschitz continuous, and
    \begin{equation}
     \Re \sp{A x}{x}_X
      \leq - \kappa \sum_{k=0}^{N-1} \abs{(\H x)^{(k)}(0)}^2,
      \quad
      x \in \dom(A)
      \nonumber
    \end{equation}
   for some $\kappa > 0$.
   Then, the $C_0$-semigroup $(T(t))_{t \geq 0}$ generated by $A$ is (asymptotically) strongly stable.
  \end{proposition}
  \begin{proof}
  See \cite{AugnerJacob_2014}.
  \end{proof}
  
  In \cite{AugnerJacob_2014}, it has also been shown that generally one cannot expect uniform exponential stability, namely there is a counter example (Schr\"odinger equation) where full dissipation at one end and a correct choice of conservative boundary conditions at the other end only lead to asymptotic stability, but not to uniform exponential stability.
  (For the counter example, one can compute the resolvents $(\ii \beta - A)^{-1}$ and show that they are not uniformly bounded for $\beta \in \R$.)
  However, under further conditions on the boundary conditions at the conservative end, more can be said:
  
  \begin{theorem}[Augner, Jacob (2014)]
   Let $A$ be a (closed by linear static feedback) port-Hamiltonian operator $A$ of order $N = 2$ and $\H \in \Lip([0,1];\K^{d \times d})$ be Lipschitz continuous and assume that
    \begin{equation}
     \Re \sp{A x}{x}_X
      \leq - \kappa \left( \abs{(\H x)(0)}^2 + \abs{(\H x)'(0)}^2 + \abs{(\H x)(1)}^2 \right)
      \nonumber
    \end{equation}
   for all $x \in \dom(A)$.
   Then the $C_0$-semigroup $(T(t))_{t \geq 0}$ generated by $A$ is uniformly exponentially stable.
  \end{theorem}
  \begin{proof}
  See \cite{AugnerJacob_2014}.
  \end{proof}
  
 \begin{remark}
  By the way, fully dissipative boundary conditions at \emph{both ends}
   \[
    \Re \sp{A x}{x}_X
     \leq - \kappa \left( \sum_{k=0}^{N-1} \abs{(\H x)^{(k)}(0)}^2 + \abs{(\H x)^{(k)}(1)}^2 \right),
     \quad
     x \in \dom(A)
   \]  
  for some $\kappa > 0$, always lead to uniform exponential stability, for all port-Hamiltonian systems of arbitrary order $N \in \N$ and for Lipschitz continuous $\H$, see \cite{Augner_2016}.
 \end{remark}
 
 In this article, we are concerned with the case where a port-Hamiltonian system $\mathfrak{S}^1$ is interconnected with further port-Hamiltonian systems in a energy preserving or dissipative way, e.g.\
 
 \begin{example}[Interconnection of impedance passive PHS]
  Let $\mathfrak{S}^1$ and $\mathfrak{S}^2$ be two impedance passive port-Hamiltonian systems with $N^1 d^1 = N^2 d^2$, i.e.\ the input and output spaces for $\mathfrak{S}^1$ and $\mathfrak{S}^2$ should have the same dimension, then the operator $A: \dom(A) \subseteq X \rightarrow X$ defined by
   \begin{align*}
    A (x^1,x^1)
     &:= (\mathfrak{A}^1 x^1, \mathfrak{A}^1 x^2)
     \\
    \dom(A)
     &= \{ x = (x^1,x^2) \in \dom(\mathfrak{A}^1) \times \dom(\mathfrak{A}^2): \, \mathfrak{B}^1 x^1 = - \mathfrak{C}^2 x^2, \, \mathfrak{B}^2 x^2 = \mathfrak{C}^1 x^1 \}
   \end{align*}
  is dissipative on the product Hilbert space $X = X^1 \times X^2$ and generates a strongly continuous contraction semigroup $(T(t))_{t \geq 0}$ on $X$.
 \end{example}
 
 \begin{proof}
 The dissipativity of the operator $A$ can be checked using the impedance passivity of the two subsystems:
  \begin{align*}
   \Re \sp{A x}{x}_X
    &= \sum_{j = 1}^2 \Re \sp{\mathfrak{A}^j x^j}{x^j}_{X^j}
    \\
    &\leq \Re \sp{\mathfrak{B}^1 x^1}{\mathfrak{C}^1 x^1} + \Re \sp{\mathfrak{B}^2 x^2}{\mathfrak{C}^2 x^2}
    = 0,
    \quad
    x \in \dom(A).
  \end{align*}
 For the generation result, see Proposition \ref{prop:gen_static_feedback} in the next section.
 \end{proof}

 \begin{example}[Interconnection of scattering passive PHS]
  Let $\mathfrak{S}^1, \ldots, \mathfrak{S}^m$ be scattering passive port-Hamiltonian systems.
  Then
   \begin{align*}
    A(x^1, \ldots, x^m)
     &= (\mathfrak{A}^1 x^1, \ldots, \mathfrak{A}^m x^m)
     \\
    \dom(A)
     &= \{ x = (x^1, \ldots, x^m) \in \prod_{j=1}^m \dom(\mathfrak{A}^j), \, \mathfrak{B}^1 x^1 = 0, \, \mathfrak{B}^j x^j = \mathfrak{C}^{j-1} x^{j-1} \ (j \geq 1) \}
   \end{align*}
  is dissipative on $X = \prod_{j=1}^m X^j$ and generates a strongly continuous contraction semigroup $(T(t))_{t \geq 0}$ on $X$.
 \end{example}
 \begin{proof}
 Dissipativity can be checked easily by using the scattering-passivity of the subsystems $\mathfrak{S}^j$.
 For the generation result, see Proposition \ref{prop:gen_static_feedback} in the next section.
 \end{proof}
 
 We comment on stability properties later on.

\section{Port-Hamiltonian Systems: Networks}
\label{sec:PHS_networks}

 After recalling some known results on different static or dynamic closure relations for port-Hamiltonian systems, let us focus on the main topic of this paper, namely the interconnection of several infinite-dimensional port-Hamiltonian subsystems to a \emph{network} of port-Hamiltonian systems.
 Assume that $\mathcal{J} = \{1, 2, \ldots, m\}$ is a finite index set.
 Here, the number $m \in \N$ is the number of infinite-dimensional port-Hamiltonian subsystems $\mathfrak{S}^j = (\mathfrak{A}^j, \mathfrak{B}^j, \mathfrak{C}^j)$ the network consists of.
 Moreover, $\mathcal{J}_c = \{1, 2, \ldots, m_c\}$ denotes another index set, corresponding to a finite number of finite-dimensional \emph{linear control systems} $\Sigma_c^j = (A_c^j , B_c^j, C_c^j, D_c^j)$, we may interconnect the port-Hamiltonian systems with, the case $m_c = 0$, i.e.\ $\mathcal{J}_c = \emptyset$ being allowed, but w.l.o.g.\ we may always assume that $m_c = m \in \N$.

 We generally assume that $\mathfrak{S}^j = (\mathfrak{A}^j, \mathfrak{B}^j, \mathfrak{C}^j)$ $(j \in \mathcal{J})$ are (open-loop, linear, infinite-dimensional) port-Hamiltonian systems (on a one-dimensional spatial domain) on spaces $X^j = L_2(0,1;\K^{d_j})$ (all equipped with their respective energy norm $\| \cdot \|_{X^j} = \| \cdot \|_{\H^j}$, thus being Hilbert spaces for the energy inner product $\sp{\cdot}{\cdot}_{X^j} = \sp{\cdot}{\cdot}_{\H^j}$ and input and output spaces $U^j = Y^j = \K^{N^j d^j}$, and similarly $\Sigma_c^j = (A_c^j, B_c^j, C_c^j. D_c^j)$ $(j \in \mathcal{J}_c)$ are finite-dimensional linear control systems with finite dimensional state space $X_c^j$ and finite-dimensional input and output space $U_c^j = Y_c^j$.
 We further set
  \begin{align*}
   \hat X
    &:= X \times X_c
    := \prod_{j = 1}^m X^j \times \prod_{j=1}^{m_c} X_c^j,
    \\
   \hat U
    &:= U \times U_c
    := \prod_{j = 1}^m U^j \times \prod_{j=1}^{m_c} U_c^j
    = \prod_{j = 1}^m Y^j \times \prod_{j=1}^{m_c} Y_c^j
    =: Y \times Y_c
    =: \hat Y.
  \end{align*}
 We equip these spaces with their respective product inner product and the induced norms, i.e.\
  \begin{align*}
    \sp{\hat x}{\hat y}_{\hat X}
     &= \sum_{j \in \mathcal{J}} \sp{x^j}{y^j}_{X^j} + \sum_{j \in \mathcal{J}_c} \sp{x_c^j}{y_c^j}_{X_c^j},
     \quad
     \hat x, \hat y \in \hat X
     \\
    \| \hat x \|_{\hat X}
     &= \sqrt{ \sum_{j \in \mathcal{J}} \| x^j \|_{X^j}^2 + \sum_{j \in \mathcal{J}_c} \| x_c^j \|_{X_c^j}^2 },
     \quad
     \hat x \in \hat X
  \end{align*}
 and accordingly for $\hat U = \hat Y$.
  For $x \in \prod_{j \in \mathcal{J}} \dom(\mathfrak{A}^j)$ we write
   \begin{align*}
    \mathfrak{A} x
     &:= (\mathfrak{A}^1 x^1, \ldots, \mathfrak{A}^m x^m)
     \in X,
     \\
    \mathfrak{B} x
     &:= (\mathfrak{B}^1 x^1, \ldots, \mathfrak{B}^m x^m)
     \in U,
     \\
    \mathfrak{C} x
     &:= (\mathfrak{C}^1 x^1, \ldots, \mathfrak{C}^m x^m)
     \in Y.
   \end{align*}
  This defines linear operators \[ \mathfrak{A}: \dom(\mathfrak{A}) = \prod_{j \in \mathcal{J}} \dom(\mathfrak{A}^j) \subseteq X \rightarrow X, \quad \mathfrak{B}, \mathfrak{C}: \dom(\mathfrak{B}) = \dom(\mathfrak{C}) = \dom(\mathfrak{A}) \subseteq X \rightarrow U = Y. \]
  Also, we define $A_c \in \B(X_c)$, $B_c \in \B(U_c,X_c)$, $C_c \in \B(X_c, Y_c)$ and $D_c \in \B(U_c, Y_c)$ by
   \begin{align*}
    A_c x_c
     &= (A_c^1 x_c^1, \ldots, A_c^{m_c} x_c^{m_c}),
     \quad
     x_c \in X_c
     \\
    B_c u_c
     &= (B_c^1 u_c^1, \ldots, B_c^{m_c} u_c^{m_c}),
     \quad
     u_c \in U_c
     \\
    C_c x_c
     &= (C_c^1 x_c^1, \ldots, C_c^{m_c} x_c^{m_c}),
     \quad
     x_c \in X_c
     \\
    D_c u_c
     &= (D_c^1 u_c^1, \ldots, D_c^{m_c} u_c^{m_c}),
     \quad
     u_c \in X_c.
   \end{align*}
 Further, let
  \[
   E_c \in \B(Y;U_c),
    \quad
   E \in \B(Y_c; U). 
  \]
 Now, interconnect the subsystems via the relation
  \[
   \mathfrak{B} x
    = - E (C_c x_c + D_c E_c \mathfrak{C} x).
  \]
 \begin{remark}
  To keep the presentation as simple as possible, in this exposition we will always assume that $Y = U_c$ and $U = Y_c$ as well as $E_c = I$ and $E = I$ are the identity maps.
 \end{remark}
 We may then define the following operator $\hat A$ on $\hat X$
  \begin{align*}
   \hat A \hat x
    &:= (\mathfrak{A} x, A_c x_c + B_c \mathfrak{C} x)
    \\
   \dom(\hat A)
    &:= \left\{ \hat x = (x, x_c) \in \dom(\mathfrak{A}) \times X_c: \, \mathfrak{B} x = - (C_c x_c + D_c \mathfrak{C} x) \right\}.
  \end{align*}
  
  \begin{example}
   Let us consider two particular special cases:
    \begin{enumerate}
     \item
      If $m_c = 0$, i.e.\ $X_c = U_c = Y_c = \{ 0 \}$, we can identify $\hat X = X$, $\hat Y = Y$, $\hat U = U$ and $C_c x_c + D_c \mathfrak{C} x = D_c \mathfrak{C} x$, so that
       \begin{align*}
        \hat A \hat x
         &= \mathfrak{A} x
         \\
        \dom(\hat A)
         &= \{ x \in \dom(\mathfrak{A}): \mathfrak{B} x = - D_c \mathfrak{C} x \}.
       \end{align*}
      In this case, no finite dimensional control system is present and this just describes the interconnection of port-Hamiltonian systems $\mathfrak{S}_i$ by boundary feedback, with the special case $m = 1$ being the case of a port-Hamiltonian system closed by linear boundary feedback.
     \item
      If $m = m_c = 1$ and $U = Y_c = U_c = Y$, the operator $\hat A$ reads as
       \begin{align*}
        \hat A \hat x
         &= (\mathfrak{A} x, A_c x_c + B_c \mathfrak{C} x)
         \\
        \dom(\hat A)
         &= \{ \hat x \in \dom(\mathfrak{A}) \times X_c: \, \mathfrak{B} x = - C_c x_c - D_c \mathfrak{C} x \}
       \end{align*}
      so that we are in the case of dynamic boundary feedback with a finite dimensional control system interconnected by standard feedback interconnection with the port-Hamiltonian system.
    \end{enumerate}
  \end{example}
  
  \begin{remark}
   \begin{enumerate}
   \item
   Note that the abstract Cauchy problem
    \[
     \frac{\dd}{\dd t} \hat x(t)
      = \hat A \hat x(t)
      \, (t \geq 0), \,
      \quad
     \hat x(0)
      = (x_0, x_{c,0})
      \in X \times X_c
    \]
   is equivalent to the system of PDE and ODE
    \begin{align*}
     \frac{\dd}{\dd t} x(t)
      &= \mathfrak{A} x(t),
      \\
     \frac{\dd}{\dd t} x_c(t)
      &= A_c x_c(t) + B_c u_c(t),
      \\
     \mathfrak{B} x(t)
      &= - (C_c x_c + D_c u_c),
      \\
     u_c(t)
      &= \mathfrak{C} x(t),
      &&t \geq 0.
    \end{align*}
   \item
   The definition of $\hat A$ does, at first glance, not allow complex systems where the input into one finite dimensional control system depends upon the output from another finite dimensional controller.
   However, in most cases it should be possible, to merge such two finite control systems into a larger control system, by plugging in the equations of one of these systems into the other.
   \end{enumerate}
  \end{remark}

 Our proof of the general generation result, Theorem \ref{thm:generation_network-dynamic}, below  uses the special case where $X_c = \{0\}$, i.e.\ no finite dimensional control systems are present within the network.
 We, therefore, begin by considering the generation result for this particular special case. 
 
  \begin{proposition}
  \label{prop:gen_static_feedback}
  Assume that $X_c = \{0\}$. Then $\hat A$ generates a contractive $C_0$-semigroup on $\hat X \cong X$ if and only if $\hat A$ is dissipative.
 \end{proposition}
 
 \begin{proof}
 Since $\H = \operatorname{diag}_{j \in \mathcal{J}} \H^j$ is a strictly coercive (matrix) multiplication operator on $X$, by Lemma 7.2.3 in \cite{JacobZwart_2012} we can restrict ourselves to the case $\H = I \in \B(X)$.
 Further, let us for the moment assume that all $P_0^j = 0$ (or a constant matrix independent of $\z \in (0,1)$ with negative semi-definite symmetric part).
 Since $\mathfrak{B} x = \mathfrak{C} x = 0$ for all $x \in \prod_{j \in \mathcal{J}} C_c^\infty(0,1;\K^{d_j})$ and this set is dense in $X$, the operator $\hat A$ is densely defined, so that by the Lumer-Phillips Theorem, see e.g.\ Theorem II.3.15 in \cite{EngelNagel_2000}, it remains to prove that $\lambda I - \hat A$ is surjective for some $\lambda > 0$ whenever $\hat A$ is dissipative.
 Here, we choose $\lambda = 1$.
 Take $f = (f_j)_{j \in \mathcal{J}} \in \prod_{j \in \mathcal{J}} X^j = X$.
 Then we have to find $x \in \dom(\mathfrak{A})$ such that
  \begin{align*}
   (\mathfrak{A} - I) x
    &= f
    \\
   \mathfrak{B} x
    &= - D_c \mathfrak{C} x
    =: K \mathfrak{C} x.
  \end{align*}
 We can identify the operator $\hat A: \dom(\hat A) \subseteq X \times \{0\} \rightarrow X \times \{0\}$ with the operator $A = \mathfrak{A}|_{\ker (\mathfrak{B} - K \mathfrak{C} x)}: \dom(A) = \ker (\mathfrak{B} - K \mathfrak{C}) \subseteq X \rightarrow X$.
 For every $j \in \mathcal{J}$ we now write
  \[
   h^j
    = (x^j, (x^j)', \ldots, (x^j)^{(N_j - 1)}),
    \quad
   g^j
    = (0, \ldots, 0, (P^j_{N_j})^{-1} f^j),
    \quad
    j \in \mathcal{J}.
  \]
 Then
  \begin{align*}
   &(\mathfrak{A} - I) x
    = f
    \\
    \quad &\Leftrightarrow \quad
   (\mathfrak{A}^j - I) x^j
    = f^j,
    \quad
    j \in \mathcal{J}
    \\
   &\Leftrightarrow \quad
   \sum_{k=0}^{N_j} P^j_k (x^j)^{(k)}(\z) - x^j(\z) = f^j(\z),
   \quad
   \text{a.e.\ } \z \in (0,1), \, j \in \mathcal{J}
   \\
   &\Leftrightarrow \quad
   (x^j)^{(N_j)}(\z)
    = (P^j_{N_j})^{-1} \left( x^j(\z) - \sum_{k=0}^{N_j-1} P^j_k (x^j)^{(k)}(\z) + f^j(\z) \right),
   \quad
   \text{a.e.\ } \z \in (0,1), \, j \in \mathcal{J}
   \\
   &\Leftrightarrow \quad
   (h^j)'(\z)
    = L^j h^j(\z) + g^j(\z),
   \quad
   \text{a.e.\ } \z \in (0,1), \, j \in \mathcal{J}
   \\
   &\Leftrightarrow \quad
   h^j(\z)
    = \ee^{\z L^j} h^j(0) + \int_0^\z \ee^{(\z - s) L^j} g^j(s) \dd s,
   \quad
   \text{a.e.\ } \z \in (0,1), \, j \in \mathcal{J}
  \end{align*}
 where
  \[
   L^j
    = \left[ \begin{array}{ccccc} 0 & 1 &&& 0 \\  \vdots && \ddots && \vdots \\ 0 &&& 1 & 0 \\ (P^j_{N_j})^{-1} - (P^j_{N_j})^{-1} P^j_0 & - (P^j_{N_j})^{-1} P^j_1 & \cdots & \cdots & - (P^j_{N_j})^{-1} P^j_{N_j - 1} \end{array} \right]
    \in \K^{N^j d^j \times N^j d^j}.
  \]
 In that case, we have, writing $D_c = (D_c^{ij})_{i,j \in \mathcal{J}}$ for $D_c^{ij} \in \B(Y^j; U^i)$, that $x \in \dom(A)$ if and only if
  \begin{align*}
   &W^j_B \left[ \begin{array}{c} \ee^{L^j} \\ I \end{array} \right] h^j(0)
    + W^j_B \left[ \begin{array}{c} \int_0^1 \ee^{(1-s)L^j} g^j(s) \dd s \\ 0 \end{array} \right]
    \\
    &= W_B^j \tau^j(\H^j x^j)
    = - \sum_{i \in \mathcal{J}} D_c^{ji} W_C^i \tau^i(\H^i x^i)
    \\
    &= - \sum_{i \in \mathcal{J}} D_c^{ji} \left( W^i_C \left[ \begin{array}{c} \ee^{L^i} \\ I \end{array} \right] h^i(0)
    + W^i_C \left[ \begin{array}{c} \int_0^1 \ee^{(1-s)L^i} g^i(s) \dd s \\ 0 \end{array} \right] \right),
    \quad
    j \in \mathcal{J}.
  \end{align*}
 Letting
  \begin{align*}
   \hat \xi
    &:= (\hat \xi^j)_{j \in \mathcal{J}},
    \quad
   \hat \xi^j
    := \left[ \begin{array}{c} \int_0^1 \ee^{(1-s)L^j} g^j(s) \dd s \\ 0 \end{array} \right]
    \\
   \hat h
    &:= (\hat h^j)_{j \in \mathcal{J}},
    \quad
   \hat h^j
    := \left[ \begin{array}{c} \ee^{L^j} \\ I \end{array} \right] h^j(0),
    \quad
    j \in \mathcal{J}
  \end{align*}
 and defining
  \[
   W_B
    := \operatorname{diag}\, ( W_B^j )_{j \in \mathcal{J}},
    \quad
   W_C
    := \operatorname{diag}\, ( W_C^j )_{j \in \mathcal{J}}
    \in \B(U^2; U),
    \quad
   T
    := \operatorname{diag} \left( \left[ \begin{array}{c} \ee^{M_j} \\ I \end{array} \right] \right)_{j \in \mathcal{J}}
    \in \B(U; U^2)
  \]
 this equation reads as
  \[
   (W_B + D_c W_C) (T \hat h - \hat \xi)
    = 0
  \]
 where $\xi$ is determined by $f$.
 We are done after showing that $(W_B + D_c W_C) T \in \B(U)$ is invertible.
 Namely, then the unique solution $x \in \dom(\hat A)$ is given by
  \[
   x^j = h^j_1,
    \quad
    h^j(0) = \hat h^j,
    \quad
    \hat h = ((W_B - K W_C) T)^{-1} (W_B + D_c W_C) \hat \xi.
  \]
 So, let us show that $(W_B + D_c W_C) T \in \B(U)$ is invertible. Since $U$ is finite dimensional it suffices to show that $(W_B + D_c W_C) T$ is injective.
 \emph{Assume} there were $\hat h \in U \setminus \{0\}$ such that
  \[
   (W_B + D_c W_C) T \hat h
    = 0
  \]
 Then $h^j(0) = \hat h^j$, $j \in \mathcal{J}$, are well-defined and for $f = 0$ the problem $(I - \mathfrak{A}) x = 0$ has a solution $x = (x^j)_{j \in \mathcal{J}} := (h^j)_{j \in \mathcal{J}} \in \dom(\mathfrak{A})$ for which we also have
  \[
   \mathfrak{B} x - K \mathfrak{C} x
    = (W_B + D_c W_C) T \hat h
    = 0,
  \]
 i.e.\ $x \in \dom(A)$ with $A x = x$, a contradiction to $A$ being dissipative, so $1 \not\in \sigma(A)$.
 This concludes the proof for the case $P_0 = 0$.
 \newline
 For the case of general $P_0 \not= 0$, note that $P_0 \H$ is a bounded perturbation of $\hat A - P_0 \H$, hence $\hat A - P_0 \H$ generates a $C_0$-semigroup if and only if $\hat A$ generates a $C_0$-semigroup.
 The proof is then completed by the following small observation.
 \end{proof}
 
 \begin{lemma}
 \label{lem:P_0}
  The operator $\hat A$ is dissipative on $\hat X$ if and only if the operator $\hat A'$ where the $P_0^j$ are replaced by constant zero matrices is dissipative and additionally for all $j \in \mathcal{J}$ one has
   \[
    \Re \sp{P_0^j(\z) \xi^j}{\xi^j}_{\K^{d^j}} \leq 0,
     \quad
     \text{a.e.\ } \z \in (0,1), \, \text{all } \xi^j \in \K^{d^j}, \, j \in \mathcal{J},
   \]
  i.e.\ the symmetric parts $\Sym P_0^j(\z) = \frac{P_0^j(\z) + P_0^j(\z)^*}{2} \leq 0$ are negative semi-definite for a.e.\ $\z \in (0,1)$.
 \end{lemma}
 
 \begin{proof}
 Use the same strategy as in the proof of Theorem 2.3 and Lemma 2.4 in \cite{AugnerJacob_2014} for every $j \in \mathcal{J}$.
 \end{proof}
 
 Having the generation result for static feedback interconnection at hand, we are able to prove the generation result for dynamic feedback interconnection via a finite dimensional linear control system as well. 
 
 \begin{theorem}
 \label{thm:generation_network-dynamic}
  $\hat A$ generates a contractive $C_0$-semigroup on $\hat X$ if and only if $\hat A$ is dissipative.
  In that case, $\hat A$ has compact resolvent.
 \end{theorem}
  \begin{proof}
  Clearly, by the Lumer-Phillips Theorem, $\hat A$ is necessarily dissipative if it generates a strongly continuous contraction semigroup $(T(t))_{t \geq 0}$.
  Therefore, we only have to show that this condition is (just as for single port-Hamiltonian systems with static or dynamic boundary feedback) even sufficient.
  Let us further note that we can restrict ourselves to the case $\H^j = I$ for all $j \in \mathcal{J}$, see e.g.\ Lemma 7.2.3 in \cite{JacobZwart_2012}, and $P_0 = 0$, cf.\ Lemma \ref{lem:P_0}.
  By the Lumer-Phillips Theorem, we have to show that $\operatorname{ran}\, (\lambda - \hat A) = \hat X$ for some $\lambda > 0$ and that $\hat A$ is densely defined.
  First, we show that $\hat A$ is densely defined.
  Take any $(x,x_c) \in \hat X = X \times X_c$ and $\varepsilon > 0$.
  Then, the condition
   \[
    \mathfrak{B} x = - (C_c x_c + D_c \mathfrak{C} x)
   \]
  is equivalent to the condition
   \[
    \mathfrak{B} x + D_c \mathfrak{C} x
     = - C_c x_c
     =: w \in U
   \]
  The left hand side can be written as
   \[
    \mathfrak{B} x + D_c \mathfrak{C} x
    = \left[ \begin{array}{cc} I & D_c \end{array} \right] \left[ \begin{array}{c} W_B \\ W_C \end{array} \right] \tau(\H x)
   \]
  where we used the notation $W_B = \operatorname{diag}_{j \in \mathcal{J}}\, \{ W^j_B \} \in \B(U \times Y; U)$, $W_C = \operatorname{diag}_{j \in \mathcal{J}}\, \{ W^j_C \} \in \B(U \times Y; Y)$ and $\tau(\H x) = (\tau^j(\H_j x^j))_{j \in \mathcal{J}} \in U^2 = U \times Y = Y^2$.
  By the definition of a port-Hamiltonian system, the matrix $\left[ \begin{smallmatrix} W_B \\ W_C \end{smallmatrix} \right]$ is invertible as it is similar to the block-diagonal matrix $\operatorname{diag}\, \left( \left[ \begin{smallmatrix} W_B^j \\ W_C^j \end{smallmatrix} \right] \right)_{j \in \mathcal{J}}$.
  Moreover, the matrix $\left[ \begin{array}{cc} I & D_c \end{array} \right] \in \B(U \times Y; U)$ has full rank, in particular $\left[ \begin{array}{cc} I & D_c \end{array} \right] \left[ \begin{smallmatrix} W_B \\ W_C \end{smallmatrix} \right]$ is surjective, i.e.\ there is $v \in U \times Y$ such that
   \[
    \left[ \begin{array}{cc} I & D_c \end{array} \right] \left[ \begin{array}{c} W_B \\ W_C \end{array} \right] v
    = w.
   \]
  One then finds $x_0 \in \dom(\mathfrak{A})$ such that $\tau (\H x_0) = v$, hence $(x_0, x_c) \in \dom(\hat A)$.
  Since $\prod_{j \in \mathcal{J}} C_c^\infty(0,1;\K^{d_j})$ is dense is $X$, there is $x_1 \in \prod_{j \in \mathcal{J}} C_c^\infty(0,1;\K^{d_j})$ such that $\| x_1 - (x - x_0) \|_X = \| (x_0 + x_1) - x \|_X \leq \varepsilon$, i.e.\ we find that $\hat x_2 := (x_0 + x_1, x_c) \in \dom(\hat A)$ with $\| (x,x_c) - \hat x_2 \|_{\hat X} \leq \varepsilon$, i.e.\ $\dom(\hat A)$ is dense in $\hat X$.\newline
  It remains to show that $\operatorname{ran}\, (\lambda - \hat A) = \hat X$ for some $\lambda > 0$.
  Here, we take $\lambda > 0$ large enough such that $\lambda \in \rho(A_c)$, i.e.\ $(\lambda - A_c)^{-1} \in \B(X_c)$ exists.
  (Note that $X_c$ is finite dimensional, hence such a choice is always possible.)
  Take $(f,f_c) \in \hat X$.
  We need to find $(x, x_c) \in \dom(\hat A)$ such that $(\lambda I - \hat A)(x,x_c) = (f,f_c)$, i.e.\ $(x, x_c) \in \dom(\mathfrak{A}) \times X_c$ such that $(\lambda I - \mathfrak{A}) x = f$, $(\lambda - A_c) x_c - B_c \mathfrak{C} x = f_c$ and
   \[
    \mathfrak{B} x
     = - (C_c x_c + D_c \mathfrak{C} x)
   \]
  Since $\lambda \in \rho(A_c)$, this means that in particular $x_c \in X_c$ is given by
   \[
    x_c
     = (\lambda - A_c)^{-1} (f_c + B_c \mathfrak{C} x)
   \]
  and the interconnection condition then reads
   \begin{align*}
    &\mathfrak{B} x
     = - C_c (\lambda - A_c)^{-1} f_c - (C_c (\lambda - A_c)^{-1} B_c + D_c) \mathfrak{C} x
     \\
    &\Leftrightarrow \quad
    \mathfrak{B}_{\mathrm{cl}} x
     := \mathfrak{B} x + (C_c (\lambda - A_c)^{-1} B_c + D_c) \mathfrak{C} x
     = - C_c (\lambda - A_c)^{-1} f_c
     =: \tilde f_c.
   \end{align*}
  Just as in the single port-Hamiltonian system case, the boundary operator $\mathfrak{B}_{\mathrm{cl}} \in \B(\dom(\mathfrak{A}); U)$ has a right-inverse $B_{\mathrm{cl}} \in \B(U; \dom(\mathfrak{A}))$, so we may set
   \[
    x_{\mathrm{new}}
     := x - B_{\mathrm{cl}} \tilde f_c
   \]
  which, therefore, has to be a solution of the problem
   \begin{align*}
    (\lambda I - \mathfrak{A}) x_{\mathrm{new}}
     &= f - (\lambda I - \mathfrak{A}) B_{\mathrm{cl}} \tilde f_c
     =: \tilde f
     \\
    \mathfrak{B}_{\mathrm{cl}} x_{\mathrm{new}}
     &= \mathfrak{B}_{\mathrm{cl}} x - \mathfrak{B}_{\mathrm{cl}} B_{\mathrm{cl}} \tilde f_c
     = 0     
   \end{align*}
  To show that this problem has a (unique) solution, we show that the operator $A_{\mathrm{cl}} := \mathfrak{A}_{\mathrm{cl}}|_{\ker \mathfrak{B}_{\mathrm{cl}}}$ is dissipative and hence generates a strongly continuous contraction semigroup on $\hat X$, in particular, $x_{\mathrm{new}} = (\lambda - A_{\mathrm{cl}})^{-1} \tilde f$.
  In fact, for any $x \in \dom(A_{\mathrm{cl}})$, set $x_c := (\lambda - A_c)^{-1} B_c \mathfrak{C} x \in X_c$.
  Then
   \[
    \mathfrak{B} x + C_c x_c + D_c \mathfrak{C} x
     = \mathfrak{B} x + (C_c (\lambda - A_c)^{-1} B_c + D_c) \mathfrak{C} x
     = \mathfrak{B}_{\mathrm{cl}} x
     = 0
   \]
  so that $(x, x_c) \in \dom(\hat A)$ and hence
   \begin{align*}
    &\Re \sp{A_{\mathrm{cl}} x}{x}_X
     \\
     &= \Re \sp{\mathfrak{A} x}{x}_X
     = \Re \sp{\hat A(x,x_c)}{(x,x_c)}_{\hat X}
      - \Re \sp{A_c x_c + B_c \mathfrak{C} x}{x_c}_{X_c}
      \\
     &\leq - \Re \sp{A_c x_c + B_c \mathfrak{C} x}{x_c}_{X_c}
     \\
     &= - \Re \sp{A_c (\lambda - A_c)^{-1} B_c \mathfrak{C} x + (\lambda - A_c)(\lambda - A_c)^{-1} B_c \mathfrak{C} x}{(\lambda - A_c)^{-1} B_c \mathfrak{C} x}_{X_c}
     \\
     &= - \lambda \| (\lambda - A_c)^{-1} B_c \mathfrak{C} x \|_{X_c}^2
     \leq 0.
   \end{align*}
  This shows that $A_{\mathrm{cl}}$ is dissipative and by Proposition \ref{prop:gen_static_feedback} above $A_{\mathrm{cl}}$ generates a strongly continuous contraction semigroup on $X$, in particular $(0, \infty) \subseteq \rho(A_{\mathrm{cl}})$ and hence $x_{\mathrm{new}} = (\lambda - A_{\mathrm{cl}})^{-1} \tilde f$.
  Putting everything together, we obtain the desired $(x, x_c) \in \dom(\hat A)$ by solving the problem $(\lambda - \hat A) (x,x_c) = (f,f_c)$ as
   \begin{align*}
    x
     &= x_{\mathrm{new}} + B_{\mathrm{cl}} \tilde f_c
     = (\lambda - A_{\mathrm{cl}})^{-1} \tilde f + B_{\mathrm{cl}} \tilde f_c
     \\
     &= (\lambda - A_{\mathrm{cl}})^{-1} (f - (\lambda - \mathfrak{A}) B_{\mathrm{cl}} \tilde f_c) + B_{\mathrm{cl}} \tilde f_c,
     \\
    x_c
     &= (\lambda - A_c)^{-1} (f_c + B_c \mathfrak{C} x)
     \\
     &=  (\lambda - A_c)^{-1} (f_c + B_c \mathfrak{C} ((\lambda - A_{\mathrm{cl}})^{-1} (f - (\lambda - \mathfrak{A}) B_{\mathrm{cl}} \tilde f_c) + B_{\mathrm{cl}} \tilde f_c)).
   \end{align*}
  The operator $\lambda - \hat A$ therefore is surjective and the Lumer-Phillips Theorem provides the characterisation of the generator property.
  The compactness of the resolvent follows for generators $\hat A$ since $\dom(\hat A) \subset \dom(\mathfrak{A}) \times X_c$, where $\dom(\mathfrak{A}) = \prod_{j \in \mathcal{J}} \dom(\mathfrak{A}_j)$ is relatively compact as a product of relatively compact (in $X_j$) spaces $\dom(\mathfrak{A}_j)$ (by the Rellich-Kondrachev theorem, see e.g.\ Theorem 8.9 in \cite{LiebLoss_2001}; all spaces shall be equipped with their respective graph norms), and $X_c$ is finite dimensional, so compactly embedded into itself.
  \end{proof}
 
 Similar to the case of Dirac structures, where an interconnection of Dirac structures is a Dirac structure again, the interconnection of port-Hamiltonian systems in boundary control and observation form defines a boundary control and observation system.
 
 \begin{definition}
  \label{def:hybrid_PH-ODE}
  For a system as above consisting of a family of port-Hamiltonian systems $\mathfrak{S}^j$ and finite dimensional control system $\Sigma_c^j$, we may also introduce external inputs and outputs by setting
   \[
    \hat u = \hat {\mathfrak{B}} \hat x = \mathfrak{B} x + C_c x_c + D_c \mathfrak{C} x \in U,
     \quad
    \hat y = \hat {\mathfrak{C}} \hat x = \mathfrak{C} x \in Y,
     \quad
     (x, x_c) \in \dom(\mathfrak{A}) \times X_c,
   \]
  see Figure \ref{fig:diagram_network_phs}.
  \begin{figure}
	\centering
	\includegraphics[scale = 1.5]{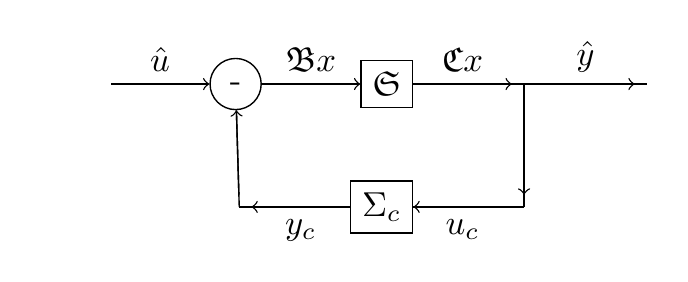}
	\caption{A system of port-Hamiltonian type $\mathfrak{S}$ coupled with a finite-dimensional controller $\Sigma_c$ and external input $\hat u$ and output $\hat y$.}
	\label{fig:diagram_network_phs}
 \end{figure}
 Moreover, we define the triple $\hat{\mathfrak{S}} = (\hat {\mathfrak{A}}, \hat {\mathfrak{B}}, \hat {\mathfrak{C}})$ with
  \begin{align*}
   \hat {\mathfrak{A}} (x, x_c)
    &= \left[ \begin{array}{cc} \mathfrak{A} & 0 \\ B_c \mathfrak{C} & A_c \end{array} \right] (x, x_c)
    \\
   \dom(\hat {\mathfrak{A}})
    &= \dom(\mathfrak{A}) \times X_c.
  \end{align*}
 In the following we call $\hat {\mathfrak{S}}$ an (open-loop) \emph{hybrid port-Hamiltonian system}.
 Note that $\hat A = \hat {\mathfrak{A}}|_{\ker \hat {\mathfrak{B}}}$.
  More generally, we also call $\hat {\mathfrak{S}} = (\hat {\mathfrak{A}}, \hat {\mathfrak{B}}, \hat {\mathfrak{C}})$ an (open-loop) \emph{hybrid port-Hamiltonian system}, if
   \[
    \left( \begin{array}{c} \hat {\mathfrak{B}} \\ \hat {\mathfrak{C}} \end{array} \right) \hat x
     = \hat W \left( \begin{array}{c} \mathfrak{B} x + C_c x_c + D_c \mathfrak{C} x \\ \mathfrak{C} x \end{array} \right),
     \quad
     \hat x \in \dom(\hat {\mathfrak{A}})
   \]
  for some invertible matrix $\hat W \in \B(U \times Y)$.
 \end{definition}

  These input and output maps $\hat {\mathfrak{B}}$ and $\hat {\mathfrak{C}}$ may then be used to interconnect several of such hybrid PDE-ODE systems $\mathfrak{S}$ with each other.
  As each of such systems consists of infinite-dimensional port-Hamiltonian systems on an interval and finite dimensional control systems, the interconnection of such hybrid systems then again generates a contractive $C_0$-semigroup if and only if the interconnection makes the total system dissipative.
 Therefore, with respect to well-posedness such a point of view does not give more information than just considering the system of these hybrid PH-ODE systems as one large hybrid PH-ODE system. In the next section, however, we exploit structural conditions on the arrangement of such a system to deduce better stability results, i.e.\ stability under less restrictive conditions.
 
 \begin{remark}
  If one chooses $\hat W = I$ in the above definition of an open-loop hybrid PH-ODE system, and additionally all port-Hamiltonian systems $\mathfrak{S}^j = (\mathfrak{A}^j, \mathfrak{B}^j, \mathfrak{C}^j)$ and the linear controller $\Sigma_c = (A_c, B_c, C_c, D_c)$ are impedance passive, then the triple $\mathfrak{S} = (\hat {\mathfrak{A}}, \hat {\mathfrak{B}}, \hat {\mathfrak{C}})$ is impedance passive as well, since for all $\hat x \in \dom(\hat {\mathfrak{A}})$ one has
   \begin{align*}
    \Re \sp{\hat {\mathfrak{A}} \hat x}{\hat x}_{\hat X}
     &= \Re \sp{\mathfrak{A} x}{x}_X + \Re \sp{A_c x_c + B_c \mathfrak{C} x}{x_c}_{X_c}
     \\
     &\leq \Re \sp{\mathfrak{B} x}{\mathfrak{C} x}_{U}
      + \Re \sp{C_c x_c + D_c \mathfrak{C} x}{\mathfrak{C} x}_{U_c}
      \\
     &= \Re \sp{(\mathfrak{B} + D_c \mathfrak{C}) x + C_c x_c}{\mathfrak{C} x}_U
     = \Re \sp{\hat {\mathfrak{B}} \hat x}{\hat {\mathfrak{C}} \hat x}_{U}.
   \end{align*}
 \end{remark}

\section{Stability Properties of Hybrid Multi-PHS-control systems}
\label{sec:hybrid_multi_phs}

Let us take the operator $\hat A$ from the previous section, i.e.\
 \begin{align*}
  \hat A \hat x
   &= (\mathfrak{A} x, A_c x_c + B_c \mathfrak{C} x)
   = ((\mathfrak{A}^j x^j)_{j \in \mathcal{J}}, (A_c^j x_c^j + B_c^j \mathfrak{C} x)_{j \in \mathcal{J}_c}),
   \\
  \dom(\hat A)
   &= \{ \hat x = (x, x_c) \in \dom(\mathfrak{A}) \times X_c: \quad \mathfrak{B} x = - (C_c x_c + D_c \mathfrak{C} x) \}
 \end{align*}
 in particular we assume $U_c = Y_c = U = Y$ and $E_c = I$, $E = I$.
 Stability, as for single port-Hamiltonian operators, is much more involved than the generation property.
  \begin{proposition}
   \label{prop:strict_dissipation_N_j=1}
   Let $\hat A$ be as in Theorem \ref{thm:generation_network-dynamic} with port-Hamiltonian order $N^j = 1$ for all $j \in \mathcal{J}$ and assume that the Hamiltonian density matrix functions $\H^j: [0,1] \rightarrow \K^{d^j \times d^j}$ are Lipschitz continuous for all $j \in \mathcal{J}$. If
    \[
     \Re \sp{\hat A \hat x}{\hat x}_{\hat X}
      \lesssim - \sum_{j=1}^m \abs{(\H^j x^j)(0)}^2,
      \quad
      \hat x \in \dom(\hat A)
    \]
   and $\sigma_p(A_c) \subseteq \C_0^-$, then the $C_0$-semigroup $(\hat T(t))_{t \geq 0}$ generated by $\hat A$ is uniformly exponentially stable.
  \end{proposition}
 
 \begin{proof}
 This result already follows from Corollary 3.10 in \cite{AugnerJacob_2014}.
 \end{proof}
 
 Note that the condition imposed in Proposition \ref{prop:strict_dissipation_N_j=1} on the interconnection is by far too restrictive for complex systems consisting of several subsystems of infinite-dimensional port-Hamiltonian type and finite-dimensional control systems: All port-Hamiltonian subsystems have to be interconnected in a way that they dissipate energy at the boundary, and all control systems have to be internally stable.
  The result does in no way require any special structure for the interconnection of the port-Hamiltonian systems, whereas for systems which interconnection structure forms a special class of graphs much less restrictive condition on the dissipative terms can be expected.

In the following, we restrict ourselves to impedance passive port-Hamiltonian systems and strictly input passive control systems as follows.
 \begin{assumption}
  \label{assmpt:hybrid}
  We assume that the following hold.
   \begin{enumerate}
    \item
     $\mathfrak{S} = (\mathfrak{A}, \mathfrak{B}, \mathfrak{C})$ is \emph{impedance passive}, i.e.\ 
      \[
       \Re \sp{\mathfrak{A} x}{x}_X
        \leq \Re \sp{\mathfrak{B} x}{\mathfrak{C} x}_U
         - \norm{\mathfrak{R} x}_Z^2,
         \quad
         x \in \dom(\mathfrak{A})
      \]
     for some linear operator $\mathfrak{R}: \dom(\mathfrak{R}) = \dom(\mathfrak{A}) \subseteq X \rightarrow Z$ and some Hilbert space $Z$,
    \item
     $\Sigma_c = (A_c, B_c, C_c, D_c)$ is \emph{strictly input passive}, more precisely, there is an orthogonal projection $\Pi: U_c \rightarrow U_c$ such that
      \[
       \ker \Pi
        = \ker D_c
        \subseteq \ker B_c
      \]
     and for some $\kappa > 0$
      \[
       \Re \sp{A_c x_c + B_c u_c}{x_c}_{X_c}
        \leq \Re \sp{C_c x_c + D_c u_c}{u_c}_{U_c}
         - \kappa \abs{\Pi u_c}_{U_c}^2,
         \quad
         x_c \in X_c, \, u_c \in U_c,
      \]
    \item
     $\sigma_p(A_c) \subseteq \C_0^-$, i.e.\ $(\ee^{t A_c})_{t \geq 0}$ is uniformly exponentially stable on the finite dimensional space $X_c$, and
    \item
     there are linear operators $\mathfrak{R}^j: \dom(\mathfrak{R}^j) = \dom(\mathfrak{A}^j) \subseteq X^j \rightarrow Z^j$ (for some Hilbert spaces $Z^j$), $j \in \mathcal{J}$, such that
      \[
       \norm{\mathfrak{R} x}_Z^2 + \abs{\Pi \mathfrak{C} x}_U^2 + \abs{\mathfrak{B} x}_Y^2
        \geq \sum_{j \in \mathcal{J}} \norm{\mathfrak{R}^j x^j}_{Z^j}^2,
        \quad
        x \in \dom(\mathfrak{A}).
      \]
   \end{enumerate}
 \end{assumption}

 \begin{remark}
  Note that, as a consequence of Assumption \ref{assmpt:hybrid},
   \[
    \Re \sp{\hat A \hat x}{\hat x}_{\hat X}
     \leq - \norm{\mathfrak{R} x}_Z^2 - \kappa \abs{\Pi \mathfrak{C} x}_Y^2,
     \quad
     \hat x \in \dom(\hat A).
   \]
  Moreover, $\ker D_c \subseteq \ker C_c^*$.
 \end{remark}
 \begin{proof}
  The first assertion directly follows from impedance passivity and standard feedback interconnection.
  Let us shhow that $\ker D_c \subseteq \ker C_c^*$.
  Take $u_c \in \ker D_c \subseteq \ker B_c$.
  Then, from the impedance passivity of $\Sigma_c$, we have for all $x_c \in X_c$ that
   \begin{align*}
    \Re \sp{A_c x_c}{x_c}_{X_c}
     &= \sp{A_c x_c + B_c u_c}{x_c}_{X_c}
     \\
     &\leq \Re \sp{C_c x_c + D_c u_c}{u_c}_{U_c}
     = \Re \sp{x_c}{C_c^* u_c}_{U_c}.
   \end{align*}
  Since this inequality holds for all $x_c \in X_c$, we deduce that $C_c^* u_c \in X_c^\bot = \{0\}$.
 \end{proof}
 
 To relate stability properties of the interconnected system, i.e.\ the $C_0$-semigroup $(\hat T(t))_{t \geq 0}$ with structural and damping properties of the involved port-Hamiltonian subsystems, let us introduce the following notions:
 properties ASP and AIEP (which have already been used in the research article \cite{AugnerJacob_2014}), as well as property $\mathrm{AIEP}_S$ (which is a slight modification of property AIEP).

 \begin{definition}
  Let $B: \dom(B) \subseteq H_1 \rightarrow H_1$ be a closed linear operator and $R \in \B(\dom(B); H_2)$, $S \in \B(\dom(B); H_3)$ for Hilbert spaces $H_1$, $H_2$ and $H_3$, and where $\dom(B)$ is equipped with its graph norm. We then say that the pair $(B, R)$ has property
   \begin{enumerate}
    \item
     \ldots ASP, if $\ker (\ii \beta - B) \cap \ker R = \{0\}$ for all $\beta \in \R$, i.e.\
      \[
       \ii \beta x = B x \quad \text{and} \quad R x = 0
        \quad \Rightarrow \quad
        x = 0.
      \]
    \item
     \ldots AIEP, if for all sequences $(x_n, \beta_n)_{n \geq 1} \subseteq \dom(B) \times \R$ with $\sup_{n \geq 1} \norm{x_n} < \infty$ and $\abs{\beta_n} \rightarrow \infty$,
      \[
       \ii \beta_n x_n - B x_n \rightarrow 0 \quad \text{and} \quad R x_n \rightarrow 0
        \quad \Rightarrow \quad
        x_n \rightarrow 0 \quad \text{in } H_1.
      \]
    \item
     \ldots $\mathrm{AIEP}_S$, if for all sequences $(x_n, \beta_n)_{n \geq 1} \subseteq \dom(B) \times \R$ with $\sup_{n \geq 1} \norm{x_n} < \infty$ and $\abs{\beta_n} \rightarrow \infty$,
      \[
       \ii \beta_n x_n - B x_n \rightarrow 0 \quad \text{and} \quad R x_n \rightarrow 0
        \quad \Rightarrow \quad
        x_n \rightarrow 0 \quad \text{in } H_1
        \quad \text{and} \quad
        S x_n \rightarrow 0 \quad \text{in } H_3.
      \]
   \end{enumerate}
 \end{definition} 

 With these abstract notions at hand, we can formulate the following stability results. 
 
 \begin{theorem}[Stability properties]
  \label{thm:stability}
  Assume that $\hat A$ satisfies Assumption \ref{assmpt:hybrid}.
   \begin{enumerate}
    \item
     If all pairs $(\mathfrak{A}^j, \mathfrak{R}^j)$, $j \in \mathcal{J}$, have property ASP, then the $C_0$-semigroup $(\hat T(t))_{t \geq 0}$ generated by $\hat A$ is (asymptotically) strongly stable.
    \item
     If $(\hat T(t))_{t \geq 0}$ is asymptotically stable and all pairs $(\mathfrak{A}^j, \mathfrak{R}^j)$ have property AIEP, then $(\hat T(t))_{t \geq 0}$ is uniformly exponentially stable.
    \item
     If all pairs $(\mathfrak{A}^j, \mathfrak{R}^j)$ have property $\mathrm{AIEP}_{\tau^j \circ \H^j}$, then the pair
      \[
       \left( \left[ \begin{smallmatrix} \mathfrak{A} & 0 \\ B_c \mathfrak{C} & A_c \end{smallmatrix} \right], \mathfrak{B} x + D_c \mathfrak{C} x + C_c x_c \right)
      \]
     has property $\mathrm{AIEP}_{\tau \circ \H}$ as well, where $\tau (\H x) = (\tau^j (\H_j x^j))_{j \in \mathcal{J}})$.
   \end{enumerate}
 \end{theorem}

 \begin{proof}
 \begin{enumerate}
  \item
   We show strong stability by demonstrating that $\sigma_p(\hat A) \subseteq \C_0^-$, which by the Arendt-Batty-Lyubich-V\~u Theorem is enough for strong stability as $\hat A$ has compact resolvent.
   Clearly, since $\hat A$ is dissipative, we have $\sigma(\hat A) \subseteq \overline{\C_0^-}$, i.e.\ we only need to check that no $\ii \beta \in \ii \R$ is an eigenvalue of $\hat A$.
   Thus, let $\hat x = (x, x_c) \in \dom(\hat A)$ be such that $\hat A \hat x = \ii \beta \hat x$ for some $\beta \in \R$.
   Then, in particular
    \[
     \left( \begin{array}{c} \mathfrak{A} x \\ A_c x_c + B_c \mathfrak{C} x \end{array} \right)
      = \left( \begin{array}{c} \ii \beta x \\ \ii \beta x_c \end{array} \right)
      \quad \Rightarrow \quad
     \left( \begin{array}{c} (\mathfrak{A} - \ii \beta) x \\ x_c \end{array} \right)
      = \left( \begin{array}{c} 0 \\ (\ii \beta - A_c)^{-1} B_c \mathfrak{C} x \end{array} \right)
    \]
   (note that $\ii \R \subseteq \rho(A_c)$ by Assumption \ref{assmpt:hybrid}).
   Since $\hat x \in \dom(\hat A)$, we then have
    \[
     \mathfrak{B} x
      = - (C_c x_c + D_c \mathfrak{C} x)
      = - [ C_c (\ii \beta - A_c)^{-1} B_c + D_c] \mathfrak{C} x
    \]
   and from impedance passivity of $\mathfrak{S}$ and $\Sigma_c$, we obtain
    \begin{align*}
     0
      &= \Re \sp{\ii \beta x}{x}_X
      = \Re \sp{\mathfrak{A} x}{x}_X
      \\
      &\leq \Re \sp{\mathfrak{B} x}{\mathfrak{C} x}_{U}
      = - \Re \sp{(C_c (\ii \beta - A_c)^{-1} B_c + D_c) \mathfrak{C} x}{\mathfrak{C} x}_{Y}
      \\
      &\leq - \Re \sp{A_c (\ii \beta - A_c)^{-1} B_c \mathfrak{C} x + B_c \mathfrak{C} x}{(\ii \beta - A_c)^{-1} B_c \mathfrak{C} x}_{X_c} - \kappa \abs{\Pi \mathfrak{C} x}^2
      \\
      &= - \Re \sp{\ii \beta (\ii \beta - A_c)^{-1} B_c \mathfrak{C} x}{(\ii \beta - A_c)^{-1} B_c \mathfrak{C} x}_{X_c}  - \kappa \abs{\Pi \mathfrak{C} x}^2
      =  - \kappa \abs{\Pi \mathfrak{C} x}^2
      \leq 0.
    \end{align*}
   This chain of inequalities shows that $\Pi \mathfrak{C} x = 0$, hence $B_c \mathfrak{C} x = 0$ due to $\ker D_c \subseteq \ker B_c$, and then $x_c = (\ii \beta - A_c)^{-1} B_c \mathfrak{C} x = 0$ so that
    \[
     \mathfrak{B} x
      = - [ C_c (\ii \beta - A_c)^{-1} B_c + D_c] \mathfrak{C} x
      = 0.
    \]
   Moreover,
    \[
     0
      = \Re \sp{\ii \beta \hat x}{\hat x}_{\hat X}
      = \Re \sp{\hat A \hat x}{\hat x}_{\hat X}
      \leq - \norm{\mathfrak{R} x}_Z^2 - \kappa \abs{\Pi \mathfrak{C} x}^2
      \leq 0,
    \]
   so that $\mathfrak{R} x = 0$, $\mathfrak{B} x = 0$ and $\Pi \mathfrak{C} x = 0$, in particular $\mathfrak{R}^j x^j = 0$ for all $j \in \mathcal{J}$, and by property ASP of the pairs $(\mathfrak{A}^j, \mathfrak{R}^j)$ this implies that $x^j = 0$ for all $j \in \mathcal{J}$, but then $\mathfrak{C} x = 0$ as well as $x_c = 0$, i.e.\ $\hat x = 0$ and $\sigma_p(\hat A) \cap \ii \R = \emptyset$.
   Strong stability follows.
  \item
   For uniform exponential stability, we use the Gearhart-Pr\"uss-Huang Theorem, i.e.\ we show that $\sup_{\beta \in \R} \norm{(\ii \beta - \hat A)^{-1}}_{\B(\hat X)} < \infty$.
   By remark \ref{rem:GPH}, this property is equivalent to showing that for every sequence $(\hat x_n, \beta_n)_{n \geq 1} \subseteq \dom(\hat A) \times \R$ with $\sup_{n \in \N} \norm{\hat x_n}_{\hat X} < \infty$ and $\abs{\beta_n} \rightarrow \infty$ and $\hat A \hat x_n - \ii \beta_n \hat x_n$, we have $\hat x_n \rightarrow 0$ in $\hat X$.
   In view of the third assertion, we even show a little bit more, namely
    \[
     \left.
      \begin{array}{l}
      (\hat x_n)_{n \geq 1} \subseteq \dom(\mathfrak{A}) \times X_c, \, \sup_{n \in \N} \norm{\hat x_n}_{\hat X} < \infty \\
      (\beta_n)_{n \geq 1} \subseteq \R, \, \abs{\beta_n} \rightarrow \infty \\
      (\ii \beta_n - \mathfrak{A}) x_n \rightarrow 0 \, \text { in } X \\
      (\ii \beta_n - A_c) x_{c,n} - B_c \mathfrak{C} x_n \rightarrow 0 \, \text{ in } X_c \\
      \mathfrak{B} x_n + C_c x_{c,n} + D_c \mathfrak{C} x_n \rightarrow 0 \, \text{ in } \ran \left[ \begin{array}{cc} C_c & D_c \end{array} \right] \subseteq U
      \end{array}
     \right\}
      \quad \Rightarrow \quad
      \hat x_n \rightarrow 0 \, \text{ in } \hat X.
      \tag{$\ast$}
      \label{ast}
    \]
   Let $(\hat x_n, \beta_n)_{n \geq 1}$ be a sequence as on the left hand side.
   Using Assumption \ref{assmpt:hybrid}, we obtain that
    \begin{align*}
     0
      &\leftarrow \Re \sp{(\mathfrak{A} - \ii \beta_n) x_n}{x_n}_X
      = \Re \sp{\mathfrak{A} x_n}{x_n}_X
      \leq \Re \sp{\mathfrak{B} x_n}{ \mathfrak{C} x_n}_{U}
       - \norm{\mathfrak{R} x_n}_Z^2
       \\
     0
      &\leftarrow \Re \sp{(A_c - \ii \beta_n) x_{c,n} + B_c \mathfrak{C} x_n}{x_{c,n}}_{X_c}
      \leq \Re \sp{C_c x_{c,n} + D_c \mathfrak{C} x_n}{\mathfrak{C} x_n}
      - \kappa \abs{\Pi \mathfrak{C} x_n}^2
    \end{align*}
   and adding up these two inequalities we derive
    \[
     \liminf_{n \rightarrow \infty} \Re \sp{(\mathfrak{B} + D_c \mathfrak{C}) x_n + C_c x_{c,n}}{\mathfrak{C} x_n} - \norm{\mathfrak{R} x_n}_Z^2 - \kappa \abs{\Pi \mathfrak{C} x_n}^2
     \geq 0. 
    \]
   Now, since $\ker D_c \subseteq \ker B_c \cap \ker C_c^*$, and $(\mathfrak{B} + D_c \mathfrak{C}) x_n + C_c x_{c,n}$ by choice of the sequence, cf.\ \eqref{ast}, lies in $\ran  \left[ \begin{array}{cc} C_c & D_c \end{array} \right]$, this inequality is equivalent to the statement
    \[
     \liminf_{n \rightarrow \infty} \Re \sp{(\mathfrak{B} + D_c \mathfrak{C}) x_n + C_c x_{c,n}}{\Pi \mathfrak{C} x_n} - \norm{\mathfrak{R} x_n}_Z^2 - \kappa \abs{\Pi \mathfrak{C} x_n}^2
     \geq 0.
    \]
   Namely, for every $C_c \eta$, $D_c \mu$ one has
    \begin{align*}
     \sp{C_c \eta}{(I - \Pi) \mathfrak{C} x}
      &= \sp{\eta}{C_c^* (I - \Pi) \mathfrak{C} x}
      = 0
      \\
     \sp{D_c \mu}{(I - \Pi) \mathfrak{C} x}
      &= \sp{(I - \Pi) D_c \mu}{\mathfrak{C} x}
      = 0
    \end{align*}
   as $(I - \Pi)$ projects onto $\ker D_c$.
   Since $(\mathfrak{B} + D_c \mathfrak{C}) x_n + C_c x_{c,n} \rightarrow 0$ by \eqref{ast}, we then deduce that $\Pi \mathfrak{C} x_n \rightarrow 0$ and $\mathfrak{R} x_n \rightarrow 0$:
   \emph{Assume} that $\limsup_{n \rightarrow \infty} \abs{\Pi \mathfrak{C} x_n} > 0$. Deviding by $\abs{\Pi \mathfrak{C} x_n}$ for a suitable subsequence then gives
    \[
     \liminf_{n \rightarrow \infty} - \frac{\norm{\mathfrak{R} x_n}_Z^2}{\norm{\Pi \mathfrak{C} x_n}} - \kappa \abs{\Pi \mathfrak{C} x_n} \geq 0
    \]
   and $\limsup_{n \rightarrow\infty} \abs{\Pi \mathfrak{C} x_n} = 0$, a contradiction. Hence, $\lim_{n \rightarrow \infty} \abs{\mathfrak{C} x_n} = 0$ and then
    \[
     \liminf_{n \rightarrow \infty} - \norm{\mathfrak{R} x_n}_Z^2 = 0
    \]
   gives $\lim_{n \rightarrow \infty} \mathfrak{R} x_n = 0$ as well.
   Since $\ker \Pi \subseteq \ker B_c \cap \ker D_c$, this also implies that
    \[
     B_c \mathfrak{C} x_n, D_c \mathfrak{C} x_n
      \rightarrow 0 \, \text{ in } Y.
    \]
   Therefore,
    \[
     x_{c,n}
      = (\ii \beta_n - A_c)^{-1} \left[ B_c \mathfrak{C} x_n - (B_c \mathfrak{C} x_n + A_c x_{c,n} - \ii \beta_n x_{c,n}) \right]
      \rightarrow 0 \, \text{ in } X_c,
    \]
   using that $\sup_{\beta \in \R} \norm{ (\ii \beta - A_c)^{-1} }_{\B(X_c)} < \infty$ and both $B_c \mathfrak{C} x_n$ and $B_c \mathfrak{C} x_n + A_c x_{c,n} - \ii \beta_n x_{c,n}$ tend to zero.
   As a consequence, also
    \[
     \mathfrak{B} x_n
      = (\mathfrak{B} x_n + D_c \mathfrak{C} x_n + C_c x_{c,n}) - D_c \mathfrak{C} x_n - C_c x_{c,n}
      \rightarrow 0 \, \text{ in } U
    \]
   as all three summands converge to zero.
   Then
    \[
     \sum_{j \in \mathcal{J}} \norm{\mathfrak{R}^j x_n^j}_{Z_j}^2
      \leq \abs{\mathfrak{B} x_n}_U^2 + \abs{\Pi \mathfrak{C} x_n}_Y^2 + \norm{\mathfrak{R} x_n}_Z^2
      \rightarrow 0
      \quad \Rightarrow \quad
      \mathfrak{R}^j x_n^j \rightarrow 0, \quad j \in \mathcal{J}.
    \]
   Now, for every $j \in \mathcal{J}$, we have $(\mathfrak{A}^j - \ii \beta_n) x_n^j \rightarrow 0$ and $\mathfrak{R}^j x_n^j \rightarrow 0$, so that by property AIEP we obtain $x_n^j \rightarrow 0$ in $X^j$ for all $j \in \mathcal{J}$, i.e.\ $x_n \rightarrow 0$ in $X$ as well, i.e.\ $\hat x_n \rightarrow 0$ in $\hat X$.
    
   Next, let us show the assertion on uniform exponential stability.
   By the Gearhart-Greiner-Pr\"uss-Huang Theorem, we need to show that
    \[
     \left\{
      \begin{array}{l}
       (\hat x_n)_{n \geq 1} \subseteq \dom(\hat A),
        \, \sup_{n \in \N} \norm{\hat x_n}_{\hat X} < \infty
        \\
       (\hat \beta_n)_{n \geq 1} \subseteq \R,
        \, \abs{\beta_n} \rightarrow \infty
        \\
       (\hat A - \ii \beta_n) x_n
        \rightarrow 0 \text{ in } \hat X 
      \end{array}
       \quad
       \Rightarrow
       \quad
       \hat x_n \rightarrow 0 \text{ in } \hat X.
     \right.
    \]
   So let $(\hat x_n, \beta_n)_{n \geq 1} \subseteq \dom(\hat A) \times \R$ be such a sequence.
   Then, by dissipativity of $\hat A$ we have
    \[
     0
      \leftarrow \Re \sp{(\hat A - \ii \beta_n) \hat x_n}{\hat x_n}_{\hat X}
      = \Re \sp{\hat A \hat x_n}{\hat x_n}_{\hat X}
      \leq - \norm{\mathfrak{R} x_n}_Z^2 - \kappa \abs{\Pi \mathfrak{C} x_n}_Y^2
      \leq 0
    \]
   and therefore $\mathfrak{R} x_n \rightarrow 0$ and $\Pi \mathfrak{C} x_n \rightarrow 0$.
   Moreover, $(\mathfrak{B} + D_c \mathfrak{C}) x_n + C_c x_{c,n} = 0$ by definition of $\dom(\hat A)$ and $(\hat A - \ii \beta_n) \hat x_n \rightarrow 0$ means that in particular
     \[
      (\mathfrak{A} - \ii \beta_n) x_n \rightarrow 0,
       \quad
       (A_c - \ii \beta_n) x_{c,n} + B_c \mathfrak{C} x_n \rightarrow 0.
     \]
   By property \eqref{ast}, this means that $\hat x_n \rightarrow 0$ in $\hat X$ and uniform exponential stability follows.
  \item
   If for all $j \in \mathcal{J}$, we even have property $\mathrm{AIEP}_{\tau^j \circ \H^j}$, then for the sequence $(\hat x_n, \beta_n)_{n \geq 1}$ as in \eqref{ast} of the previous case we do not only have $x_n^j \rightarrow 0$, but also $\tau^j(\H^j x_n^j) \rightarrow 0$ for all $j \in \mathcal{J}$, so that the last assertion follows as well. 
 \end{enumerate}
 \end{proof}

\section{Networks of Hybrid PH-ODE Systems}
\label{sec:networks_hybrid_systems}

 Next, we want to exploit possible \emph{structural conditions} on the hybrid interconnected port-Hamiltonian-control system to have uniform exponential stability under more restrictive structural assumptions, but weaker assumptions on the dissipativity of the subsystems.
 Instead of viewing the system as a family of port-Hamiltonian systems $\mathfrak{S}^j$ which are coupled via boundary feedback and control with a finite-dimensional control system, we cluster the port-Hamiltonian systems and parts of the finite-dimensional control system into hybrid PH-ODE systems $\hat {\mathfrak{S}}^j$ ($j \in \hat {\mathcal{J}}$) as in Definition \ref{def:hybrid_PH-ODE} and assume that the resulting evolutionary system
 can be written in an equivalent \emph{serially connected} (or, maybe more precisely, \emph{rooted graph}) form
  \[
   \begin{cases}
    \frac{\dd}{\dd t} \hat x^j
     = \hat {\mathfrak{A}}^j \hat x^j
     := \hat {\mathfrak{A}}^j \hat x^j
     \\
    \hat {\mathfrak{B}}^j \hat x^j
     = \sum_{i \in \hat {\mathcal{J}}} \hat K^{ij} \hat {\mathfrak{C}}^i \hat x^i,
     \quad
     j \in \hat {\mathcal{J}}
   \end{cases}
  \]
 where $\dom(\hat {\mathfrak{A}}^j) = \dom(\hat {\mathfrak{C}}^j) = \dom(\hat {\mathfrak{B}}^j) = \dom(\hat {\mathfrak{A}}^j)$ and
  \[
   \hat {\mathfrak{B}}^j: \dom(\hat {\mathfrak{B}}^j) \subseteq \hat X \rightarrow \hat U^j,
    \quad
   \hat {\mathfrak{C}}^j: \dom(\hat {\mathfrak{C}}^j) \subseteq \hat X \rightarrow \hat Y^j
  \]
 and where the Hilbert spaces $\hat U^j$ and $\hat Y^j$ may be distinct, but $\hat U^j \times \hat Y^j \cong \hat U \times \hat Y$.
 Moreover, for this interconnection to be \emph{serial} (or, in \emph{rooted graph form}) we demand the following.
  \begin{assumption}
   \label{assmpt:serial}
   Assume that \ $\hat K = (\hat K^{ij})_{i,j \in \hat {\mathcal{J}}}$ \ is strictly lower-block triangular, i.e.\ $\hat K^{ij} = 0$ for $i,j \in \hat {\mathcal{J}}$ with $\ii \leq j$.
  \end{assumption}
 Under this assumption one can hope for better (i.e.\ less restrictive) conditions for asymptotic or uniform exponential stability, similar to the interconnection of a PHS with a finite dimensional control system.
  \begin{figure}
	\centering
	\includegraphics[scale = 0.8]{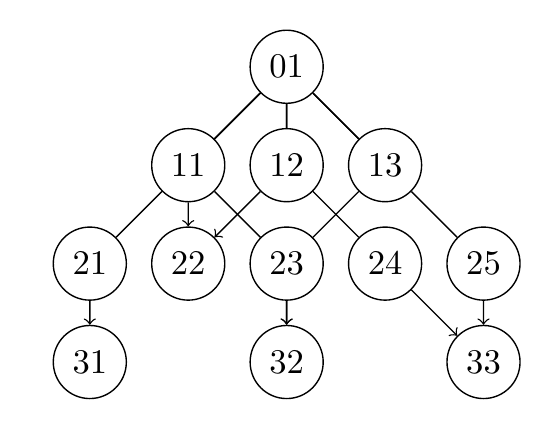}
	\caption{Example of a rooted graph.}
	\label{fig:diagram_spring-mass-beam-system}
 \end{figure}
 
 \begin{assumption}
  \label{assmpt:dissipation}
  There are linear maps $\hat {\mathfrak{R}}^j: \dom(\hat {\mathfrak{R}}^j) = \dom(\hat {\mathfrak{A}}^j) \rightarrow \hat {\mathcal{Z}}^j$ ($j \in \hat {\mathcal{J}}$) such that
   \[
    \Re \sp{\hat A \hat x}{\hat x}_{\hat X}
     \leq - \sum_{j \in \hat {\mathcal{J}}} \norm{\hat {\mathfrak{R}}^j \hat x^j}_{\hat {\mathcal{Z}}^j}^2,
     \quad
     \hat x \in \dom(\hat A). 
   \]
 \end{assumption}
 
 Under these two assumptions we can formulate the following
 
 \begin{theorem}[Asymptotic stability]
  \label{thm:asymptotic_stability}
  Let Assumptions \ref{assmpt:serial} and \ref{assmpt:dissipation} hold true.
  Assume that $\sigma(A_c) \subseteq \C_0^-$, i.e.\ $(\ee^{t A_c})_{t \geq 0}$ is an exponentially stable semigroup on $X_c$, and that for all $j \in \hat {\mathcal{J}}$ the pairs $(\hat {\mathfrak{A}}^j, (\hat {\mathfrak{B}}^j, \hat {\mathfrak{R}}^j))$ have property ASP, i.e.\
   \[
    \left.
    \begin{array}{l}
    \hat x^j \in \dom(\hat {\mathfrak{A}}^j) \\
    \beta \in \R \\
    \hat {\mathfrak{A}}^j \hat x^j = \ii \beta \hat x^j \\
    (\hat {\mathfrak{B}}^j \hat x^j, \hat {\mathfrak{R}}^j \hat x^j)
     = 0
    \end{array}
    \right\}
    \quad
    \Longrightarrow
    \quad
    \hat x^j = 0
    \tag{ASP}
   \]
  Then $\hat A$ generates an (asymptotically) strongly stable $C_0$-semigroup $(\hat T(t))_{t \geq 0}$ on $\hat X$.
 \end{theorem}
  
 \begin{proof}
 We use the Arendt-Batty-Lyubich-V\~u Theorem again.
 Since $\hat A$ generates a contractive $C_0$-semigroup and has compact resolvent by Theorem \ref{thm:generation_network-dynamic}, we need to show that $\sigma_p(\hat A) \cap \ii \R = \emptyset$.
 Let $\hat x \in \dom(\hat A)$ such that $\hat A \hat x = \ii \beta \hat x$ for some $\beta \in \R$.
 Then, in particular
  \[
   0
    = \Re \sp{\ii \beta \hat x}{\hat x}_{\hat X}
    = \Re \sp{\hat A \hat x}{\hat x}_{\hat X}
    \leq - \sum_{j \in \hat {\mathcal{J}}} \norm{\hat {\mathfrak{R}}^j \hat x^j}_{\mathcal{Z}^j}^2
    \leq 0
  \]
 and therefore $\hat {\mathfrak{R}}^j \hat x^j = 0$ for all $j \in \hat {\mathcal{J}}$.
 Moreover, by definition of $\hat A$ and Assumption \ref{assmpt:serial}, we have
  \[
   \hat {\mathfrak{B}}^j \hat x^j
    = \sum_{i = 1}^{j-1} \hat K^{ij} \hat {\mathfrak{C}}^i \hat x^i,
    \quad
    j \in \hat {\mathcal{J}}
  \]
 Hence, whenever we know that $\hat x^i = 0$ for all $i < j$, then $(\hat {\mathfrak{R}}^j \hat x^j, \hat {\mathfrak{B}}^j \hat x^j) = 0$ and since also $\hat {\mathfrak{A}}^j \hat x^j = \ii \beta \hat x^j$, property ASP implies that then $\hat x^j = 0$ as well.
 Since this is certainly true for $j = 1$, it follows iteratively that $\hat x^j = 0$ for all $j \in \hat {\mathcal{J}}$, i.e.\ $\hat x = 0$ and therefore $\sigma_p(\hat A) \cap \ii \R = \emptyset$.
 The Arendt-Batty-Lyubich-V\~u Theorem gives us strong stability of the semigroup $(\hat T(t))_{t \geq 0}$.
 \end{proof}
 
 Similarly, for uniform exponential stability the following result relies on property $\mathrm{AIEP}_\tau$.

 \begin{theorem}[Uniform exponential stability]
  \label{thm:exp_stability}
  Assume that Assumption \ref{assmpt:serial} and \ref{assmpt:dissipation} hold true.
  Further assume that $\hat A$ generates an (asymptotically) strongly stable contraction semigroup $(\hat T(t))_{t \geq 0}$ on $\hat X$, and that for all $j \in \mathcal{J}$ the pairs $(\hat {\mathfrak{A}}^j, (\hat {\mathfrak{B}}^j, \hat {\mathfrak{R}}^j))$ have property $\mathrm{AIEP}_{\hat {\mathfrak{C}}^j}$, i.e.\
   \[
    \left.
    \begin{array}{l}
     (\hat x^j_n)_{n \geq 1} \subseteq \dom(\hat {\mathfrak{A}}^j), \\
     \sup_{n \in \N} \| \hat x^j_n \|_{\hat X^j} < \infty \\
     (\beta_n)_{n \geq 1} \subseteq \R, \\
     \abs{\beta_n} \rightarrow \infty \\
     (\hat {\mathfrak{A}}^j - \ii \beta_n) \hat x_n \rightarrow 0 \\
     (\hat {\mathfrak{B}}^j \hat x^j_n, \hat {\mathfrak{R}}^j \hat x^j_n) \rightarrow 0
    \end{array}
    \right\}
    \quad
    \Longrightarrow
    \quad
    \left\{
    \begin{array}{l}
    \hat x^j_n \rightarrow 0 \quad \text{in } \hat X^j \\
    \hat {\mathfrak{C}}^j \hat x^j_n \rightarrow 0 \quad \text{in } \hat Y^j
    \end{array}
    \right.
    \tag{$\mathrm{AIEP}_{\hat {\mathfrak{C}}^j}$}
    \label{eqn:AIEP_C^j_ser}
   \]
  Then the $C_0$-semigroup $(\hat T(t))_{t \geq 0}$ is uniformly exponentially stable.
 \end{theorem}
 
 \begin{remark}
  The assumption that in \eqref{eqn:AIEP_C^j_ser} one has $\hat {\mathfrak{C}}^j \hat x^j_n \rightarrow 0 \quad \text{in } \hat Y^j$ could be weakened to $\Pi^j \hat {\mathfrak{C}}^j \hat x^j_n \rightarrow 0 \quad \text{in } \hat Y^j$ where $\Pi^j: \hat Y^j \rightarrow \hat Y^j$ is the orthogonal projection onto $(\cap_{i > j} \ker \hat K^{ij})^\bot$, however in concrete examples this does not make any difference.
  If necessary, one could extend the system by an artificial additional hybrid system $\mathfrak{S}$ to ensure the structure of Theorem \ref{thm:exp_stability}.
 \end{remark}

 \begin{proof}[Proof of Theorem \ref{thm:exp_stability}]
 Since $\hat A$ generates an asymptotically stable semigroup and has compact resolvent, $\sigma(\hat A) = \sigma_p(\hat A) \subseteq \C_0^-$ and we thus only have to prove that $\sup_{\beta \in \R} \norm{(\ii \beta - \hat A)^{-1}} < \infty$.
 Therefore, take any sequence $(\hat x_n, \beta_n)_{n \geq 1} \subseteq \dom(\hat A) \times \R$ such that $\sup_{n \in \N} \norm{\hat x_n}_{\hat X} < \infty$, $\abs{\beta_n} \rightarrow \infty$ and $(\ii \beta_n - \hat A) \hat x_n \rightarrow 0$ in $\hat X$.
 Then, by Assumption \ref{assmpt:dissipation} we obtain
  \[
   0
    \leftarrow \Re \sp{(\hat A - \ii \beta_n) \hat x_n}{\hat x_n}_{\hat X}
    = \Re \sp{\hat A \hat x_n}{\hat x_n}_{\hat X}
    \leq - \sum_{j \in \hat {\mathcal{J}}} \norm{\hat {\mathfrak{R}}^j \hat x^j_n}_{\hat {\mathcal{Z}}^j}^2
    \leq 0
  \]
 and therefore $\hat {\mathfrak{R}}^j \hat x^j \rightarrow 0$ for all $j \in \hat {\mathcal{J}}$.
 Moreover, by Assumption \ref{assmpt:serial}, we have
  \[
   \hat {\mathfrak{B}}^j \hat x^j_n
    = \sum_{i = 1}^{j-1} \hat K^{ji} \hat {\mathfrak{C}}^i \hat x^i_n,
    \quad
    j \in \hat {\mathcal{J}}
  \]
 and property $\mathrm{AIEP}_{\hat {\mathfrak{C}}^j}$ now implies that $\hat x^j_n \rightarrow 0$ and $\hat {\mathfrak{C}}^j \hat x^j_n \rightarrow 0$ whenever $\hat {\mathfrak{C}}^i \hat x^i_n \rightarrow 0$ for all $i < j$.
 Again, this is true for $j = 1$ and by induction it follows that $\hat x^j \rightarrow 0$ and $\hat {\mathfrak{C}}^j \hat x^j_n \rightarrow 0$ for all $j \in \hat {\mathcal{J}}$.
 In particular, $\hat x_n \rightarrow 0$ in $\hat X$ and, therefore, by the Gearhart-Pr\"uss-Huang Theorem the semigroup $(\hat T(t))_{t \geq 0}$ is uniformly exponentially stable.
 \end{proof}

 \section{Applications}
 \label{sec:applications}

 We now discuss the properties ASP and $\mathrm{AIEP}_{\hat {\mathfrak{C}}^j}$ for some particular classes of PDE which are of port-Hamiltonian type.
 We aim to give several types of interconnection structures, thus motivating the abstract results of the previous sections.
 We begin with

 \begin{proposition}
  Assume that $N_j = 1$ for all $\mathcal{J} = \hat {\mathcal{J}}$ (i.e.\ every hybrid PH-ODE systems consists of exactly one port-Hamiltonian system $\mathfrak{S}^j$ and a controller $\Sigma^j_c$),
  all Hamiltonian matrix density functions $\H^j$ ($j \in \mathcal{J}$) are Lipschitz continuous on $[0,1]$,
$\sigma_p(A^j_c) \subseteq \C_0^-$ for all $j \in {\mathcal{J}}$,
 there are $\mathfrak{R}^j: \dom({\mathfrak{A}}^j) \rightarrow {\mathcal{Z}}^j$ \ such that
   \[
    \Re \sp{\hat A x}{x}_{\hat X}
     \leq - \sum_{j \in {\mathcal{J}}} \abs{{\mathfrak{R}}^j x^j}^2
   \]
  and
   \[
    \abs{(\H_j x^j)(0)}
     \lesssim \abs{{\mathfrak{R}}^j \hat x^j} + \abs{{\mathfrak{B}}^j x^j},
     \quad
     x^j \in \dom({\mathfrak{A}}^j),
     \quad
     j \in {\mathcal{J}}.
   \]
  Then the $C_0$-semigroup generated by $\hat A$ is uniformly exponentially stable.
 \end{proposition}
 
 \begin{proof}
 This proposition follows from the Theorems \ref{thm:asymptotic_stability} and \ref{thm:exp_stability} above, and with the following lemma on port-Hamiltonian systems of order $N = 1$ and Theorem \ref{thm:stability} (the latter traducing properties ASP and $\mathrm{AIEP}_{\tau \circ \H}$ from the systems $\mathfrak{S}^j$ ($j \in \mathcal{J}$) to ${\mathfrak{A}}^j$ ($j \in {\mathcal{J}}$)).
 \end{proof}
 
 \begin{lemma}
 \label{lem:ASP_AIEP_N=1}
  Let $\mathfrak{S} = (\mathfrak{A}, \mathfrak{B}, \mathfrak{C})$ be a port-Hamiltonian system of order $N = 1$ and  $\H: [0,1] \rightarrow \K^{d \times d}$ be Lipschitz continuous.
  Then the following assertions hold true:
   \begin{enumerate}
    \item
     If $x \in \dom(\mathfrak{A})$ with $\mathfrak{A} x = \ii \beta x$ for some $\beta \in \R$, and additionally $(\H x)(0) = 0$, then $x = 0$.
    \item
     If $(x_n, \beta_n)_{n \geq 1} \subseteq \dom(\mathfrak{A}) \times \R$ with $\sup_{n \in \N} \norm{x_n}_X < \infty$, $\abs{\beta_n} \rightarrow \infty$ and $(\mathfrak{A} - \ii \beta_n) x_n \rightarrow 0$ in $X$, $(\H x_n)(0) \rightarrow 0$ in $\K^d$, then $x_n \rightarrow 0$ in $X$ and $(\H x_n)(1) \rightarrow 0$ in $\K^d$.
   \end{enumerate}
 \end{lemma}
 
 \begin{proof}
 \begin{enumerate}
  \item
   See the proof of Proposition 2.11 in \cite{AugnerJacob_2014}.
  \item
   For the property that $x_n \rightarrow 0$ in $X$, see the proof of Proposition 2.12 in \cite{AugnerJacob_2014}.
   Repeating the proof presented there for $q = 1$ shows that that
    \[
     \frac{1}{2} \norm{x_n}_X^2 + \frac{1}{2} \left[ \sp{x_n(\z)}{\H(\z) x_n(\z)}_{\K^d} \right]_0^1
     \rightarrow 0, 
    \]
   and since $x_n \rightarrow 0$ in $X$, $(\H x_n)(0) = \H(0) x_n(0) \rightarrow 0$ and $\H(1)$ is symmetric positive definite, this implies that $(\H x_n)(1) \rightarrow 0$ as well.
 \end{enumerate}
   \end{proof}
 
 \begin{example}[Serially Connected Strings]
  \label{exa:chain_wave_eqn}
  As an example where the structure of the interconnection can be employed to ensure uniform exponential stability, consider the following chain of serially connected strings, see Figure \ref{fig:Diagram_Serial_Interconnection_Of_Strings}, which are modelled by the \emph{non-uniform one-dimensional wave equation}:
  \[
   \rho(\z) \omega_{tt}(t,\z)
    - ( T(\z) \omega_\z)_\z(t,\z)
    = 0,
    \quad
    \z \in (\z^{j-1}, \z^j), \, t \geq 0, \quad j = 1, \ldots, m
  \]
 where $0 =: \z^0 < \z^1 < \ldots < \z^m := L$ and $0 < \varepsilon \leq \rho^j := \rho|_{(\z^{j-1},\z^j)}, T^j := T|_{(\z^{j-1},\z^j)} \in \Lip(\z^{j-1}, \z^j;\R)$.
 \begin{figure}
  \centering
  \includegraphics[scale = 0.8]{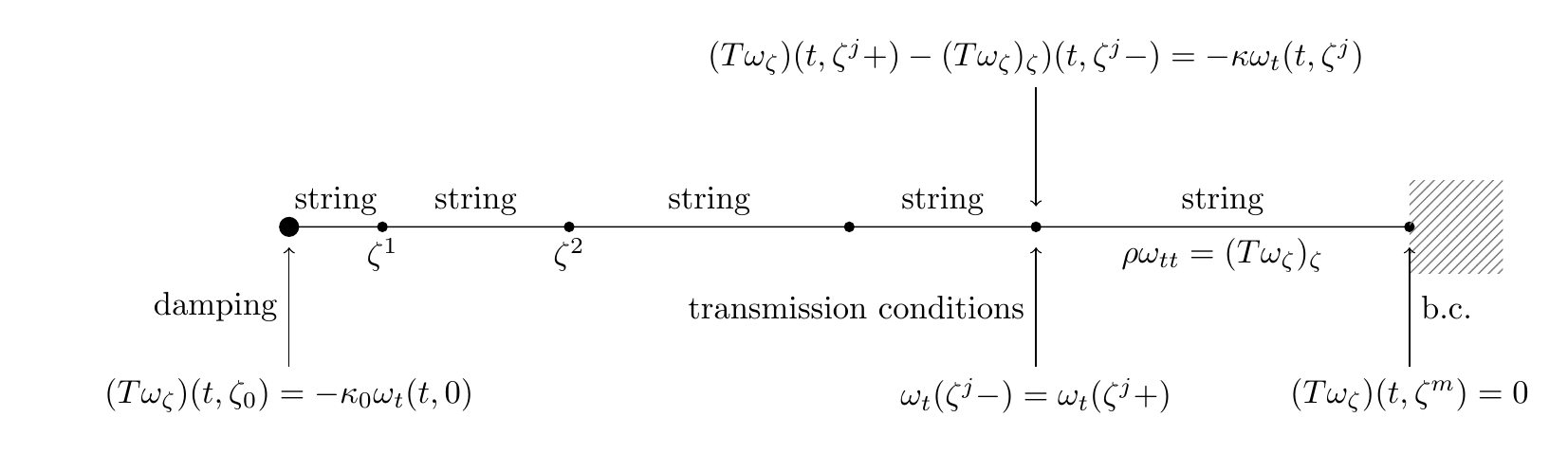}
	\caption{A Chain of Serially Interconnected Strings.}
	\label{fig:Diagram_Serial_Interconnection_Of_Strings}
 \end{figure}
 The chain of strings is damped at the left end, free at the right end, and interconnected in a dissipative or conservative way:
  \begin{align}
   \left( T \omega_\z \right)(t,\z_0)
     &= - \kappa^0 \omega_t (t,\z_0),
     \quad
    t \geq 0 \quad (\text{for some } \kappa^0 > 0) \nonumber \\
    \left( T \omega_\z \right)(t,\z^m)
     &= 0,
     \quad
    t \geq 0 \nonumber \\
    \omega_t (t, \z^j-)
     &= \omega_t (t, \z^j+),
     \quad
     t \geq 0, \, j = 1, \ldots, m-1
     \nonumber \\
    \left( T \omega_\z \right)(t, \z^j-)
     &- \left( T \omega_\z \right)(t, \z^j+)
     = - \kappa^j \omega_t (t, \z^j),
     \nonumber \\
     &\qquad
     t \geq 0, \, j = 1, \ldots, m-1 \quad (\text{for some } \kappa^j \geq 0)
     \label{eqn:interconnection_string_wave_eqn}
  \end{align}
 \end{example}
 We show that this example can be written as a network of port-Hamiltonian systems of order $N = 1$, and the theory developed in this section can be applied to deduce stability properties for this system.
 Using a scaling argument we may and reduce the general case to the special case $\z^j = j$.
 We may then identify $x^j(t,\z) := (\rho(j + \z) \omega_t(t,j + \z), - \omega_\z(t,j+\z))$ and $\H^j(\z) := \operatorname{diag}\, (1/\rho(j+\z), T(j + \z))$ and obtain for $P^j_1 = \left[ \begin{smallmatrix} 0 & 1 \\ 1 & 0 \end{smallmatrix} \right], P^j_0 = 0 \in \K^{2 \times 2}$, $j = 2, \ldots, m$, the first order port-Hamiltonian systems
  \[
   \mathfrak{A}^j x^j
    = \left[ P_1 \frac{\partial}{\partial \z} + P_0 \right] (\H^j x^j)(\z),
    \quad
    x^j \in \dom(\mathfrak{A}^j)
    = \{ x^j \in L^2(0,1;\K^2): \, (\H^j x^j) \in H^1(0,1;\K^2) \}
  \]
 with boundary input and output maps
  \[
   \mathfrak{B}^j x^j
    = \left( \begin{array}{c} - (\H^j_2 x^j_2)(0) \\ (\H^j_1 x^j_1)(1) \end{array}\right),
    \quad
   \mathfrak{C}^j x^j
    = \left( \begin{array}{c} (\H^j_1 x^j_1)(0) \\ (\H^j_2 x^j_2)(1) \end{array} \right),
    \quad
   \dom(\mathfrak{B}^j)
    = \dom(\mathfrak{C}^j)
    = \dom(\mathfrak{A}^j),
    \quad
   j \in \mathcal{J} \setminus \{m\}
  \]
 and
  \[
   \mathfrak{B}^m x^m
    = \left( \begin{array}{c} - (\H^m_2 x^m_2)(0) \\ (\H^m_2 x^m_2)(1) \end{array}\right),
    \quad
   \mathfrak{C}^m x^m
    = \left( \begin{array}{c} (\H^m_1 x^m_1)(0) \\ (\H^m_1 x^m_1)(1) \end{array} \right),
    \quad
   \dom(\mathfrak{B}^m)
    = \dom(\mathfrak{C}^m)
    = \dom(\mathfrak{A}^m).
  \]
 For this choice of the boundary input and output maps, the port-Hamiltonian systems $\mathfrak{S}^j = (\mathfrak{A}^j, \mathfrak{B}^j, \mathfrak{C}^j)$ become impedance passive with energy state spaces $X^j = (L_2(0,1;\K^2), \norm{\cdot}_{\H^j})$ and input and output spaces $U^j = Y^j = \K^2$.
 This property corresponds to the formal \emph{energy balance equation}
  \[
   \frac{\dd}{\dd t} H_{\mathrm{WE}}(t)
    := \frac{\dd}{\dd t} \frac{1}{2} \int_0^1 \rho(\z) \abs{\omega_t(t,\z)}^2 + T(\z) \abs{\omega_\z(t,\z)}^2 \, \dd \z
    = \Re \left[ \sp{\omega_t(t,\z)}{T(\z) \omega_\z(t,\z)}_\K \right]_0^1 
  \]
 for the wave equation $\rho(\z) \omega_{tt}(t,\z) = (T(\cdot) \omega_\z)_\z(t,\z)$.
 The interconnection structure \eqref{eqn:interconnection_string_wave_eqn} can then be written in the boundary feedback form
  \[
   \mathfrak{B} x
    = \left[ \begin{array}{ccccccccc}
     - \kappa_0 & 0 & 0 & 0 & 0 & 0 & 0 & 0 & 0 \\
     0 & 0 & 1 & 0 & 0 & 0 & 0 & 0 & 0 \\
     0 & -1 & -\kappa_1 & 0 & 0 & 0 & 0 & 0 & 0 \\
     0 & 0 & 0 & 0 & 1 & 0 & 0 & 0 & 0 \\
     0 & 0 & 0 & -1 & -\kappa_2 & 0 & 0 & 0 & 0 \\
     0 & 0 & 0 & 0 & \ddots & \ddots & \ddots & 0 & 0 \\
     0 & 0 & 0 & 0 & 0 & 0 & 0 & 1 & 0 \\
     0 & 0 & 0 & 0 & 0 & 0 & -1 & - \kappa_m & 0 \\
     0 & 0 & 0 & 0 & 0 & 0 & 0 & 0 & 0
    \end{array} \right]
    \mathfrak{C} x
    =: K \mathfrak{C} x.
  \]
 Clearly, the symmetric part of $K$,
  \[
   \Sym K
    = \left[ \begin{array}{ccccccccc}
     - \kappa_0 & 0 & 0 & 0 & 0 & 0 & 0 & 0 & 0 \\
     0 & 0 & 0 & 0 & 0 & 0 & 0 & 0 & 0 \\
     0 & 0 & -\kappa_1 & 0 & 0 & 0 & 0 & 0 & 0 \\
     0 & 0 & 0 & 0 & 0 & 0 & 0 & 0 & 0 \\
     0 & 0 & 0 & 0 & -\kappa_2 & 0 & 0 & 0 & 0 \\
     0 & 0 & 0 & 0 & \ddots & \ddots & \ddots & 0 & 0 \\
     0 & 0 & 0 & 0 & 0 & 0 & 0 & 0 & 0 \\
     0 & 0 & 0 & 0 & 0 & 0 & 0 & - \kappa_{m-1} & 0 \\
     0 & 0 & 0 & 0 & 0 & 0 & 0 & 0 & 0
    \end{array} \right]
  \]
 is negative semi-definite, thus the operator
  \[
   A x = (\mathfrak{A}^j x^j)_{j \in \mathcal{J}},
    \quad
    \dom(A)
     = \{x \in \dom(\mathfrak{A}) = \prod_{j \in \mathcal{J}} \dom(\mathfrak{A}^j): \, \mathfrak{B} x = K \mathfrak{C} x\}
  \]
 is dissipative on the product Hilbert space $X = \prod_{j \in \mathcal{J}} X^j$ and thus generates a contractive $C_0$-semigroup $(T(t))_{t \geq 0}$ on $\hat X$ by Theorem \ref{thm:generation_network-dynamic} (or, by Theorem 4.1 in \cite{LeGorrecZwartMaschke_2005}).
 We employ Theorem \ref{thm:asymptotic_stability} and Theorem \ref{thm:exp_stability} to deduce uniform exponential stability, as long as the parameter functions $\rho^j$ and $T$ are Lipschitz continuous on $(\z^{j-1}, \z^j)$, $j = 1, \ldots, m$.
 For this end, we reformulate the boundary conditions in a form more suitable for the setting of these theorems, and set
  \begin{align*}
   \mathfrak{B}^1 x^1
    &= ((\H^1_2 x^1_2)(0) + \kappa^0  (\H^1_1 x^1_1)(0)) \in \K,
    \\
   \mathfrak{C}^1 x^1
    &= ((\H^1_1 x^1_1)(0), (\H^1 x^1)(1)) \in \K^3,
    \\
   \mathfrak{B}^j x^j
    &= (\H^j x^j)(0) \in \K^2,
    \\
   \mathfrak{C}^j x^j
    &= (\H^j x^j)(1) \in \K^2,
    \qquad
    j = 2, \ldots, m-1,
    \\
   \mathfrak{B}^m x^m
    &= ((\H^m x^m)(0), (\H^m_2 x^m_2)(1)) \in \K^3,
    \\
   \mathfrak{C}^m x^m
    &= (\H^m_1 x^m_1)(0) \in \K.
  \end{align*}
 (In this situation, we simply have $X_c = \{0\}$.)
 Then, the boundary conditions can be rewritten in the form
  \[
   \mathfrak{B}^j x^j
    = \sum_{i=1}^{j-1} K^{ij} \mathfrak{C}^i x^i,
    \quad
    j \in \hat {\mathcal{J}} = \{1, \ldots, m\}
  \]
 for appropriate matrices $K^{ij}$, $i,j \in {\mathcal{J}}$, and such that $K = (K^{ij})_{i,j \in {\mathcal{J}}}$ is strictly lower-block triangular.

 \begin{corollary}
  In the situation of Example \ref{exa:chain_wave_eqn}, assume that $\rho^j, T^j: (\z^{j-1}, \z^j) \rightarrow (0, \infty)$ are Lipschitz continuous for each string $j \in \mathcal{J}$ of the serially connected chain, and assume that $\kappa^0 > 0$ whereas $\kappa^j \geq 0$ for $j \in \mathcal{J}$.
  Then the problem is well-posed, i.e.\ for every initial datum
   \[
    (\omega(0,\cdot), \omega_t(0,\cdot))
     = (\omega_0, \omega_1) \in \prod_{j \in\mathcal{J}} H^1(\z^{j-1}, \z^j) \times \prod_{j \in \mathcal{J}} L_2(\z^{j-1}, \z^j)
   \]
  there is a unique strong solution $\omega: \R \rightarrow \prod_{j \in\mathcal{J}} H^1(\z^{j-1}, \z^j)$ such that
   \[
    \omega \in C(\R_+; \prod_{j \in\mathcal{J}} H^1(\z^{j-1}, \z^j)),
     \quad
     \omega_t \in C(\R_+; \prod_{j \in\mathcal{J}} L_2(\z^{j-1}, \z^j))
   \]
  with non-increasing energy
   \[
    H_\mathrm{WE}(t)
     = \frac{1}{2} \sum_{j=1}^m \int_{\z_{j-1}}^{\z_j} \rho_j |\omega_t|^2 + T_j |\omega_\z|^2 \dd \z
   \]  
  and there are constants $M \geq 1$ and $\eta < 0$ such that
   \[
    H_\mathrm{WE}(t)
     \leq M \ee^{\eta t} H_\mathrm{WE}(0),
     \quad
     t \geq 0
   \]
  holds uniformly for all initial data.
  Moreover, if additionally \[ ((T \omega_0)_\z , \omega_1) \in \prod_{j \in\mathcal{J}} H^1(\z^{j-1}, \z^j) \times \prod_{j \in\mathcal{J}} H^1(\z^{j-1}, \z^j) \] and satisfy the compatibility conditions for \eqref{eqn:interconnection_string_wave_eqn}, i.e.\
   \[
    \begin{cases}
     \left( T (\omega_0)_\z \right)(\z_0)
     = - \kappa^0 \omega_1^1(\z_0),
     \\
    \left( T (\omega_0^m)_\z \right)(\z^m)
     = 0,
     \\
    \omega_1(\z^j-)
     = \omega_1 (\z^j+),
     \\
    \left( T (\omega_0)_\z \right)(\z^j-)
     - \left( T (\omega_0)_\z \right)(\z^j+)
     = - \kappa^j \omega_1(\z^j),
     &
     j = 1, \ldots, m-1,
    \end{cases}
   \]
  the solution is classical, i.e.\
   \begin{align*}
    \omega &\in C^1(\R_+; \prod_{j \in\mathcal{J}} H^1(\z^{j-1}, \z^j)),
     \\
     \omega_t &\in C(\R_+; \prod_{j \in\mathcal{J}} H^1(\z^{j-1}, \z^j)),
     \\
     T \omega_\z &\in C(\R_+; \prod_{j \in\mathcal{J}} H^1(\z^{j-1}, \z^j)).
   \end{align*}
 \end{corollary}
 
 \begin{proof}
 By the port-Hamiltonian formulation, we can see that the impedance passivity of the systems $\mathfrak{S}^j = (\mathfrak{A}^j, \mathfrak{B}^j, \mathfrak{C}^j)$ and the structure of the interconnection by the static feedback matrix $K$ imply that
  \begin{align*}
   \Re \sp{A x}{x}_X
    &\leq - \sum_{j=1}^m \kappa_{j-1} \abs{\mathfrak{B}^j_1 x^j}^2
    = - \sum_{j=1}^m \kappa_{j-1} \abs{(\H^j_2 x^j_2)(0)}^2
    \\
    &\leq - \kappa_0 \abs{(\H^1_2 x^1_2)(0)}^2
    = - \frac{1}{\kappa_0} \abs{(\H^1_1 x^1_1)(0)}^2
    \\
    &\leq - \frac{1}{2} \min \{\kappa_0, \kappa_0^{-1}\} \abs{(\H^1 x^1)(0)}^2,
    \quad
    x \in \dom(A).
  \end{align*}
 This already implies well-posedness.
 Moreover, for each $j \geq 2$ we have
  \[
   \abs{\mathfrak{B}^j x^j}
    \leq \abs{(\H^j x^j)(0)}.
  \]
 Since all the pairs $(\mathfrak{A}^j, (\H^j x^j)(0))$ have property ASP by Lemma \ref{lem:ASP_AIEP_N=1}, as long as the parameter functions $\rho^j, T^j$ are Lipschitz continuous, it follows asymptotic stability from Theorem \ref{thm:asymptotic_stability}, and then, since by Lemma \ref{lem:ASP_AIEP_N=1} the pairs $(\mathfrak{A}^j, (\H^j x^j)(0))$ also have property $\mathrm{AIEP}_{\tau^j \circ \H^j}$ as well, uniform exponential stability follows by Theorem \ref{thm:exp_stability}.
 \end{proof}
 
 \begin{remark}
  It would be nice if one could apply Theorems \ref{thm:asymptotic_stability} and \ref{thm:exp_stability} to the case of a chain of Euler-Bernoulli beam models, cf.\ \cite{ChenEtAl_1987}, as well. Unfortunately, as it turns out a dissipativity condition like
   \begin{equation}
    \Re \sp{\hat A x}{x}_X
     \leq - \kappa \left( \abs{(\H x)(0)}^2 + \abs{(\H x)'(0)}^2 \right),
     \quad
     x \in \dom(\hat A)
     \label{eqn:dissipation_2nd_order}
   \end{equation}
  is not sufficient for uniform exponential stability of (closed-loop) port-Hamiltonian systems of order $N = 2$, and also for the special case of an Euler-Bernoulli beam such a property is not known.
  In particular, though clearly $(\mathfrak{A}, ((\H x)(0), (\H x)'(0)))$ has property ASP for port-Hamiltonian operators of order $N = 2$ with Lipschitz-continuous $\H: [0,1] \rightarrow \K^{d \times d}$, it is not known whether there are classes, e.g.\ Euler-Bernoulli beam type systems, for which properties AIEP and $\mathrm{AIEP}_\tau$ hold for the pair $(\mathfrak{A}, ((\H x)(0), (\H x)'(0)))$.
  Even more, dissipation of the form \eqref{eqn:dissipation_2nd_order} is not what can be ensured by the most usual damping conditions for the Euler-Bernoulli beam, namely only dissipation in three of the four components (or, the component being zero by the boundary conditions imposed on the system) of $((\H x)(0), (\H x)'(0))$ for the Euler-Bernoulli beam (where $d = 2$) is a realistic assumption.
  However, it is already known for thirty years \cite{ChenEtAl_1987}, that serially interconnected, homogeneous (i.e.\ constant parameters along each beam) Euler-Bernoulli beams can be uniformly exponentially stabilised at one end by suitable (realistic) boundary conditions, if one additionally assumes that the parameters are ordered in a monotone way.
  The same result for inhomogeneous beams, where the parameter functions on each beam are allowed to have Lipschitz continuous dependence on the spatial parameter $\z$, but still satisfy monotonicity conditions at the joints $\z^j$, will be shown in a forthcoming paper \cite{Augner_2018+a}.
%  There, it will be shown, that for such a series of Euler-Bernoulli beams there are particular choices of $\mathfrak{R}$ such that the pair $(\mathfrak{A}, \mathfrak{R})$ has properties ASP and AIEP, and then property $\mathrm{AIEP}_\tau$ follows as well.
%  Hence, in combination with our results here, this allows us to consider interconnection structures where, say, strings modelled by wave equations are interconnected (in a dissipative way) with one or more serially connected Euler-Bernoulli beams.
 \end{remark}

 \begin{example}[The Euler-Bernoulli Beam]
 \label{exa:euler-bernoulli_beam}
 The Euler-Bernoulli beam equation
  \[
   \rho(\z) \omega_{tt}(t,\z)
    + \frac{\partial^2}{\partial \z^2} \left( EI(\z) \omega_{\z\z}(t,\z) \right),
     \quad
     t \geq 0,
     \, \z \in (a,b)
  \]
 can be written in port-Hamiltonian form for $N = 2$ and the identification
  \[
   x(t,\z)
    = \left( \begin{array}{c} \rho(\z) \omega_t(t,\z) \\ \omega_{\z\z}(t,\z) \end{array} \right),
    \quad
   \H(\z)
    = \left[ \begin{array}{cc} \rho(\z)^{-1} & 0 \\ 0 & EI(\z) \end{array} \right].
  \]
 Choosing $P_2 = \left[ \begin{smallmatrix} 0 & -1 \\ 1 & 0 \end{smallmatrix} \right], P_1 = P_0 = 0 \in \K^{2 \times 2}$, we arrive at the first order in time, second order in space evolution equation
  \[
   \frac{\partial}{\partial t} x(t,\z)
    = \mathfrak{A} x(t,\z)
    := \left[ P_2 \frac{\partial^2}{\partial \z^2} + P_1 \frac{\partial}{\partial \z} + P_0 \right] (\H(\z) x(t,\z)),
     \quad
     t \geq 0, \, \z \in (a,b).
  \]
 \end{example}
 After appropriate scaling, w.l.o.g.\ we may and will assume that $a = 0$ and $b = 1$ in the following.
 There are several possible choices for conservative boundary conditions (e.g.\ at the right end), such as
   \begin{enumerate}
    \item
     $\omega(t,1) = (EI \omega_{\z\z})(t,1)$ (\emph{simply supported} or \emph{pinned} right end),
    \item
     $\omega_{\z\z}(t,1) = (EI \omega_{\z\z})_\z(t,1) = 0$ (\emph{free} right end),
    \item
     $\omega_\z(t,1) = (EI \omega_{\z\z})_\z(t,0) = 0$ (\emph{shear hinge} right end),
    \item
     $\omega_t(t,1) = \omega_\z(t,1)$ (\emph{clamped} left end),
    \item
     $\omega_t(t,1) = (EI \omega_{\z\z})(t,1) = 0$,
    \item
     $\omega_{t\z}(t,1) = (EI \omega_{\z\z})_\z(t,1) = 0$.
   \end{enumerate}
  Here, the first and third case are just special cases of the fifth (there: $\omega(t,1) = c \in \K$) and sixth case (there: $\omega_\z(t,1) = c \in \K$), so the most important conservative boundary conditions in energy state space formulation read as
   \begin{enumerate}
    \item
     $(\H^m x^m)(1) = 0$,
    \item
     $(\H^m_2 x^m_2)(1) = (\H^m_2 x^m_2)'(1) = 0$,
    \item
     $(\H^m x^m)'(1) = 0$,
    \item 
     $(\H^m_1 x^m_1)(1) = (\H^m_1 x^m_1)'(1) = 0$.
   \end{enumerate}
 At the other end we want to impose dissipative boundary conditions to obtain uniform exponential energy decay for the solution of the Euler-Bernoulli beam model closed in this linear way, the most popular being (cf.\ \cite{ChenEtAl_1987})
 \begin{equation}
  \left( \begin{array}{c} (EI \omega_{\z\z})(0) \\ - (EI \omega_{\z\z})_\z(0) \end{array} \right)
   = - K_0 \left( \begin{array}{c} \omega_{t\z}(t,\z) \\ \omega_t(t,\z) \end{array} \right)
   \label{eqn:dissipative_end}
 \end{equation}
 for some matrix $K_0 \in \K^{2 \times 2}$ such that
  \[
   \text{either} \quad
    K_0
     = \left[ \begin{array}{cc} k_0^{11} & 0 \\ 0 & 0 \end{array} \right]
     \text{ for some } k_0^{11} > 0,
    \quad \text{or} \quad
    \Sym (K_0) > 0 \text{ is positive definite.}
  \]
 For the first of these options, conservative boundary conditions at the right end of type \emph{clamped end} or \emph{shear hinge right end} ensure well-posedness and uniform exponential energy, whereas in the second case any of the conservative boundary conditions listed above, i.e.\ also \emph{free right end} or \emph{pinned right end} boundary conditions are allowed, lead to well-posedness with uniform exponential decay of the energy functional.
 
 \begin{lemma}
  For the Euler-Bernoulli beam of Example \ref{exa:euler-bernoulli_beam} assume that $\rho, EI: [0,1] \rightarrow \R$ are uniformly positive and Lipschitz continuous.
  Then, for $\mathfrak{A}$ and the following choices of $\mathfrak{R}: \dom(\mathfrak{R}) = \dom(\mathfrak{A}) \rightarrow \K^4$, the pair $(\mathfrak{A}, \mathfrak{R})$ has property ASP:
  \[
   \mathfrak{R} x
    =
   \left( \begin{array}{c} (\H_1 x_1)(0) \\ (\H_1 x_1)'(0) \\ (\H_2 x_2)(0) \\ (\H_2 x_2)'(1) \end{array} \right) \, \text{or} \,
    \left( \begin{array}{c} (\H_1 x_1)(0) \\ (\H_1 x_1)'(0) \\ (\H_2 x_2)'(0) \\ (\H_2 x_2)(1) \end{array} \right) \, \text {or} \,
    \left( \begin{array}{c} (\H_1 x_1)(0) \\ (\H_2 x_2)(0) \\ (\H_2 x_2)'(0) \\ (\H_1 x_1)'(1) \end{array} \right) \, \text{or} \, 
    \left( \begin{array}{c} (\H_1 x_1)'(0) \\ (\H_2 x_2)(0) \\ (\H_2 x_2)'(0) \\ (\H_1 x_1)(1) \end{array} \right).
  \]
 Moreover, for the following choices of $\mathfrak{R}': \dom(\mathfrak{R}') = \dom(\mathfrak{A}) \rightarrow \K^5$, the pair $(\mathfrak{A}, \mathfrak{R}')$ has property $\mathrm{AIEP}_\tau$
  \begin{align*}
   \mathfrak{R}' x
     &= \left( \begin{array}{c} (\H x)(0) \\ (\H_1 x_1)'(0) \, \text{or} \, (\H_2 x_2)'(0) \\ (\H_1 x_1)(1) \, \text{or} \, (\H_2 x_2)'(1) \\ (\H_1 x_1)'(1) \, \text{or} \, (\H_2 x_2)(1) \end{array} \right),
     \quad \text{or} \\
   \mathfrak{R}' x
    &= \left( \begin{array}{c} (\H x)(0) \\ (\H x)(1) \\ (\H_{j_0} x_{j_0})'(\z_0) \end{array} \right)
    \quad \text{for some } j_0 \in \{1, 2\}, \z_0 \in \{0, 1\}
  \end{align*}
 In particular, for the following choices of $\mathfrak{R}'$, the pair $(\mathfrak{A}, \mathfrak{R}')$ has both properties ASP and $\mathrm{AIEP}_\tau$:
  \begin{align*}
   \mathfrak{R}' x
    &= \left( \begin{array}{c} (\H x)(0) \\ (\H_1 x_1)'(0) \\ (\H_1 x_1)(1) \, \text{or} \, (\H_2 x_2)'(1) \\ (\H_2 x_2)(1) \end{array} \right),
    \quad \text{or} \\
   \mathfrak{R}' x
    &= \left( \begin{array}{c} (\H x)(0) \\ (\H_2 x_2)'(0) \\ (\H_1 x_1)(1) \, \text{or} \, (\H_2 x_2)'(1) \\ (\H_1 x_1)'(1) \end{array} \right),
    \quad \text{or} \\
   \mathfrak{R}' x
    &= \left( \begin{array}{c} (\H x)(0) \\ (\H x)(1) \\ (\H_{j_0} x_{j_0})'(\z_0) \end{array} \right)
    \quad \text{for some } j_0 \in \{1, 2\}, \z_0 \in \{0, 1\}
  \end{align*}
 \end{lemma}
 
 \begin{proof}
 Partly, this is part of Proposition 2.9 in \cite{AugnerJacob_2014}. For the full proof of properties ASP and AIEP considered here, except for the latter case, and even in the more general setting of a chain of Euler-Bernoulli beams, see the upcoming article \cite{Augner_2018+a}.
 In these cases it remains to prove property $\mathrm{AIEP}_\tau$.
 This follows from property AIEP and Lemma \ref{lem:6.2} in the appendix.
 Let us prove the statement for the choice $\mathfrak{R}' x = ((\H x)(0), (\H x)(1), (\H_1 x_1)'(0))$, then it is clear how the remaining other choices for the fifth component can be handled.
 First of all, the pair $(\mathfrak{A}, \mathfrak{R}')$ has property ASP which can be seen by using e.g.\ \cite[Lemma 4.2.9]{Augner_2016} and in fact is a special case of \cite[Corollary 4.2.10]{Augner_2016}.
 Then, by \cite[(4.28) on p.\ 108]{Augner_2016} in the proof of \cite[Proposition 4.3.19]{Augner_2016}, for every sequence $(x_n, \beta_n)_{n \geq 1} \subseteq \dom(\mathfrak{A}) \times \R$ with $\sup_{n \in \N} \| x_n \|_X < \infty$ and $\abs{\beta_n} \rightarrow \infty$ such that $\mathfrak{A} x_n - \ii \beta_n x_n \rightarrow 0$ in $X$, and for every $q \in C^2([0,1];\R)$ one has the equality
  \begin{align*}
   &\Re \sp{x_n}{(2q'\H - q\H')x_n}_{L_2}
    \\
    &= \left[ - 2 \Re \sp{(\H_2 x_{n,2})'(\z)}{\tfrac{iq(\z)}{\beta_n} (\H_1 x_{n,1})'(\z)}_\K
     - \sp{x_n(\z)}{(q \H)(\z) x_n(\z)}_\K \right.
     \\
     &\quad
     \left.
     - \Re \sp{(\H_2 x_{n,2})'(\z)}{\tfrac{iq'(\z)}{\beta_n} (\H_1 x_{n,1})(\z)}_\K
     + \Re \sp{(\H_2 x_{n,2})'(\z)}{\tfrac{iq'(\z)}{\beta_n} (\H_1 x_{n,1})(\z)}_\K \right]_0^1
     + o(1)
  \end{align*}
 where $o(1)$ denotes terms  that vanish as $n \rightarrow \infty$.
 Also $\frac{\H x_n}{\beta_n} \rightarrow 0$ in $C^1([0,1];\K^2)$ has been shown there.
 Therefore, if we additionally assume that $\mathfrak{R}'x_n \rightarrow 0$ and take $q \in C^2([0,1];\R)$ such that
  \[
   2 q' \H - q \H \geq \varepsilon I,
    \quad
    \text{a.e.\ } \z \in (0,1),
  \]
 which is possible by the coercivity of $\H$ and the uniform boundedness of $\H'$, we obtain that
  \[
   \varepsilon \norm{x_n}_{L_2}^2
    \leq \sp{x_n}{(2q'\H - q \H)x_n}_{L_2}
    = o(1)
  \]
 and thus $x_n \rightarrow 0$ in $X$. This shows property AIEP.
 By Lemma \ref{lem:6.2} in the Appendix, it follows that $\tau(\H x_n) \rightarrow 0$ as well, so that $(\mathfrak{A}, \mathfrak{R}')$ has properties ASP and $\mathrm{AIEP}_\tau$.
 \end{proof}

 \begin{example}
  Consider the system of Figure \ref{fig:Diagram_Damper-String-Beam-System} consisting of a string which is damped at the left end, and is interconnected at the right end with an Euler-Bernoulli beam.
 \begin{figure}
  \centering
  \includegraphics[scale = 0.8]{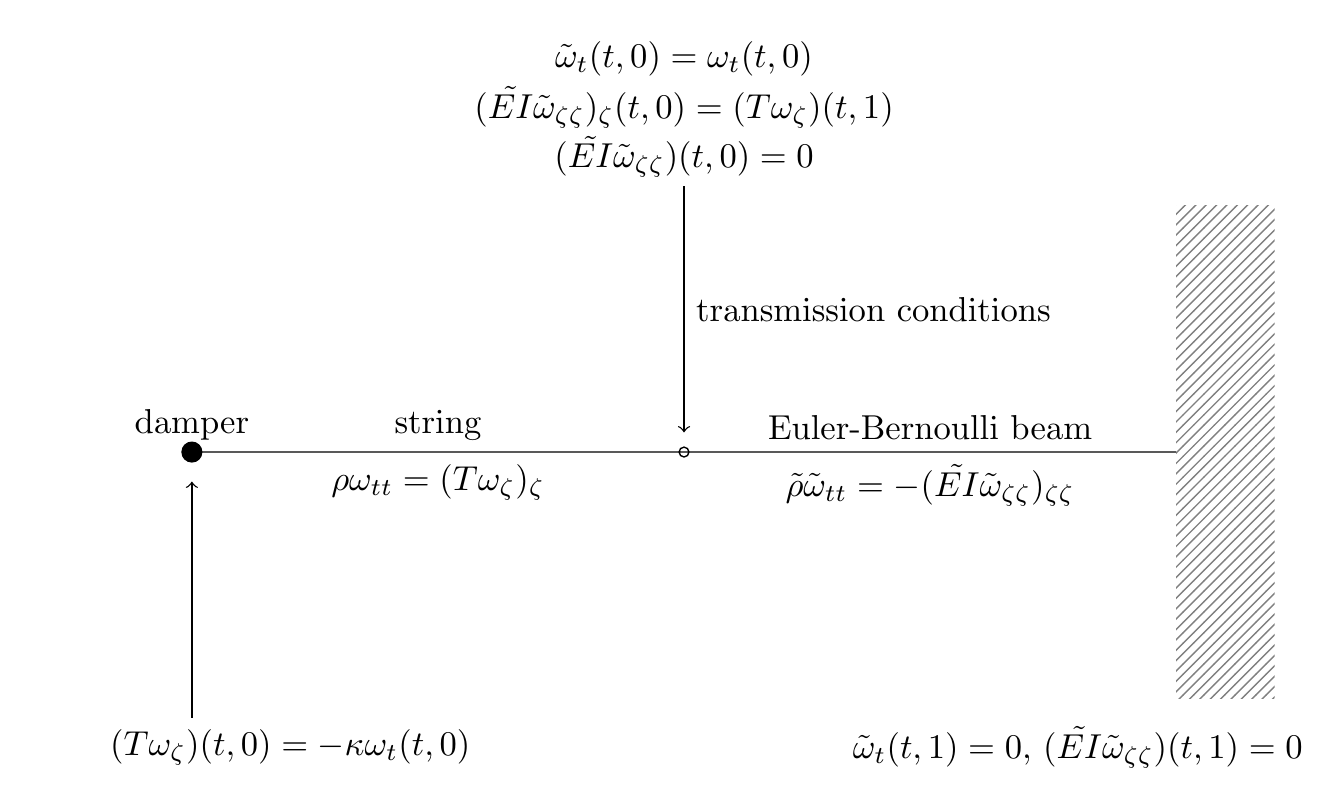}
	\caption{A damper-string-beam system.}
	\label{fig:Diagram_Damper-String-Beam-System}
 \end{figure}
  We denote by $\omega(t,\z)$ and $\tilde \omega(t,\z)$ the transversal position of the string and the Euler-Bernoulli beam at time time $t \geq 0$ and position $\z \in (0,1)$, respectively. (Here, w.l.o.g.\ we may and assume that both the string and the beam have unit length.)
  Moreover, we denote by $\rho(\z)$ and $\tilde \rho(\z)$ the mass density times transversal area at position $\z \in (0,1)$ for the string and the Euler-Bernoulli beam, respectively, by $T(\z)$ Young's modulus of the string and by $\tilde{EI}(\z)$ the elasticity times moment of inertia per area element of the Euler-Bernoulli beam. We assume that $\rho, \tilde \rho, T, \tilde{EI}$ are all Lipschitz-continuous and uniformly positive on $[0,1]$.
  Then the dynamics of the system are described by the evolution equations for the string and the beam,
   \begin{align*}
    \rho(\z) \omega_{tt}(t,\z)
     &= + (T \omega_\z)_{\z}(t,\z) \\
    \tilde \rho(\z) \tilde \omega_{tt}(t,\z)
     &= - (\tilde{EI} \tilde{\omega}_{\z\z})_{\z\z}(t, \z),
     &t \geq 0, \, \z \in (0,1)
   \end{align*}
  the damping by feedback boundary condition for the string at the left end
   \[
    (T \omega_\z)(t,0)
     = - \kappa \omega_t(t,0),
     \quad
     t \geq 0
   \]
  for some constant $\kappa > 0$, the transmission conditions
   \begin{align*}
    \tilde \omega_t(t,0)
     &= \omega_t(t,1)
     \\
    (\tilde{EI} \tilde \omega_{\z\z})_\z(t,0)
     &= - (T \omega_\z)(t,1)
     \\
    (\tilde{EI} \tilde \omega_{\z\z})(t,0)
     &= 0,
     &t \geq 0
   \end{align*}
  an the conservative \emph{pinned end} boundary conditions of the Euler-Bernoulli beam at the right end
   \begin{align*}
    \tilde \omega_t(t,1)
     &= 0
     \\
    (\tilde{EI} \tilde{\omega}(t,\z))(t,1)
     &= 0,
     &t \geq 0.
   \end{align*}
 \end{example} 
  The total energy of this system consists of the string part and the beam part of the energy
   \begin{align*}
    H_{\mathrm{tot}}(t)
     &= H_{\mathrm{WE}}(t) + H_{\mathrm{EB}}(t)
     \\
     &= \frac{1}{2} \left[ \int_0^1 \rho(\z) \abs{\omega_t(t,\z)}^2 + T(\z) \abs{\omega_{\z}}^2 \dd \z
       + \int_0^1 \tilde \rho(\z) \abs{\tilde \omega_t(t,\z)}^2 + \tilde{EI}(\z) \abs{\tilde \omega_{\z\z}(t,\z)}^2 \dd \z \right]
   \end{align*}
  and along solutions of the systems which are sufficiently regular, one readily computes
   \begin{align*}
    \frac{\dd}{\dd t} H_{\mathrm{tot}}(t)
     &= \Re \left[ \sp{(T \omega_\z)(t,\z)}{\omega_t(t,\z)}_\K
      + \sp{- (\tilde {EI} \tilde \omega_{\z\z})_\z(t,\z)}{\tilde \omega_t(t,\z)}_\K
      + \sp {(\tilde {EI} \tilde \omega_{\z\z})(t,\z)}{\tilde \omega_{t\z}}_\K \right]_0^1
      \\
     &= - \kappa \abs{\omega_t(t,0)}^2
     \leq 0,
     \quad
     t \geq 0.
   \end{align*}
  Hence, the system is dissipative, and the corresponding operator of port-Hamiltonian type $\hat A$ below generates a contractive $C_0$-semigroup.
  Since the subsystems are a string modelled by the one-dimensional wave equation and an Euler-Bernoulli beam, the port-Hamiltonian formulation reads as follows.
   \begin{align*}
    &X^1 = L_2(0,1;\K^2) \quad \text{with} \quad \H^1 = \operatorname{diag}\left( \frac{1}{\rho},T \right), \quad x^1(t,\z) = \left( \begin{array}{c} (\rho \omega_t)(t,\z) \\ \omega_\z(t,\z) \end{array} \right) \\
    &X^2 = L_2(0,1;\K^2) \quad \text{with} \quad \H^2 = \operatorname{diag}\left( \frac{1}{\tilde \rho},\tilde{EI} \right), \quad x^2(t,\z) = \left( \begin{array}{c} (\tilde \rho \tilde \omega_t)(t,\z) \\ \omega_{\z\z}(t,\z) \end{array} \right),
   \end{align*}
  there is no dynamic controller (i.e.\ $X_c^j = \{0\}$) and the differential operators are given by
   \begin{align*}
    \mathfrak{A}^1 x^1
     &= \left[ \begin{array}{cc} 0 & 1 \\ 1 & 0 \end{array} \right] \frac{\partial}{\partial \z} (\H^1 x^1)
     \\
    \mathfrak{A}^2 x^2
     &= \left[ \begin{array}{cc} 0 & 1 \\ -1 & 0 \end{array} \right] \frac{\partial^2}{\partial \z^2} (\H^2 x^2)
   \end{align*}
  and we get
   \begin{align*}
    A x
     &= \left( \begin{array}{c} \mathfrak{A}^1 x^1 \\ \mathfrak{A}^2 x^2 \end{array} \right)
     \\
    \dom(A)
     &= \{(x^1, x^2) \in X^1 \times X^2: \, (\H^1 x^1) \in H^1(0,1;\K^2), \, (\H^2 x^2) \in H^2(0,1;\K^2), \, (\H^1_2 x^1_2)(0) = - \kappa (\H^1_1 x^1_1)(0),\\
     &\qquad
     (\H^2_1 x^2_1)(0) = (\H^1_1 x^1_1)(1), \, (\H^2_2 x^2_2)'(0) = (\H^1_2 x^1_2)(0), \, (\H^2_2 x^2_2)(0), (\H^2 x^2)(1) = 0 \}.
   \end{align*}
  For this operator one has
   \[
    \Re \sp{Ax}{x}_X
     \leq - \kappa \abs{(\H_1^1 x_1^1)(0)}^2
     = - \frac{1}{\kappa} \abs{(\H^1_2 x^1_2)(0)}^2
     = - \frac{1}{2} \left[\kappa \abs{(\H_1^1 x_1^1)(0)}^2 + \frac{1}{\kappa} \abs{(\H^1_2 x^1_2)(0)}^2 \right],
     \quad
     x \in \dom(\hat A).
   \]
  Let us give its formulation as a serial interconnection of port-Hamiltonian systems:
   \begin{align*}
    \hat {\mathfrak{A}}^j
     &= \mathfrak{A}^j|_{\dom(\hat {\mathfrak{A}}^j)}
     \\
    \dom(\hat {\mathfrak{A}}^1)
     &= \{ x^1 \in \dom(\mathfrak{A}^1): \, (\H^1_2 x^1_2)(0) = - \kappa (\H^1_1 x^1_1)(0) \},
     \\
    \dom(\hat {\mathfrak{A}}^2)
     &= \{ x^1 \in \dom(\mathfrak{A}^1): \, (\H^2_2 x^2_2)(0) = 0, \, (\H^2 x^2)(0) = 0 \}
   \end{align*}
  and the conditions of the stability theorems are satisfied for
   \begin{align*}
    \mathfrak{R}^1 x^1 = (\sqrt{\kappa} (\H^1_1 x^1_1)(0), \sqrt{\kappa}^{-1} (\H^1_2 x^1_2)(0))
     \\
    \mathfrak{B}^2 x^2 = ((\H^2 x^2)(0), (\H^2 x^2)(1), (\H^2_2 x^2_2)'(0)).
   \end{align*}
  Now, the pairs $(\hat {\mathfrak{A}}^1, \mathfrak{R}^1)$ and $(\hat {\mathfrak{A}}^2, \mathfrak{B}^2)$ have properties ASP and $\mathrm{AIEP}_{\tau^j \circ \H^j}$ since both the pairs $(\mathfrak{A}^1, (\H^1 x^1)(0))$ and $(\mathfrak{A}^2, ((\H^2 x^2)(0), (\H^2 x^2)(1), (\H^2_1 x^2_1)(0))$ have properties ASP and $\mathrm{AIEP}_{\tau^j \circ \H^j}$.
  Therefore, by Theorems \ref{thm:asymptotic_stability} and \ref{thm:exp_stability} the operator $\hat A$ generates a uniformly exponentially stable contraction semigroup on $X = X^1 \times X^2$, i.e.\ there are constants $M \geq 1$ and $\eta < 0$ such that uniformly for all finite energy initial data the energy decays uniformly exponentially,
   \[
    H_{\mathrm{tot}}(t)
     \leq M \ee^{\eta t} H_{\mathrm{tot}}(0),
     \quad
     t \geq 0.
   \]
 \qed

 \begin{example}
  Consider the following interconnection of a string modelled by a wave equation, damped at the left end by a spring-mass damper and attached to an Euler-Bernoulli beam at the right, and where the latter is pinned at the right end, see Figure \ref{fig:Euler-Bernoulli-Beam-Spring-Mass-Damper-String-System}.
  \begin{figure}
	\centering
	\includegraphics[scale = 0.8]{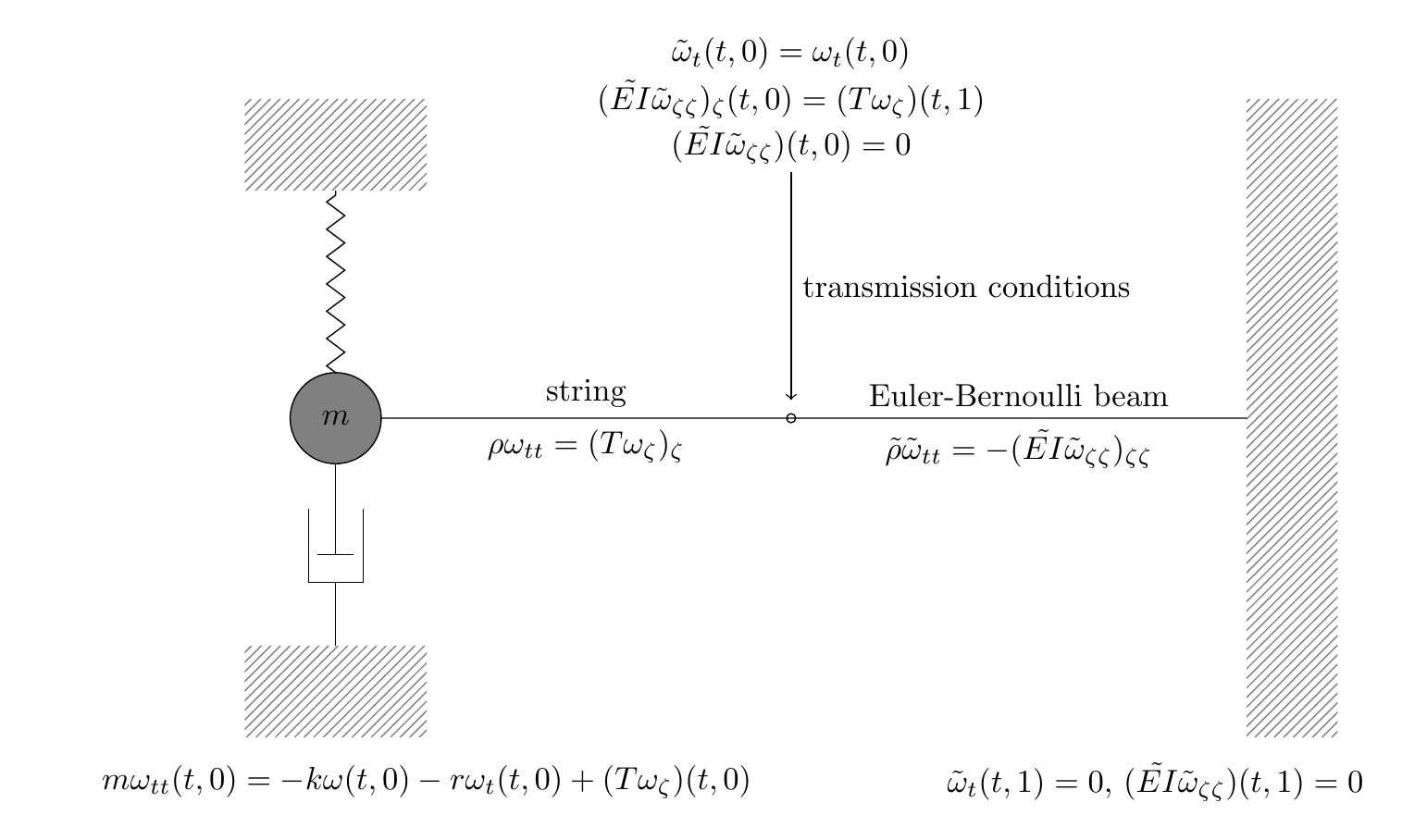}
	\caption{An Euler-Bernoulli-Beam-Spring-Mass-Damper-String-System.}
	\label{fig:Euler-Bernoulli-Beam-Spring-Mass-Damper-String-System}
 \end{figure}
  For the interconnection, the transmission conditions
   \[
    \omega(t,1) = \tilde \omega_t(t,0),
     \,
     (T \omega_\z)(t,1) + (\tilde{EI} \tilde \omega_{\z\z})_\z(t,0) = 0,
     \,
     \tilde \omega_\z(t,0) = 0,
     \quad
     t \geq 0.
   \]
  are assumed, i.e.\ in particular the transversal position of the string and the beam continuous is at the joint and no force is acting on the joint.
 The spring-mass damper is modelled by the ODE
  \[
   m \omega_{tt}(t,0)
    = - k \omega(t,0) - r \omega_t(t,0) + (T \omega_{\z})(t,0),
  \]
 i.e.\ the tip of mass $m > 0$ moves under the influence of forces from a spring with spring constant $k > 0$ and a damper with damping constant $r > 0$, as well as the stress $(T \omega_\z)(t,0)$ of the string at the left end.
 The pinned end boundary conditions of the Euler-Bernoulli beam are modelled by
  \[
   \tilde \omega(t,1) = 0,
    \,
    (\tilde{EI} \tilde \omega_{\z\z})(t,1) = 0.
  \]
\end{example}
 The total energy of this system is given by the potential and kinetic energies of the spring, the string and the beam
  \begin{align*}
   H_{\mathrm{tot}}(t)
    &= H_{\mathrm{WE}}(t) + H_{\mathrm{EB}}(t) + H_{m,k}(t)
    \\
    &= \frac{1}{2} \left[ \int_0^1 \rho(\z) \abs{\omega_t(t,\z)}^2 \dd \z + T(\z) \abs{\omega_\z(t,\z)} \dd \z \right]
     + \frac{1}{2} \left[ \int_0^1 \tilde \rho(\z) \abs{\tilde \omega_t(t,\z)}^ 2 + \tilde{EI}(\z) \abs{\tilde \omega_{\z\z}(t,\z)}^2 \dd \z \right]
     \\
     &\quad
     + \frac{1}{2} \left[ m \abs{\omega_t(t,0)}^2 + k \abs{\omega(t,0)}^2 \right].
  \end{align*}
 Then, the formal energy balance along sufficiently regular solutions shows that
  \begin{align*}
   \frac{\dd}{\dd t} H_{\mathrm{tot}}(t)
    &= - r \abs{\omega_t(t,0)}^2
    \leq 0,
    \quad
    t \geq 0.
  \end{align*}
 Therefore, the system is \emph{dissipative} and after reformulation as network of port-Hamiltonian type, it is clear that well-posedness in the sense of unique solutions with non-increasing energy holds for all sufficiently regular initial data.
 For this end, we take $X^1 = L_2(0,1;\K^2)$, $X^2 = L_2(0,1;\K^2)$ and $X_c^2 = \{0\}$ as in the previous example, but this time
  \[
   X_c^1 = \K^2,
    \quad
    x_c^1 = \left( \begin{array}{c} \omega(t,0) \\ \omega_t(t,0) \end{array} \right),
    \quad
    \norm{x_c}_{X_c^1}^2 = k \abs{x_{c,1}^1}^2 + m \abs{x_{c,2}^1}^2.
  \]
 Also the operators $\mathfrak{A}^1$ and $\mathfrak{A}^2$ are defined as before, but now we additionally have the control system given by the operators
  \[
   A_c^1 = \left[ \begin{array}{cc} 0 & 1 \\ - \tfrac{k}{m} & - \tfrac{r}{m} \end{array} \right],
    \,
    B_c^1 = \left[ \begin{array}{c} 0  \\ 1 \end{array} \right],
    \,
    C_c^1 = (B_c^1)^* = \left[ \begin{array}{cc} 0 & 1 \end{array} \right],
    \,
    D_c^1 = 0
  \]
 for $U_c^1 = \K$.
 The resulting operator $\hat A: \dom(\hat A) \subseteq \hat X = X \times X_c = X^1 \times X^2 \times X_c^1  \rightarrow \hat X$ is therefore given by
  \begin{align*}
   \hat A \left( \begin{array}{c} x^1 \\ x^2 \\ x_c^1 \end{array} \right)
    &= \left( \begin{array}{c} \mathfrak{A}^1 x^1 \\ \mathfrak{A}^2 x^2 \\ A_c^1 x_c^1 + B_c^1 (\H^1_2 x^1_2)(0) \end{array} \right)
    \\
   \dom(\hat A)
    &= \left\{ (x^1, x^2, x_c^1) \in \dom(\mathfrak{A}^1) \times \dom(\mathfrak{A}^2) \times X_c^1: \, (\H^1_1 x^1_1)(0) = - C_c^1 x_c^1, \, (\H^2_1 x^2_1)(0) = (\H^1_1 x^1_1)(1),
     \right. \\ &\qquad \left.
    (\H^2_2 x^2_2)'(0) = - (\H^1_2 x^1_2)(1), \, (\H^2_2 x^2_2)(0) = 0, (\H^2 x^2)(1) = 0 \right\}
  \end{align*}
 and it is dissipative with
  \[
   \Re \sp{\hat A \hat x}{\hat x}_{\hat X}
    = - r \abs{x_{c,2}^1}^2
    = - r \abs{(\H^1_1 x^1_1)(0)}^2,
    \quad
    \hat x = (x^1, x^2, x_c^1) \in \dom(\hat A).
  \]
 As a result, by Theorem \ref{thm:generation_network-dynamic} the operator $\hat A$ generates a strongly continuous contraction semigroup on $\hat X$.
 Let us investigate stability properties next.
 For this end, we write
  \begin{align*}
   \hat {\mathfrak{A}}^1 \hat x^1
    &= \left( \begin{array}{c} \mathfrak{A}^1 x^1 \\ A_c^1 x_c^1 + B_c^1 (\H^1_2 x^1_2)(0) \end{array} \right),
    \\
   \dom(\hat {\mathfrak{A}}^1)
    &= \{ \hat x^1 = (x^1, x_c^1) \in \dom(\mathfrak{A}^1) \times X_c^1: \, (\H^1_1 x^1_1)(0) = - C_c^1 x_c^1 \}
    \\
   \hat {\mathfrak{A}}^2 \hat x^2
     &= \mathfrak{A}^1 x^1
     \\
    \dom(\hat {\mathfrak{A}}^2)
     &= \{ \hat x^2 = x^2 \in \dom(\mathfrak{A}^2): \, (\H^2_2 x^2_2)(0) = 0, \, (\H^2 x^2)(1) = 0 \}.
  \end{align*}
 Then $\hat A \cong \operatorname{diag}\, (\hat {\mathfrak{A}}^1, \hat {\mathfrak{A}}^2)|_{\dom(\hat A)}$ for
  \[
   \dom(\hat A)
    = \left\{ \hat x = (\hat x^1, \hat x^2) \in \dom(\hat {\mathfrak{A}}^1) \times \dom(\hat {\mathfrak{A}}^2): \, \left( \begin{smallmatrix} (\H^1_1 x^1_1)(1) \\ (\H^1_2 x^1_2)(1) \end{smallmatrix} \right) = \left( \begin{smallmatrix} (\H^2_1 x^2_1)(0) \\ - (\H^2_2 x^2_2)'(0) \end{smallmatrix} \right) \right\}.
  \]
 To show uniform exponential energy decay, by Theorems \ref{thm:asymptotic_stability} and \ref{thm:exp_stability} it suffices to prove that the pairs $(\hat {\mathfrak{A}}^1, (\H^1_1 x^1_1)(0))$ and $(\hat {\mathfrak{A}}^2, ((\H^2_1 x^2_1)(0), (\H^2_2 x^2_2)'(0))$ have properties ASP and $\mathrm{AIEP}_\tau$.
 The latter, we have already seen in the previous example, as long as $\tilde \rho, \tilde{EI} \in \Lip(0,1;\R)$.
 It remains to prove these properties for the pair $(\hat {\mathfrak{A}}^1, (\H^1_1 x^1_1)(0))$. We assume that $\rho, T \in \Lip(0,1;\R)$ are Lipschitz continuous as well.
 For the matrix $A_c^1$ we can calculate the eigenvalues as $\lambda_{1,2} = - \frac{r \pm \sqrt{r^2 - 4km}}{2m} \in \C_0^-$, thus $A_c^1$ is a Hurwitz matrix and $(\ee^{t A_c^1})_{t \geq 0}$ uniformly exponentially stable on $X_c^1$. Since the pair $(\mathfrak{A}^1, (\H^1 x^1)(0))$ has properties ASP and $\mathrm{AIEP}_\tau$, this implies that also the pair $(\hat {\mathfrak{A}}^1, (\H^1_1 x^1_1(0))$ has properties ASP and AIEP. Uniform exponential stability of the semigroup $(\hat T(t))_{t \geq 0}$ on $\hat X$ thus follows by Theorems \ref{thm:asymptotic_stability} and \ref{thm:exp_stability}.
\qed

 \section{Conclusion and Open Problems}
 \label{sec:conclusion}
 
 In this paper, we have considered dissipative systems resulting from conservative or dissipative interconnection of several infinite-dimensional port-Hamiltonian systems $\mathfrak{S}^j = (\mathfrak{A}^j, \mathfrak{B}^j, \mathfrak{C}^j)$ of arbitrary, possibly distinct orders $N^j$ via boundary control and observation and static or dynamic feedback via a finite-dimensional linear control system $\Sigma_c = (A_c, B_c, C_c, D_c)$ such that the total, interconnected system on the product energy Hilbert space $\hat X$ becomes dissipative.
 The generation theorems from single infinite-dimensional port-Hamiltonian systems (or, port-Hamiltonian systems of the same differential order $N^j = N$ for all $j \in \mathcal{J}$) with static or dynamic boundary feedback have been shown to directly extend to systems of mixed-order port-Hamiltonian systems: The existence of a contractive $C_0$-semigroup $(\hat T(t))_{t \geq 0}$ acting as the (unique) solution operator for the abstract Cauchy problem
  \[
   \frac{\dd}{\dd t} \hat x(t) = \hat A \hat x(t) \, (t \geq 0),
    \quad
    \hat x(0) = \hat x_0 \in \hat X
  \]
 is equivalent to the operator $\hat A$ simply being dissipative (w.r.t.\ the energy inner product $\sp{\cdot}{\cdot}_{\hat X}$).
 Therefore, whenever beam and wave equations are interconnected with each other and finite dimensional control systems via boundary control and observation, it is enough to choose the boundary and interconnection conditions such that the energy does not increase along classical solutions.\newline
 For multi-component systems consisting of subsystems of finite dimensional or infinite-dimensional port-Hamiltonian type on an interval, we presented a scheme to ensure asymptotic and uniform exponential stability from the structure of the interconnection and dissipative elements. Especially, we applied the results to a chain of strings modelled by the wave equation and hinted at possible arrangements of beam-string-controller-dissipation structures leading to uniform stabilisation of the total interconnected system.\newline
 All results presented here are based on linear semigroup theory, especially the Arendt-Batty-Lyubich V\~u Theorem and the Gearhart-Greiner-Pr\"uss-Huang Theorem on stability properties for one-parameter semigroups of linear operators. Therefore, the techniques used are not accessible for nonlinear problems, e.g.\ nonlinear boundary feedback or nonlinear control systems which may be encountered in practice a lot.
 Whereas for the generation theorem the Komura-Kato Theorem is a nonlinear analogue to the Lumer-Phlipps Theorem for the generation of strongly continuous contraction semigroups by $m$-dissipative operators, handling stability properties for nonlinear systems is much more involved, see \cite{Augner_2018+} for some efforts in this direction.
 
 \section{Appendix: Some technical results on the Euler-Bernoulli Beam}
 \label{sec:appendix}

 Within this section we consider a port-Hamiltonian system operator $\mathfrak{A}$ of Euler-Bernoulli type, i.e.\ we assume that
  \[
   \mathfrak{A} x
    = \left( \left[ \begin{array}{cc} 0 & -1 \\ 1 & 0 \end{array} \right]  \frac{\partial^2}{\partial \z^2} + P_0(\z) \right) \left[ \begin{array}{cc} \H_1(\z) & 0 \\ 0 & \H_2(\z) \end{array} \right] \left( \begin{array}{c} x_1(\z) \\ x_2(\z) \end{array} \right)
  \]
 where we additionally assume that $x_1(\z), x_2(\z) \in \K$ are scalars and the Hamiltonian densities $\H_i \in \Lip(0,1)$ as well as the bounded perturbation $P_0 \in \Lip(0,1;\K^{2 \times 2})$ are Lipschitz continuous.
 We consider the situation that we have sequences $(x_n)_{n \geq 1} \subseteq \dom(\hat A) = \{ x \in L_2(0,1;\K^2): \H x \in H^2(0,1;\K^2) \}$ and $(\beta_n)_{n \geq 1} \subseteq \R$ such that the following hold (with convergence in $L_2(0,1;\K^d)$)
  \begin{equation}
   x_n \rightarrow 0,
   \quad
   |\beta_n| \rightarrow 0,
   \quad \text{and} \quad
   \mathfrak{A} x_n - \ii \beta_n x_n \rightarrow 0.
   \label{cond:sequence}
  \end{equation}
We first note that then also
 \[
  (\mathfrak{A} - P_0 \H) x_n - \ii \beta_n x_n \rightarrow 0,
 \]
Therefore, in the following we can ourselves often essentially restrict to the case $P_0 = 0$.
We investigate what can be said about the sequence of traces $\tau(\H x_n)$, if we additionally assume that parts of the trace, e.g.\ $(\H x_n)(0)$, are already known to converge to zero.

The first important observation is the following.

 \begin{lemma}
 \label{lem:6.1}
  Let $j \in \{1, 2\}$.
  Assume additionally that for both boundary points $\z \in \{0, 1\}$, either $(\H_j x_{n,j})(\z) \rightarrow 0$ or $(\H_j x_{n,j})'(\z) \rightarrow 0$ is known.
  Then
   \[
    \frac{1}{\sqrt{|\beta_n|}} \| (\H_j x_{n,j})' \|_{L_2(0,1)} \rightarrow 0,
     \quad
     n \rightarrow \infty.
   \]
 \end{lemma}
 
 \begin{proof}
 As $x_n \rightarrow 0$ and $(\mathfrak{A} - P_0 \H) x_n - \ii \beta_n \rightarrow 0$ in $L_2(0,1;\K^2)$, we also have that $\frac{1}{\beta_n} (\H_j x_{n,j})'' \rightarrow 0$ in $L_2(0,1)$, and also (already without any of the extra conditions) $\frac{1}{\beta_n} (\H_j x_{n,j})(\z), \frac{1}{\beta_n} (\H_j x_{n,j})'(\z) \rightarrow 0$ in $\K$ for $\z = 0, 1$.
 Therefore
  \begin{align*}
   &\frac{1}{|\beta_n|} \| (\H_j x_{n,j})' \|_{L_2(0,1)}^2
    \\
    &= \frac{1}{|\beta_n|} \int_0^1 \sp{ (\H_j x_{n,j})'(\z)}{(\H_j x_{n,j})'(\z)}_{\K} \dd \z
    \\
    &= - \frac{1}{|\beta_n|} \int_0^1 \sp{ (\H_j x_{n,j})(\z)}{(\H_j x_{n,j})''(\z)}_{\K} \dd \z
     + \frac{1}{|\beta_n|} \left[ \sp{ (\H_j x_{n,j})(\z)}{(\H_j x_{n,j})'(\z)}_{\K} \right]_0^1
    \\
    &\leq \| \H_j x_{n,j} \|_{L_2(0,1)} \frac{1}{|\beta_n|} \| (\H_j x_{n,j})' \|_{L_2(0,1)}
     + \sum_{\z=0}^1 \abs{(\H_j x_{n,j})(\z)} \frac{1}{|\beta_n|} \abs{(\H_j x_{n,j})'(\z)}
     \\
    &\longrightarrow 0,
    \quad
    n \rightarrow \infty
  \end{align*}
 where  in the last step we used the extra condition on the trace at boundary points $\z = 0, 1$.
 \end{proof}
 
 \begin{lemma}
 \label{lem:6.2}
  Let $(\beta_n)_{n \geq 1} \subseteq \R$ and $(x_n)_{n \geq 1} \subseteq \dom(\mathfrak{A})$ be as above, i.e.\ $\abs{\beta_n} \rightarrow \infty$, $\sup_{n \in \N} \norm{x_n}_X < \infty$, and $\mathfrak{A} x_n - \ii \beta_n x_n \rightarrow 0$ in $L_2(0,1;\K^2)$ and $P_0$ and $\H$ are Lipschitz-continuous.
   \begin{enumerate}
    \item
     \label{case-1}
     Assume that $(\H x_n)'(1) \rightarrow 0$, and $(\H^1 x_{n,1})(0) \rightarrow 0$ or $(\H_1 x_{n,1})'(0) \rightarrow 0$, and $(\H_2 x_{n,2})(0) \rightarrow 0$ or $(\H_2 x_{n,2})'(0) \rightarrow 0$ are known.
     Then
      \[
       (\H x_n)(1) \rightarrow 0.
      \] 
    \item
     \label{case-2}
     Assume that $(\H x_n)'(0) \rightarrow 0$, and $(\H_1 x_{n,1})(1) \rightarrow 0$ or $(\H_1 x_{n,1})'(1) \rightarrow 0$, and $(\H_2 x_{n,2})(1) \rightarrow 0$ or $(\H_2 x_{n,2})'(1) \rightarrow 0$ are known.
     Then
      \[
       (\H x_n)(0) \rightarrow 0.
      \] 
    \item
     \label{case-3}
     Assume that $(\H x_n)'(0) \rightarrow 0$ and $(\H x_n)'(1) \rightarrow 0$ are known.
     Then
      \[
       \tau(\H x_n) \rightarrow 0.
      \]
    \item
     \label{case-4}
     Assume that $(\H x_n)(0) \rightarrow 0$, and $(\H_1 x_{n,1})'(0) \rightarrow 0$ or $(\H_2 x_{n,2})'(0) \rightarrow 0$, as well as $|(\H_1 x_{n,1})(1)| + |(\H_1 x_{n,1})'(1)| \rightarrow 0$ or $|(\H_2 x_{n,2})(1)| + |(\H_2 x_{n,2})'(1)| \rightarrow 0$ are known.
     Then
      \[
       \tau(\H x_n) \rightarrow 0.
      \]
    \item
     \label{case-5}
     Assume that $(\H x_n)(1) \rightarrow 0$, and $(\H_1 x_{n,1})'(1) \rightarrow 0$ or $(\H_2 x_{n,2})'(1) \rightarrow 0$, as well as $|(\H_1 x_{n,1})(0)| + |(\H_1 x_{n,1})'(0)| \rightarrow 0$ or $|(\H_2 x_{n,2})(0)| + |(\H_2 x_{n,2})'(0)| \rightarrow 0$ are known.
     Then
      \[
       \tau(\H x_n) \rightarrow 0.
      \]
    \item
     \label{case-6}
     Assume that $(\H x_n)(1) \rightarrow 0$, and $(\H_1 x_{n,1})(0) \rightarrow 0$ or $(\H_1 x_{n,1})'(0) \rightarrow 0$, and $(\H_2 x_{n,2})(0) \rightarrow 0$ or $(\H_2 x_{n,2})'(0) \rightarrow 0$ are known.
     Then
      \[
       (\H x_n)'(0) \rightarrow 0.
      \] 
    \item
     \label{case-7}
     Assume that $(\H x_n)(0) \rightarrow 0$, and $(\H_1 x_{n,1})(1) \rightarrow 0$ or $(\H_1 x_{n,1})'(1) \rightarrow 0$, and $(\H_2 x_{n,2})(1) \rightarrow 0$ or $(\H_2 x_{n,2})'(1) \rightarrow 0$ are known.
     Then
      \[
       (\H x_n)'(1) \rightarrow 0.
      \]
     \item
      \label{case-8}
      Assume that $(\H x_n)(0) \rightarrow 0$ and $(\H x_n)(1) \rightarrow 0$ are known.
      Then
       \[
        \tau(\H x_n) \rightarrow 0.
       \]
   \end{enumerate}
 \end{lemma}
 
 \begin{proof}
 The first five cases \ref{case-1} to \ref{case-5} are based on the following multiplier argument.
 As in any case the sequence $\left( \frac{i q}{\beta_n} (\H x_n)' \right)_{n \geq 1} \subseteq L_2(0,1;\K^2)$ is bounded, for any fixed $q \in C^1([0,1];\R)$, we obtain from $(\mathfrak{A} - P_0 \H) x_n - \ii \beta_n x_n \rightarrow 0$ in $L_2(0,1;\K^2)$ that
   \begin{align*}
    &0
     \leftarrow \sp{(\mathfrak{A} - P_0 \H) x_n - \ii \beta_n x_n}{\frac{i q}{\beta_n} (\H x_n)'}_{L_2}
     \\
     &= - \sp{(\H_2 x_{n,2})''}{\frac{iq}{\beta_n} (\H_1 x_{n,1})'}_{L_2}
      + \sp{(\H_1 x_{n,1})''}{\frac{i q}{\beta_n} (\H_2 x_n)'}_{L_2}
      - \sp{x_n}{q (\H x_n)'}_{L_2}
      \\
     &= \sp{(\H_2 x_{n,2})'}{\frac{iq}{\beta_n} (\H_1 x_{n,1})''}_{L_2}
      + \sp{(\H_1 x_{n,1})''}{\frac{i q}{\beta_n} (\H_2 x_{n,2})'}_{L_2}
      + \sp{(\H_2 x_{n,2})'}{\frac{iq'}{\beta_n} (\H_1 x_{n,1})'}_{L_2}
      \\
      &\quad
      - \left[ \sp{(\H_2 x_{n,2})'(\z)}{\frac{i q(\z)}{\beta_n} (\H_1 x_{n,1})'(\z)}_{\K} \right]_0^1
      - \sp{(\H x_n)}{(q \H^{-1})' (\H x_n)}_{L_2}
      \\
      &\quad
      + \frac{1}{2} \left[ \sp{(\H x_n)(\z)}{(q \H^{-1})(\z) (\H x_n)(\z)}_{\K} \right]_0^1
      \\
     &= 2i \Im \sp{(\H_2 x_{n,2})'}{\frac{i q}{\beta_n} (\H_1 x_{n,1})''}_{L_2}
      + \sp{(\H_2 x_{n,2})'}{\frac{i q'}{\beta_n} (\H_1 x_{n,1})'}_{L_2}
      \\
      &\quad
      - \left[ \sp{(\H_2 x_{n,2})'(\z)}{\frac{iq}{\beta_n} (\H_1 x_{n,1})'(\z)}_{\K} \right]_0^1
      + \frac{1}{2} \left[ \sp{(\H x_n)(\z)}{(q (-q \H^{-1})(\z) (\H x_n)(\z)}_{\K} \right]_0^1
      + o(1)
   \end{align*} 
 where we denote by $o(1)$ any Terms that vanish as $n \rightarrow 0$.
 Taking the real part, this equality gives us that
  \begin{align}
   \Re \left( \sp{(\H_2 x_{n,2})'}{\frac{i q'}{\beta_n} (\H_1 x_{n,1})'}_{L_2}
    - \left[ \sp{(\H_2 x_{n,2})'(\z)}{\frac{i q(\z)}{\beta_n} (\H_1 x_{n,1})'(\z)}_{\K} \right]_0^1  \right)
    \nonumber \\
    + \frac{1}{2} \left[ \sp{(\H x_n)(\z)}{(q \H^{-1})(\z) (\H x_n)(\z)}_{\K} \right]_0^1
    &= o(1).
    \label{*}
  \end{align}
 The first of these terms can be estimated by $\frac{c}{|\beta_n|} \| (\H_2 x_{n,2})' \|_{L_2} \| (\H_1 x_{n,1})' \|_{L_2}$ which by Lemma \ref{lem:6.1} and under the constraints of the first or second case tends to zero as $n \rightarrow \infty$.
 The assertion for the first five cases then follow, namely
  \begin{enumerate}
   \item
    In the first case choose $q$ such that $q(0) = 0$ and $q(1) > 0$, then
     \[
      \frac{1}{2} \sp{(\H x_n)(1)}{(q \H^{-1})(1) (\H x_n)(1)}_{\K}
       \leq o(1),
     \]
    so $(\H x_n)(1) \rightarrow 0$ by positive definiteness of $\H(1)$.
   \item
    As before, this time choosing $q(1) = 0$ and $q(0) > 0$.
   \item
    Follows by combining the previous two cases \ref{case-1} and \ref{case-2} iteratively.
   \item
    In this case we do not have Lemma \ref{lem:6.1} at hand, but we may choose $q$ to be a constant $c \not= 0$.
    From equation \eqref{*} and the assumption on the boundary trace we then obtain that
     \[
      \sp{(\H x_n)(1)}{\H^{-1}(1) (\H x_n)(1)}_{\K}
       = o(1)
     \]
    so that $(\H x_n)(1) \rightarrow 0$ and $(\H x_n)(1) \rightarrow 0$.
    The assertion then follows from cases \ref{case-6} and \ref{case-7} below.
   \item
    For this case, repeat the argument of case \ref{case-4}.
  \end{enumerate}
We proceed by showing the assertion for the cases \ref{case-6} and \ref{case-7} by a similar multiplier argument, but this time using the multiplier $\frac{q}{\beta_n} \left[ \begin{smallmatrix} 0 & -1 \\ 1 & 0 \end{smallmatrix} \right] x_n'$. Since $\H$ is Lipschitz continuous, this is a bounded sequence in $L_2(0,1;\K^2)$ as well, so we find that
{\allowdisplaybreaks[3]
 \begin{align*}
  0
   &\leftarrow \sp{(\mathfrak{A} - P_0 \H) x_n - \ii \beta_n x_n}{\frac{q}{\beta_n} \left[ \begin{smallmatrix} 0 & -1 \\ 1 & 0 \end{smallmatrix} \right] x_n'}_{L_2}
   \\
   &= \sp{(\H_1 x_{n,1})''}{\frac{q}{\beta_n} x_{n,1}'}_{L_2}
    + \sp{(\H_2 x_{n,2})''}{\frac{q}{\beta_n} x_{n,2}'}_{L_2}
    \\
    &\quad
    - \sp{i x_{n,1}}{q x_{n,2}'}_{L_2}
    + \sp{i x_{n,2}}{q x_{n,1}'}_{L_2}
    \\
   &= \sp{(\H_1 x_{n,1})''}{\frac{q (\H_1)^{-1}}{\beta_n} (\H_1 x_{n,1})'}_{L_2}
    - \sp{(\H_1 x_{n,1})''}{\frac{q (\H_1)^{-1}}{\beta_n} (\H_1)' x_{n,1}}_{L_2}
    \\
    &\quad
    + \sp{(\H_2 x_{n,2})''}{\frac{q (\H_2)^{-1}}{\beta_n} (\H_2 x_{n,2})'}_{L_2}
    - \sp{(\H_2 x_{n,2})''}{\frac{q (\H_2)^{-1}}{\beta_n} (\H_2)' x_{n,2}}_{L_2}
    \\
    &\quad
    + \sp{i x_{n,1}'}{q x_{n,2}}_{L_2}
    - \sp{i x_{n,2}}{q x_{n,1}'}_{L_2}
    - \sp{i x_{n,2}}{q' x_{n,1}}_{L_2}
    + \left[ \sp{i x_{n,1}(\z)}{q(\z) x_{n,2}(\z)}_{\K} \right]_0^1
    \\
   &= - \frac{1}{2} \left( \sp{(\H_1 x_{n,1})'}{\frac{(q (\H_1)^{-1})'}{\beta_n} (\H_1 x_{n,1})'}_{L_2}
    + \sp{(\H_2 x_{n,2})'}{\frac{(q (\H_2)^{-1})'}{\beta_n} (\H_2 x_{n,2})'}_{L_2} \right)
    \\
    &\quad
    + \frac{1}{2} \left[ \sp{(\H_1 x_{n,1})'(\z)}{ \frac{(q (\H_1)^{-1})(\z)}{\beta_n} (\H_1 x_{n,1})'(\z)}_{\K}
    \right.    
    \\
    &\quad
    \left.
    + \sp{(\H_2 x_{n,2})'(\z)}{ \frac{(q (\H_2)^{-1})(\z)}{\beta_n} (\H_2 x_{n,2})'(\z)}_{\K} \right]_0^1
    \\
    &\quad
    + 2i \Im \sp{i x_{n_1}'}{q x_{n,2}}_{L_2}
    + \left[ \sp{i x_{n,1}(\z)}{q(\z) x_{n,2}(\z)}_{\K} \right]_0^1
    + o(1)
 \end{align*}}
Thus, taking the real part, we arrive at the equation
 \begin{align*}
  &\frac{1}{2} \left[ \sp{(\H_1 x_{n,1})'(\z)}{ \frac{(q (\H_1)^{-1})(\z)}{\beta_n} (\H_1 x_{n,1})'(\z)}_{\K}
   + \sp{(\H_2 x_{n,2})'(\z)}{ \frac{(q (\H_2)^{-1})(\z)}{\beta_n} (\H_2 x_{n,2})'(\z)}_{\K} \right]_0^1
   \\
   &+ \Re \left[ \sp{i x_{n,1}(\z)}{q(\z) x_{n,2}(\z)}_{\K} \right]_0^1
   \\
   &\qquad = \frac{1}{2} \left( \sp{(\H_1 x_{n,1})'}{\frac{(q (\H_1)^{-1})'}{\beta_n} (\H_1 x_{n,1})'}_{L_2}
    + \sp{(\H_2 x_{n,2})'}{\frac{(q (\H_2)^{-1})'}{\beta_n} (\H_2 x_{n,2})'}_{L_2} \right)
    + o(1).
 \end{align*}
Also for cases \ref{case-6} and \ref{case-7}, Lemma \ref{lem:6.1} gives us that $\frac{1}{\sqrt{|\beta_n|}} \norm{(\H x_n)'}_{L_2} = o(1)$, so that we obtain the result by choosing $q(0) = 0$ and $q(1) > 0$ or $q(0) > 0$ and $q(1) = 0$, respectively.
Finally, case \ref{case-8} follows by combining the results of cases \ref{case-6} and \ref{case-7}.
 \end{proof}

\begin{remark}
 Note that all the assertions of Lemmas \ref{lem:6.1} and \ref{lem:6.2} also hold for $\mathfrak{A}$ of the form
  \[
   \mathfrak{A} x
    = \left( \left[ \begin{array}{cc} 0 & -1 \\ 1 & 0 \end{array} \right]  \frac{\partial^2}{\partial \z^2} + P_1 \frac{\partial}{\partial \z} + P_0(\z) \right) \left[ \begin{array}{cc} \H_1(\z) & 0 \\ 0 & \H_2(\z) \end{array} \right] \left( \begin{array}{c} x_1(\z) \\ x_2(\z) \end{array} \right)
  \]
 where $P_1 \in \K^{2 \times 2}$ is any symmetric matrix.
 Namely, in cases \ref{case-1} to \ref{case-5} of Lemma \ref{lem:6.2} one may use that
  \[
   \Re \sp{P_1 (\H x_n)'}{\frac{i q}{\beta_n} (\H x_n)'}
    = 0,
  \]
 as $i q P_1$ is skew-symmetric.
 For the latter three cases \ref{case-6} to \ref{case-8} of Lemma \ref{lem:6.2} one may always use Lemma \ref{lem:6.1} to deduce that $\frac{1}{\sqrt{|\beta_n|}} \norm{(\H x_n)'}_{L_2} = o(1)$, but then also
  \[
   \abs{ \sp{P_1 (\H x_n)'}{\frac{q}{\beta_n} \left[ \begin{smallmatrix} 0 & -1 \\ 1 & 0 \end{smallmatrix} \right] x_n'} }
    \leq C \frac{1}{ \abs{\beta_n} } \norm{\H x_n}_{H^1}^2
    = o(1).
  \]
\end{remark}

\section*{Acknowledgement}

The author would like to thank Birgit Jacob for not only awaking his interest in port-Hamiltonian systems, but also sharing her knowledge, and having countless, fruitful discussions on the topic.

\end{document}